\documentclass[a4paper, 10pt]{amsart}

\usepackage{indentfirst}
\usepackage{mathrsfs}
\usepackage{amsmath}
\usepackage{amsthm}
\usepackage{amssymb}
\usepackage{tikz-cd}
\usepackage{mathtools}
\usepackage{array}

\usepackage[dvips,%
    includehead,%
    includefoot,%
    nomarginpar,%
    lmargin=1.3in,%
    rmargin=1.3in,%
    tmargin=1.4in,%
    bmargin=1.4in,%
  ]{geometry}

\usepackage{xcolor}
\definecolor{darkred}{RGB}{139,0,0}
\usepackage[hidelinks, colorlinks, citecolor=blue, linkcolor=darkred]{hyperref}

\newtheorem{Def}{Definition}[section]
\newtheorem{Prop}[Def]{Proposition}
\newtheorem{Lem}[Def]{Lemma}
\newtheorem{Thm}[Def]{Theorem}
\newtheorem{Cor}[Def]{Corollary}
\newtheorem{Rmk}[Def]{Remark}

\begin{document}
	
	\title[Mapping class groups of certain spin 5-manifolds]{Mapping class groups of simply connected closed spin 5-manifolds with no 2- and 3-torsion in homology}
	
	\author{Huize Jin}
	\address{Academy of Mathematics and Systems Sciences, Chinese Academy of Sciences, Beijing 100190, China}
	\email{jinhuize@amss.ac.cn}
	
	\date{}
	
	\begin{abstract}
		We show that the Torelli group of a simply connected closed 5-manifold $M$, which is spin and has no 2-torsion elements in homology, is isomorphic to the bordism group ${\Omega}^{\mathrm{Spin}}_6(K(H_2(M), 2))$. For $H_2(M)$ with no 2- and 3-torsion we determine this bordism group, and give explicit constructions for the generators of the Torelli group. Furthermore, for the $g$-fold connected sum ${\#}^g(S^2 \times S^3)$ we completely determine its mapping class group. We apply our results to compute the stabilization and abelianization of the mapping class group of ${\#}^g(S^2 \times S^3)$, determine the group of isotopy classes of diffeomorphisms of $M$ that are homotopic to the identity, and study the embeddings of $S^3$ in $S^2 \times S^3$.
	\end{abstract}

	\maketitle
	
	\tableofcontents

    \section{Introduction}
    
        Given a closed oriented smooth manifold $M$, let $\mathrm{Diff}^{+}(M)$ be the oriented diffeomorphism group of $M$, equipped with $C^{\infty}$-topology. The \emph{mapping class group} of $M$, \[\mathrm{MCG}(M) \coloneqq \pi_0 \mathrm{Diff}^{+}(M),\] is the group of isotopy classes of orientation preserving diffeomorphisms of $M$. The subgroup $\mathcal{I}(M)$ of $\mathrm{MCG}(M)$, consisting of diffeomorphisms that induce the identity homomorphism on $H_*(M)$, is called the \emph{Torelli group} of $M$. One may also consider the diffeomorphism group $\mathrm{Diff}(M)$ and the \emph{full mapping class group} \[\mathrm{MCG}^{\pm}(M) \coloneqq \pi_0\mathrm{Diff}(M)\] consisting of (isotopy classes of) all diffeomorphisms of $M$. Moreover, given a subspace $D \subset M$ let $\mathrm{Diff}(M, D)$ be the group of diffeomorphisms that fix $D$ pointwise. If $M$ is compact and has nonempty boundary, let $\mathrm{Diff}(M, \partial)$ be the group of diffeomorphisms that fix some collar neighborhood of the boundary pointwise. There are corresponding versions of relative mapping class groups \[\mathrm{MCG}(M, D) \coloneqq \pi_0\mathrm{Diff}(M, D), \; \mathrm{MCG}(M, \partial) \coloneqq \pi_0\mathrm{Diff}(M, \partial)\] and relative Torelli groups $\mathcal{I}(M, D)$, $\mathcal{I}(M, \partial)$.
        
        Suppose $M$ is of dimension $\geq 5$. There are several points of view to understand the role of $\mathrm{MCG}(M)$ in the topology of $M$. The determination of the mapping class group $\mathrm{MCG}(M)$ is the same as the isotopy classification of orientation preserving diffeomorphisms of $M$, which is a natural subsequent step to the characterization of the oriented diffeomorphism type of $M$. Moreover, $\mathrm{MCG}(M) = \pi_0 \mathrm{Diff}^{+}(M)$ is the group of path components of $\mathrm{Diff}^{+}(M)$ and thus contains homotopic information of $\mathrm{Diff}^{+}(M)$, the latter is a major object of interest in geometric topology, and is accessible by methods and techniques from various branches of mathematics (see for example \cite{Kup, WW01, BM20}). Despite of its significance, a complete description of $\mathrm{MCG}(M)$ can be hard to achieve. When $M$ is simply connected, the seminal work of Sullivan \cite{Sul77} implies that $\mathrm{MCG}(M)$ is finitely presentable, so one may have a better chance to compute $\mathrm{MCG}(M)$ under this assumption.
        
        In this paper we study the mapping class groups of simply connected closed 5-manifolds that are spin and having no 2- and 3-torsion elements in the second homology group. For these manifolds we could compute the Torelli groups explicitly and provide a complete set of concretely constructed generators for these groups. Our main tool is Kreck's modified surgery theory \cite{Kre99, Kre}. We summarize our main results in \ref{results} and sketch our methods in \ref{methods}.
        
        \subsection{The results}\label{results}
        
        The primary result of this paper is an explicit computation of Torelli groups of simply connected closed spin 5-manifolds with no 2- and 3-torsion elements in the second homology group. The simplest 5-manifold of such type is $S^5$. By $h$-cobordism theorem and Cerf's pseudoisotopy theorem \cite[Theorem 1]{Cer71}, for $n \geq 5$ $\mathrm{MCG}(S^n)$ is isomorphic to the Kervaire\textendash Milnor group $\Theta_{n+1}$ of smooth homotopy spheres, so $\mathrm{MCG}(S^5) = \Theta_6$ is trivial. Thus we focus on the case where the second homology of the 5-manifold is nontrivial, since otherwise this manifold is a smooth homotopy sphere which is diffeomorphic to $S^5$ by the fact that $\Theta_5$ is trivial \cite{KM63}.
        
        \begin{Thm}\label{thm:1}
        	Let $M$ be a simply connected closed spin 5-manifold.
        	\begin{itemize}
        		\item[(1)] There is a short exact sequence \[0 \rightarrow \mathcal{I}(M) \rightarrow \mathrm{MCG}(M) \xrightarrow{\phi} \mathrm{Aut}(H_2(M), L_M) \rightarrow 1,\] where $\mathrm{Aut}(H_2(M), L_M)$ denotes the group of automorphisms of $H_2(M)$ that preserve the linking form $L_M$, and $\phi$ maps an isotopy class of diffeomorphisms to its induced automorphism of $H_2(M)$.
        		
        		\item[(2)] Suppose $H_2(M)$ has no 2-torsion elements. There is an isomorphism of abelian groups \[\mathcal{I}(M) \cong {\Omega}^{\mathrm{Spin}}_6(K(H_2(M), 2)).\] The generators of $\mathcal{I}(M)$ are given in \ref{5.3}.
        		
        		\item[(3)] Suppose $H_2(M)$ has no 2- and 3-torsion elements. The action of $\mathrm{Aut}(H_2(M), L_M)$ on $\mathcal{I}(M)$ is the standard action on ${\Omega}^{\mathrm{Spin}}_6(K(H_2(M), 2))$.
        	\end{itemize}
        \end{Thm}
        
        By the \emph{standard action} we mean the action defined as follows: let $h \in \mathrm{Aut}(H_2(M), L_M)$, there is a self-map $\delta_h$ of $K(H_2(M), 2)$, unique up to homotopy, which induces $h$ on the second homology. Then there is a well-defined action \[h \cdot [N, \nu] \coloneqq [N, {\delta_h} \circ {\nu}]\] of $\mathrm{Aut}(H_2(M), L_M)$ on the bordism group ${\Omega}^{\mathrm{Spin}}_6(K(H_2(M), 2))$.
        
        We could solve the extension problem in part (1) of Theorem \ref{thm:1} for $M_g \coloneqq {\#}^g(S^2 \times S^3)$, the $g$-fold connected sum of $S^2 \times S^3$, by showing that it splits using geometric methods. Therefore we have a complete description of the group structure of $\mathrm{MCG}(M_g)$. Note that if $M = M_g$ then the linking form vanishes by definition, and $\mathrm{Aut}(H_2(M_g), L_{M_g})$ is isomorphic to the general linear group $\mathrm{GL}_g(\mathbb{Z})$.
        
        To formulate our result on $\mathrm{MCG}(M_g)$ we recall the definition of Dehn twists. Let $M$ be an $n$-dimensional manifold and $S : S^k \times D^{n-k} \hookrightarrow \mathring{M}$ be an embedding in the interior of $M$. Let $\beta \in \pi_{n-k}(\mathrm{SO}(k+1))$, then $\beta$ is represented by a smooth map $f : D^{n-k} \rightarrow \mathrm{SO}(k+1)$ mapping an open neighborhood of $\partial D^{n-k}$ to the identity matrix. The \emph{Dehn twist about $S$ and $\beta$} is the diffeomorphism \[\mathbf{t}_{S, \beta} : M \rightarrow M\] restricting to the identity map on the complement of $S(S^k \times D^{n-k})$, and mapping $S(S^k \times D^{n-k})$ to itself by \[S(x, y) \mapsto S(f(y) \cdot x, y).\] Clearly the isotopy class of $\mathbf{t}_{S, \beta}$ depends only on the isotopy class of $S$ and the homotopy class $\beta$. When $n = 5$, the isotopy class of an embedding $S : S^2 \times D^3 \hookrightarrow M$ depends only on the isotopy class of the embedding $S |_{S^2 \times \{0\}}$ of its core sphere (since $\pi_2(\mathrm{SO}(3)) = 0$) which is determined by its homotopy class by Haefliger's embedding theorem \cite{Hae61-results}. Thus, given embeddings $S_1, \cdots , S_m : S^2 \times D^3 \hookrightarrow M$, the embedding $S^2 \times D^3 \hookrightarrow M^5$ with core sphere being the connected sum of the core spheres of $S_1, \cdots , S_m$ is unique up to isotopy (such a connected sum exists since the codimension is high enough). We denote (the isotopy class of) this embedding by $S_1 \# \cdots \# S_m$.
        
        For $k = 1, \cdots , g$ let $S_k : S^2 \times D^3 \hookrightarrow M_g$ be the standard embedding in the $k$-th $S^2 \times S^3$ summand. Let $\alpha \in \pi_3(\mathrm{SO}(3)) \cong \mathbb{Z}$ be the generator such that the first Pontryagin number of the 3-dimensional vector bundle over $S^4$ with clutching function $\alpha$ is 4.
        
        \begin{Thm}\label{thm:2}
        	The short exact sequence \[0 \rightarrow \mathcal{I}(M_g) \rightarrow \mathrm{MCG}(M_g) \rightarrow \mathrm{GL}_g(\mathbb{Z}) \rightarrow 1\] splits. In other words, there is an isomorphism of groups \[\mathrm{MCG}(M_g) \cong \mathcal{I}(M_g) \rtimes \mathrm{GL}_g(\mathbb{Z})\] where the Torelli group \[\mathcal{I}(M_g) \cong {\Omega}^{\mathrm{Spin}}_6((\mathbb{CP}^{\infty})^g) \cong \mathbb{Z}^{2g + 2{\binom{g}{2}} + {\binom{g}{3}}}\] is freely generated by the following Dehn twists
        	\begin{align*}
        		& \mathbf{t}_{S_k, \alpha}, & 1 \leq k \leq g; \\
        		& \mathbf{t}_{S_k \# S_k, \alpha}, & 1 \leq k \leq g; \\
        		& \mathbf{t}_{S_k \# S_l, \alpha}, & 1 \leq k < l \leq g; \\
        		& \mathbf{t}_{S_k \# S_k \# S_l, \alpha}, & 1 \leq k < l \leq g; \\
        		& \mathbf{t}_{S_k \# S_l \# S_h, \alpha}, & 1 \leq k < l < h \leq g.
        	\end{align*}
        \end{Thm}
        
        For example, we can describe $\mathrm{MCG}(S^2 \times S^3)$ as follows. Here the orientation preserving diffeomorphism $R$ of $S^2 \times S^3$ is defined by $(x_1, x_2, x_3, y_1, y_2, y_3, y_4) \mapsto (-x_1, x_2, x_3, -y_1, y_2, y_3, y_4)$, where the coordinates are given by the standard embeddings  $S^2 \hookrightarrow \mathbb{R}^3$ and $S^3 \hookrightarrow \mathbb{R}^4$.
        
        \begin{Cor}\label{cor:product}
        	The mapping class group $\mathrm{MCG}(S^2 \times S^3)$ is isomorphic to $(\mathbb{Z} \times \mathbb{Z}) \rtimes \mathbb{Z}/2$, where the action of the nontrivial element of $\mathbb{Z}/2$ on $\mathbb{Z} \times \mathbb{Z}$ is given by sending an element to its inverse. In other words, $\mathrm{MCG}(S^2 \times S^3)$ has the following presentation: \[\left\langle a, b, v | ab = ba, v^2 = 1, vav = a^{-1}, vbv = b^{-1} \right\rangle,\] where $a$ and $b$ can be represented by the Dehn twists $\mathbf{t}_{S_1, \alpha}$ and $\mathbf{t}_{S_1 \# S_1, \alpha}$ respectively, and $v$ can be represented by $R$. It follows that $\mathrm{MCG}(S^2 \times S^3)$ is residually finite and virtually torsion-free, and has trivial center.
        \end{Cor}
        
        \begin{Rmk}
        	\rm At this point we compare our computation to relevant results in the literature.
        	\begin{itemize}
        		\item[(1)] Using techniques of pseudoisotopy Sato computed $\mathrm{MCG}(S^p \times S^q)$ for certain $p$ and $q$ \cite{Sat69}. In particular his results imply that there is a splitting exact sequence \[1 \rightarrow \mathcal{I}(S^2 \times S^3) \rightarrow \mathrm{MCG}(S^2 \times S^3) \rightarrow \mathbb{Z}/2 \rightarrow 0,\] and $\mathcal{I}(S^2 \times S^3)$ is isomorphic to $FC^3_3$, which is Haefliger's group of isotopy classes of framed 3-knots in $S^6$ and is isomorphic to $\mathbb{Z}^2$ by \cite[Theorem 5.17]{Hae66}. Our result extends Sato's computation of $\mathrm{MCG}(S^2 \times S^3)$ in the sense that we could determine the group structure and provide explicit constructions for generators of $\mathcal{I}(S^2 \times S^3)$.
        		
        		\item[(2)] Diffeomorphisms of $S^p \times S^q$ have also been considered by Turner \cite{Tur69} and Levine \cite{Lev69}. For $p < q < 2p-3$, Turner proved that $\mathcal{I}(S^p \times S^q)$ is isomorphic to a certain semi-direct product $(\pi_q(\mathrm{SO}(p+1)) \times \Theta_{p+q+1}) \rtimes \pi_p(\mathrm{SO})$ \cite[Theorem 3.10]{Tur69}. Independently, for $q < 2p-1$ and $q \geq 3$, Levine \cite[Proposition 5.2]{Lev69} proved that a diffeomorphism of $S^p \times S^q$ (not required to be orientation preserving) is pseudoisotopic to a composition of diffeomorphisms of the following three types:
        		\begin{itemize}
        			\item[(i)] $(x, y) \mapsto (\beta(y) \cdot x, y)$, where $\beta : S^q \rightarrow \mathrm{O}(p+1)$ is a smooth map;
        			
        			\item[(ii)] $(x, y) \mapsto (x, {\beta}'(x) \cdot y)$, where ${\beta}' : S^p \rightarrow \mathrm{O}(q+1)$ is a smooth map;
        			
        			\item[(iii)] a diffeomorphism supported in an embedded $(p+q)$-disc.
        		\end{itemize}
        		Note that $S^p \times S^{p+1}$ is included in Turner's result for $p \geq 5$ and in Levine's result for $p \geq 3$. Our result shows that $\mathcal{I}(S^2 \times S^3)$ is not a semi-direct product of such type, and one of its generators, $\mathbf{t}_{S_1 \# S_1, \alpha}$, cannot be written as such a composition. In this sense $S^2 \times S^3$ differs from $S^p \times S^{p+1}$ with $p \geq 3$.
        		
        		\item[(3)] For ${\#}^g(S^p \times S^p)$ where $p \geq 3$, there is an exact sequence \[1 \rightarrow \mathcal{I}({\#}^g(S^p \times S^p)) \rightarrow \mathrm{MCG}({\#}^g(S^p \times S^p)) \rightarrow G_g \rightarrow 1,\] where $G_g$ is a certain group encoding the action of diffeomorphisms on homology \cite[Theorem 2]{Kre79}. For $p$ odd this extension problem was studied by Krannich. His result \cite[Corollary C]{Kra20}, which extends earlier results of Crowley \cite{Cro11} and Galatius\textendash Randal-Williams \cite{GRW16-quotient}, states that this exact sequence does not split for $g \geq 2$, and splits for $g = 1$ and $p \neq 3, 7$. For $g = 1$, when $p = 3$ it does not split by \cite{Kry03}, and the case $p = 7$ is open. Our result indicates that the behaviors of $\mathrm{MCG}({\#}^g(S^p \times S^{p+1}))$ and $\mathrm{MCG}({\#}^g(S^p \times S^p))$ are quite different.
        	\end{itemize}
        \end{Rmk}
        
        \begin{Rmk}
        	\rm The surjectivity of $\phi: \mathrm{MCG}(M) \rightarrow \mathrm{Aut}(H_2(M), L_M)$ in Theorem \ref{thm:1} is due to Barden \cite[Theorem 2.2]{Bar65}. In general no explicit construction for diffeomorphisms which act nontrivially on $H_2(M)$ is known (for $M = M_g$, Trent Lucas has communicated his elegant construction to me). Nonetheless, we show that these diffeomorphisms cannot be Dehn twists (see \ref{5.5}).
        \end{Rmk}
        
        Using our computation of $\mathrm{MCG}(M_g)$ we can determine its abelianization.
        
        \begin{Cor}\label{cor:abelian}
        	Let $g \geq 1$.
        	\begin{itemize}
        		\item[(1)] The abelianization of $\mathrm{MCG}(M_g)$ is: \[H_1(\mathrm{MCG}(M_g)) \cong \begin{cases}
        			{(\mathbb{Z}/2)}^3, & g = 1,2; \\
        			{(\mathbb{Z}/2)}^2, & g = 3; \\
        			\mathbb{Z}/2, & g \geq 4.
        		\end{cases}\]
        		In particular, for $g \geq 4$ the homomorphism $\phi : \mathrm{MCG}(M_g) \rightarrow \mathrm{GL}_g(\mathbb{Z})$ induces an isomorphism $H_1(\mathrm{MCG}(M_g)) \rightarrow H_1(\mathrm{GL}_g(\mathbb{Z}))$.
        		
        		\item[(2)] We have $H^1(B\mathrm{Diff}^+(M_g)) = 0$ and $\mathrm{torsion}(H^2(B\mathrm{Diff}^+(M_g))) \cong H_1(\mathrm{MCG}(M_g))$ is nontrivial.
        	\end{itemize}
        \end{Cor}
        
        In particular oriented smooth fiber bundles with fiber $M_g$ have no 1-dimensional characteristic classes, but may have 2-dimensional characteristic classes of order 2.
        
        Our computation leads to several applications. The first application considers the stable phenomena of diffeomorphisms of 5-manifolds. Namely, there is a well-defined homomorphism $\mathrm{MCG}(M) \rightarrow \mathrm{MCG}(M \# (S^2 \times S^3))$ induced by taking connected sum with $id_{S^2 \times S^3}$ (see Section \ref{sec:7} for details). It fits into the following commutative diagram \[\begin{tikzcd}
        0 \arrow[r] & \mathcal{I}(M) \arrow[d] \arrow[r] & \mathrm{MCG}(M) \arrow[d] \arrow[r, "\phi"] & \mathrm{Aut}(H_2(M), L_M) \arrow[d] \arrow[r] & 1 \\
        0 \arrow[r] & \mathcal{I}(M \# (S^2 \times S^3)) \arrow[r] & \mathrm{MCG}(M \# (S^2 \times S^3)) \arrow[r, "\phi"] & \mathrm{Aut}(H_2(M) \oplus \mathbb{Z}, L_M) \arrow[r] & 1.
        \end{tikzcd}\] Two diffeomorphisms are \emph{stably isotopic} if they are isotopic after taking connected sums with $id_{S^2 \times S^3}$'s. Our computation implies the following two results.
        
        \begin{Cor}\label{cor:stable}
        	Let $M$ be a simply connected closed spin 5-manifold with no 2- and 3-torsion elements in homology. Then two orientation preserving diffeomorphisms of $M$ are not stably isotopic if they are not isotopic.
        \end{Cor}
        
        \begin{Cor}\label{cor:homologystable}
        	The group $H_1(\mathrm{MCG}(M_g))$ stabilizes when $g \geq 4$.
        \end{Cor}
        
        As the second application we compare diffeomorphisms and homotopy equivalences of $M$. We state a simplified version of our result for $M_g$ and refer to Section \ref{sec:8} for more general statements. Let $\mathcal{I}^h(M)$ be the \emph{homotopy Torelli group} of $M$, namely the group of homotopy classes of homotopy equivalences of $M$ that induce the identity on $H_*(M)$. There is an obvious forgetful homomorphism $\psi : \mathcal{I}(M) \rightarrow \mathcal{I}^h(M)$. Let $\mathcal{S}(M)$ be the (smooth) surgery structure set of $M$. Let $K(M_g)$ be the rank $g$ subgroup of $\mathcal{I}(M_g)$ freely generated by the isotopy classes $2(\mathbf{t}_{S_1 \# S_1, \alpha} - 8\mathbf{t}_{S_1, \alpha}), \cdots, 2(\mathbf{t}_{S_g \# S_g, \alpha} - 8\mathbf{t}_{S_g, \alpha})$.
        
        \begin{Thm}\label{thm:3}
        	There is an exact sequence (of based sets at $\mathcal{S}(M_g)$, of groups at other terms) \[0 \rightarrow K(M_g) \rightarrow \mathcal{I}(M_g) \xrightarrow{\psi} \mathcal{I}^h(M_g) \rightarrow \mathcal{S}(M_g) \rightarrow 0.\] The surgery structure set $\mathcal{S}(M_g)$ is a finite set containing $2^g$ elements.
        \end{Thm}
        
        In particular $K(M_g)$ consists of all diffeomorphisms of $M_g$ that are homotopic to $id_{M_g}$. More generally, in Section \ref{sec:8} we compute the group of isotopy classes of diffeomorphisms homotopic to the identity for $M$ with no 2- and 3-torsion in homology.
        
        As the third application we obtain some information on the problem of which element of $\pi_3(S^2 \times S^3)$ can be represented by an embedding of $S^3$. Let $S_1 : S^2 \hookrightarrow S^2 \times S^3$ and $T_1 : S^3 \hookrightarrow S^2 \times S^3$ be the standard embeddings, let $\eta : S^3 \rightarrow S^2$ be the Hopf map, then $\pi_3(S^2 \times S^3)$ is the rank 2 free abelian group with basis $\{[S_1 \circ \eta], [T_1]\}$.
        
        \begin{Thm}\label{thm:4}
        	Denote an element $a[S_1 \circ \eta] + b[T_1] \in \pi_3(S^2 \times S^3)$ by $(a, b)$.
        	\begin{itemize}
        		\item[(1)] The homotopy classes $(a, 1)$, $(a, -1)$ can be represented by an embedding of $S^3$ for each integer $a$. Such an embedding is homotopic to $f \circ T_1$ for some diffeomorphism $f$ of $S^2 \times S^3$.
        		
        		\item[(2)] For $b = 2$ or $b \geq 3$ odd, there exists $a \in \{0, 1, \cdots, b-1\}$ such that the homotopy classes $(a + kb, b)$, $(a + kb, -b)$ can be represented by an embedding of $S^3$ for each integer $k$. Such an embedding is not homotopic to $f \circ T_1$ for any diffeomorphism $f$ of $S^2 \times S^3$.
        	\end{itemize}
        \end{Thm}
        
        Given an embedding $T' : S^3 \hookrightarrow S^2 \times S^3$, the map $\mathrm{Diff}(S^2 \times S^3) \rightarrow \mathrm{Emb}(S^3, S^2 \times S^3)$ given by $f \mapsto f \circ T'$ is a locally trivial Serre fibration (see for example \cite{Pal60}). Here $\mathrm{Emb}(S^3, S^2 \times S^3)$ denotes the space of embeddings of $S^3$ in $S^2 \times S^3$, equipped with $C^{\infty}$-topology. A consequence of Theorem \ref{thm:4} is the following.
        
        \begin{Cor}\label{cor:fibration}
        	The Serre fibration $\mathrm{Diff}(S^2 \times S^3) \rightarrow \mathrm{Emb}(S^3, S^2 \times S^3)$ is not surjective.
        \end{Cor}

        \subsection{The methods}\label{methods}
        
        We sketch the method of using modified surgery to study the Torelli group of high-dimensional manifolds. For simplicity we restrict ourselves to the case of simply connected closed spin 5-manifolds. Let $M$ be such a manifold. The starting point is to consider the mapping torus \[M_f \coloneqq M \times I \, / \, (x, 0) \sim (f(x), 1)\] of an orientation preserving diffeomorphism $f$ of $M$. This is a closed orientable 6-manifold whose diffeomorphism type is determined by the isotopy class of $f$.
        
        We view $M_f$ as $(M \times I) \cup_{(id_M \cup f)} (M \times I)$, namely the result of gluing two copies of $M \times I$ via $id_M$ on two copies of $M \times \{0\}$ and via $f$ on two copies of $M \times \{1\}$. Let $p : B \rightarrow B\mathrm{SO}$ be the fibration appearing in the normal $2$-type of $M$. When $f$ lies in the relative Torelli group $\mathcal{I}(M, D)$, where $D \subset M$ is an embedded 5-disc, we could construct a normal $B$-structure $\overline{\nu} : M_f \rightarrow B$ such that the induced normal $B$-structures on the two copies of $M \times I$ are normal $2$-smoothings. This construction defines a map
        \begin{align*}
        	\rho : \mathcal{I}(M, D) & \rightarrow \Omega_6(B, p) \\
        	[f] & \mapsto [M_f, \overline{\nu}].
        \end{align*}
        The two versions of Torelli groups, $\mathcal{I}(M)$ and $\mathcal{I}(M, D)$, are connected by an exact sequence \[\mathbb{Z}/2 \rightarrow \mathcal{I}(M, D) \rightarrow \mathcal{I}(M) \rightarrow 1 \tag{1}\label{wal63seq}\] due to Wall \cite[page 265]{Wal63-classification2} (see Lemma \ref{lem:wallseq}).
        
        Kreck\textendash Su made the following key observation (compare \cite[page 514]{KS25}). Suppose $(M_f, \overline{\nu})$ bounds in $\Omega_6(B, p)$, then by modified surgery theory, there exists an $h$-cobordism $W$ from $M \times I$ to itself, such that the two copies of $M \times \{0\}$ are $h$-cobordant via $M \times I$ and the two copies of $M \times \{1\}$ are $h$-cobordant via the mapping cylinder of $f$. By $h$-cobordism theorem, there is a diffeomorphism $W \rightarrow M \times I \times I$ extending $id_{M \times I} : M \times I \rightarrow M \times I \times \{0\}$. The restriction of this diffeomorphism, \[\partial W \setminus (M \times I) \rightarrow \partial (M \times I \times I) \setminus (M \times I \times \{0\}),\] yields a pseudoisotopy between $id_M$ and $f$. Cerf's pseudoisotopy theorem then implies that $f$ is isotopic to $id_M$. All these diffeomrphisms and isotopies could be made to fix the disc $D$ pointwise. In our case we could show that $\rho$ is a homomorphism, thus it is injective.
        
        Kreck\textendash Su's observation suggests investigating the bordism group $\Omega_6(B, p)$. We compute this group and determine the relevant bordism invariants using Atiyah\textendash Hirzebruch spectral sequence. When $H_2(M)$ has no 2-torsion elements this bordism group has no 2-torsion either, so Wall's exact sequence (\ref{wal63seq}) implies that $\mathcal{I}(M, D) = \mathcal{I}(M)$. For $M = M_g$, we calculate the bordism invariants for the mapping tori of certain Dehn twists using explicit geometric methods, where the Hopf invariant plays a fundamental role. This leads to the determination of $\mathcal{I}(M_g)$ and its generators.
        
        Based on the determination of $\mathcal{I}(M_g)$ we could compute $\mathcal{I}(M)$. Using modified surgery we show that $\mathrm{MCG}({\natural}^g(S^2 \times D^3), \partial)$, the relative mapping class group of the $g$-fold boundary connected sum of $S^2 \times D^3$, is isomorphic to $\mathcal{I}(M_g)$. Barden's result \cite[Theorem 2.3]{Bar65} implies that $M$ is a twisted double of ${\natural}^g(S^2 \times D^3)$ for certain $g$, so there is a homomorphism $\mathrm{MCG}({\natural}^g(S^2 \times D^3), \partial) \rightarrow \mathcal{I}(M)$ induced by extending a diffeomorphism by identity. The computation of $\mathcal{I}(M)$ and the determination of its generators result from analyzing this homomorphism using bordism groups and spectral sequences.
        
        The major advantage of this modified surgery approach is that it enables us to define isotopy invariants for diffeomorphisms in the Torelli group, in terms of certain differential topological invariants (in our case, bordism invariants) of their mapping tori. (For a similar but more complicated result for certain 6-manifolds, making use of the Kreck\textendash Stolz invariants for certain 7-manifolds, see \cite{KS25}.) These invariants, or more generally, considering the bordism properties of mapping tori, converts the isotopy classification of diffeomorphisms into bordism theoretic constructions and differential topological calculations.
        
        \begin{Rmk}
        	\rm The modified surgery approach described in this subsection applies to non-spin 5-manifolds after mild modification (compare \cite[Section 10]{KS25}). Moreover one may use this approach to study the mapping class groups of closed $(4q+1)$-dimensional manifolds (see the statement of \cite[Theorem 5]{Kre99}).
        \end{Rmk}

        \subsection{Structure of the paper}
        
        In Section \ref{sec:2} we define the map $\rho$. Section \ref{sec:3} is devoted to the determination of relevant bordism groups and their bordism invariants. In Section \ref{sec:4} we compute $\mathcal{I}(M_g)$. In Section \ref{sec:5} we prove Theorem \ref{thm:1}, completing our computation of Torelli groups. In Section \ref{sec:6} we show that $\mathrm{MCG}(M_g)$ is a semi-direct product by a bordism theoretic construction, completing the proof of Theorem \ref{thm:2}, and compute the abelianization of $\mathrm{MCG}(M_g)$. The study of stabilization of $\mathrm{MCG}(M)$ is contained in Section \ref{sec:7}. The comparison to homotopy equivalences of $M$ is contained in Section \ref{sec:8}. Finally we prove Theorem \ref{thm:4} and Corollary \ref{cor:fibration} in Section \ref{sec:9}.
        
        \subsection*{Acknowledgments}
        
        I would like to express my deep gratitude to my advisor Yang Su for his consistent guidance, for many enlightening and encouraging discussions, and for making it possible for me to do mathematics at the very beginning. This work is partially supported by NSFC 12471069 and the Beijing Natural Science Foundation (grant 1262024).

    \section{The mapping torus construction}\label{sec:2}
    
      Let $M$ be a simply connected closed oriented 5-manifold, $f$ be a diffeomorphism of $M$ acting trivially on $H_*(M)$. Let $M_f$ be the mapping torus of $f$. We denote the inclusion $M = M \times \{\frac{1}{2}\} \hookrightarrow M_f$ by $i$. The Wang sequence of the fibration $M \xrightarrow{i} M_f \rightarrow S^1$ implies that $i^* : H^2(M_f; \mathbb{Z}/2) \rightarrow H^2(M; \mathbb{Z}/2)$ and $i_* : H_2(M) \rightarrow H_2(M_f)$ are isomorphisms. In particular $M_f$ is spin if and only if $M$ is spin, since $i^*w_2(M_f) = w_2(M)$.
      
      From now on we assume $M$ is spin. We will need the notion of \emph{normal $k$-type}, for the definition we refer to \cite{Kre99}. The normal $2$-type of $M$ is \[K(H_2(M), 2) \times B\mathrm{Spin} \rightarrow B\mathrm{Spin} \rightarrow B\mathrm{SO}\] where the first map is the projection to the second factor, and the second map is induced by the covering map. We denote this fibration by $p : B \rightarrow B\mathrm{SO}$. The bordism group $\Omega_6(B, p)$ is $\Omega^{\mathrm{Spin}}_6(K(H_2(M), 2))$.
      
      Now choose an embedded 5-disc $D \subset M$ and suppose $f$ represents an element of $\mathcal{I}(M, D)$. There is an embedded submanifold $S^1 \times D \subset M_f$. The ``preferred" spin structure on $M_f$ is taken to be the spin structure which restricts to the trivial spin structure on $S^1 \times D$. Let $\nu_{\mathrm{Spin}} : M_f \rightarrow B\mathrm{Spin}$ be the classifying map of the preferred spin structure lifting the normal Gauss map of $M_f$. Let $\nu_f : M_f \rightarrow K(H_2(M), 2)$ be a map inducing $i_*^{-1} : H_2(M_f) \rightarrow H_2(M)$, which is unique up to homotopy. The map \[\overline{\nu} \coloneqq (\nu_f, \nu_{\mathrm{Spin}}) : M_f \rightarrow K(H_2(M), 2) \times B\mathrm{Spin}\] determines a well-defined normal $B$-structure of $M_f$, which restricts to normal $2$-smoothings on the two copies of $M \times I$ under the identification $M_f = (M \times I) \cup_{(id_M \cup f)} (M \times I)$. This enables us to define a map (which we call the \emph{mapping torus construction})
      \begin{align*}
      	\rho : \mathcal{I}(M, D) & \rightarrow \Omega^{\mathrm{Spin}}_6(K(H_2(M), 2)) \\
      	[f] & \mapsto [M_f, \nu_f].
      \end{align*}
      
      \begin{Lem}\label{lem:mtc}
      	The mapping torus construction $\rho$ is an injective homomorphism.
      \end{Lem}
      
      \begin{proof}
      	First we show that $\rho$ is a homomorphism. Let $[f], [f'] \in \mathcal{I}(M, D)$. Let $W$ be the compact smooth manifold obtained from $M \times I \times I$ by identifying $M \times \{0\} \times [0, 1/3]$ with $M \times \{0\} \times [2/3, 1]$ via \[M \times \{0\} \times [2/3, 1] \rightarrow M \times \{0\} \times [0, 1/3], \; (x, t) \mapsto (f(x), 1-t)\] and identifying $M \times \{1\} \times [0, 1/3]$ with $M \times \{1\} \times [2/3, 1]$ via \[M \times \{1\} \times [0, 1/3] \rightarrow M \times \{1\} \times [2/3, 1], \; (x, t) \mapsto (f'(x), 1-t).\] This is a smooth fiber bundle over $D^2 \setminus (\mathring{D^2} \sqcup \mathring{D^2})$ with fiber $M$. The disjoint union $M_f \sqcup M_{f'}$ is oriented bordant to $M_{f \circ f'}$ via $W$ (see \cite[page 656]{Kre79}).
      	
      	The Leray\textendash Serre spectral sequence of the fibration $M \hookrightarrow W \rightarrow D^2 \setminus (\mathring{D^2} \sqcup \mathring{D^2})$ implies that the inclusion map $M \hookrightarrow W$ induces an isomorphism $H^2(W; \mathbb{Z}/2) \rightarrow H^2(M; \mathbb{Z}/2)$, and the projection $W \rightarrow D^2 \setminus (\mathring{D^2} \sqcup \mathring{D^2})$ induces an isomorphism $H^1(D^2 \setminus (\mathring{D^2} \sqcup \mathring{D^2}); \mathbb{Z}/2) \rightarrow H^1(W; \mathbb{Z}/2)$. So $W$ is spin, and the trivial spin structure of $(D^2 \setminus (\mathring{D^2} \sqcup \mathring{D^2})) \times D \subset W$ extends to a spin structure of $W$, which restricts to the preferred spin structures on the boundary of $W$. Thus $W$ is a spin bordism between $M_f \sqcup M_{f'}$ and $M_{f \circ f'}$. Moreover, $M \hookrightarrow W$ factors through $M \xhookrightarrow{i} M_f \hookrightarrow W$, so the inclusion map $M_f \hookrightarrow W$ induces an isomorphism $H_2(M_f) \rightarrow H_2(W)$. Similarly we see that the inclusion maps of $M_{f'}$ and $M_{f \circ f'}$ also induce isomorphisms on the second homology groups. Composing the isomorphism $H_2(W) \rightarrow H_2(M_f)$ to ${\nu_f}_* : H_2(M_f) \rightarrow H_2(M)$ we obtain an isomorphism $H_2(W) \rightarrow H_2(M)$ which determines a map $\widetilde{\nu} : W \rightarrow K(H_2(M), 2)$ extending $\nu_f$. The map $\widetilde{\nu}$ is well-defined up to homotopy. Note that the inclusion maps $M \hookrightarrow M_f \hookrightarrow W$, $M \hookrightarrow M_{f'} \hookrightarrow W$ and $M \hookrightarrow M_{f \circ f'} \hookrightarrow W$ are homotopic to each other, so $\widetilde{\nu}$ can be chosen so that it also extends $\nu_{f'}$ and $\nu_{f \circ f'}$. Now $(W, \widetilde{\nu})$ is a normal $B$-bordism between $(M_f, \overline{\nu}) \sqcup (M_{f'}, \overline{\nu})$ and $(M_{f \circ f'}, \overline{\nu})$, thus $\rho([f]) + \rho([f']) = \rho([f \circ f'])$.
      	
      	Then we show that $\rho$ is injective. Suppose $[M_f, \overline{\nu}] = 0 \in \Omega_6(B, p)$. Choose a coboundary $(W, \widetilde{\nu})$ of $(M_f, \overline{\nu})$. View it as a normal $B$-bordism, relative to boundary, between $M \times I$ and $M \times I$. By \cite[Theorems 4 and 5]{Kre99}, the modified surgery obstruction of $(W, \widetilde{\nu})$ is elementary, thus there is another coboundary of $(M_f, \overline{\nu})$, which is an $h$-cobordism relative to boundary. We still denote this coboundary by $W$. By $h$-cobordism theorem we have a diffeomorphism $W \cong M \times I \times I$. The embedding $S^1 \times D \hookrightarrow M_f$ extends to an embedding $D^2 \times D \hookrightarrow W$, which is an $h$-cobordism between $D \times I$ and $D \times I$. Since $M$ (and $M \setminus D$) is simply connected, by relative $h$-cobordism theorem there is a dffeomorphism $M \times I \rightarrow M \times I$ extending $id : D \times I \rightarrow D \times I$, restricting to $id_M$ on $M \times \{0\}$ and restricting to $f$ on $M \times \{1\}$. Thus $f |_{M \setminus \mathring{D}}$ is pseudoisotopic, relative to boundary, to $id_{M \setminus \mathring{D}}$. Cerf's pseudoisotopy theorem implies that $f |_{M \setminus \mathring{D}}$ is isotopic, relative to boundary, to $id_{M \setminus \mathring{D}}$. Gluing $D$ back we obtain an isotopy between $f$ and $id_M$, fixing $D$ pointwise. So $[f] = [id_M] \in \mathcal{I}(M, D)$. Therefore $\rho$ is injective.
      \end{proof}
      
      The forgetful homomorphism $\mathrm{MCG}(M, D) \rightarrow \mathrm{MCG}(M)$, which restricts to $\mathcal{I}(M, D) \rightarrow \mathcal{I}(M)$, is surjective by the disc theorem and isotopy extension theorem. We have the following result of Wall \cite[page 265]{Wal63-classification2} concerning the kernel of this homomorphism.
      
      \begin{Lem}[Wall]\label{lem:wallseq}
      	Let $M$ be a simply connected oriented $n$-manifold with $n \geq 5$, $D \subset M$ be an embedded $n$-disc. Then there is a commutative diagram of exact sequences \[\begin{tikzcd}
      		\pi_1(\mathrm{SO}(n)) \arrow[r] & \mathrm{MCG}(M, D) \arrow[r] & \mathrm{MCG}(M) \arrow[r] & 1 \\
      		\pi_1(\mathrm{SO}(n)) \arrow[r] \arrow[u, "="] & \mathcal{I}(M, D) \arrow[r] \arrow[u] & \mathcal{I}(M) \arrow[r] \arrow[u] & 1
      	\end{tikzcd}\]
      	where the maps $\mathrm{MCG}(M, D) \rightarrow \mathrm{MCG}(M)$ and $\mathcal{I}(M, D) \rightarrow \mathcal{I}(M)$ are given by the forgetful homomorphism.
      \end{Lem}
      
      \begin{proof}
      	The exact sequence in the first row is derived from the homotopy exact sequences of the following homotopy fibrations \begin{align*}
      		\mathrm{Fr}^{+}(M) \rightarrow B\mathrm{Diff}(M, D) & \rightarrow B\mathrm{Diff}^{+}(M), \\
      		\mathrm{SO}(n) \rightarrow \mathrm{Fr}^{+}(M) & \rightarrow M,
      	\end{align*} where $\mathrm{Fr}^{+}(M)$ is the oriented frame bundle of the tangent bundle of $M$. (See \cite[proof of Lemma 1.5]{GRW16-quotient} and \cite[proof of Lemma 1.1]{Kra20}.) The nontrivial element of $\pi_1(\mathrm{SO}(n) \cong \mathbb{Z}/2$, represented by a smooth map $\beta : I \rightarrow \mathrm{SO}(n)$, is mapped to the relative isotopy class of the following diffeomorphism: choose an identification of a collar neighborhood of $M \setminus \mathring{D}$ with $S^{n-1} \times I$, define
      	\begin{align*}
      	S^{n-1} \times I & \rightarrow S^{n-1} \times I \\
      	(x, t) & \mapsto (\beta(t) \cdot x, t)
      	\end{align*}
      	and extend it to a diffeomorphism of $M$ by identity \cite[page 657]{Kre79}. Clearly this relative isotopy class lies in $\mathcal{I}(M, D)$, so we obtain the exact sequence in the second row.
      \end{proof}

    \section{The spin bordism groups of certain Eilenberg\textendash MacLane spaces}\label{sec:3}
    
      In light of Lemma \ref{lem:mtc} we study the bordism group $\Omega^{\mathrm{Spin}}_6(K(H_2(M), 2))$ where $M$ is a simply connected closed spin 5-manifold. For $H_2(M)$ having no 2- and 3-torsion we could determine the isomorphism type and the bordism invariants of this bordism group. We also study a certain homomorphism between these groups induced by the quotient homomorphism.
    
    \subsection{Low-dimensional homology of Eilenberg\textendash MacLane spaces}
    
      First we recall the low-dimensional integral homology groups of $K(\mathbb{Z}/{p^r}, 2)$, where $p$ is a prime number and $r$ is a positive integer. The cases of dimension $\leq 6$ were determined by Eilenberg and MacLane in \cite[Chapter IV]{EM54}. The 7-dimensional homology group appears to be much more difficult to investigate using their methods (see \cite[page 116]{EM54}). We compute this group using the Leray\textendash Serre spectral sequence of the homotopy fibration $K(\mathbb{Z}/{p^r}, 1) \rightarrow * \rightarrow K(\mathbb{Z}/{p^r}, 2)$ and the mod $p$ cohomology of $K(\mathbb{Z}/{p^r}, n)$, the latter was determined by Cartan \cite{Car54} for $p \geq 3$ and by Serre \cite{Ser53-mod2} for $p = 2$. Precisely, for $p \neq 2$ or $3$ the spectral sequence and the integral homology of dimension $\leq 6$ implies that the 7-dimensional homology is trivial. For $p = 2$ or $3$, the spectral sequence implies that the 7-dimensional homology is a finite abelian group which is either trivial or $\mathbb{Z}/p$, and the mod $p$ cohomology tells us that it is $\mathbb{Z}/{p^k}$ for some $k > 0$, hence it is $\mathbb{Z}/p$. In summary we have the following.
      \begin{table}[h]
      	\centering
      	\begin{tabular}{c|cccccccc}
      		$ $ & $H_0$ & $H_1$ & $H_2$ & $H_3$ & $H_4$ & $H_5$ & $H_6$ & $H_7$ \\
      		\hline
      		$p = 2$ & $\mathbb{Z}$ & $0$ & $\mathbb{Z}/{2^r}$ & $0$ & $\mathbb{Z}/{2^{r+1}}$ & $\mathbb{Z}/2$ & $\mathbb{Z}/{2^r}$ & $\mathbb{Z}/2$ \\
      		$p = 3$ & $\mathbb{Z}$ & $0$ & $\mathbb{Z}/{3^r}$ & $0$ & $\mathbb{Z}/{3^r}$ & $0$ & $\mathbb{Z}/{3^{r+1}}$ & $\mathbb{Z}/3$ \\
      		otherwise & $\mathbb{Z}$ & $0$ & $\mathbb{Z}/{p^r}$ & $0$ & $\mathbb{Z}/{p^r}$ & $0$ & $\mathbb{Z}/{p^r}$ & $0$
      	\end{tabular}
      \end{table}
    
      The low-dimensional homology of $K(H_2(M), 2)$ can be calculated using K\"unneth theorem. For example the results for $K({(\mathbb{Z}/{p^r})}^2, 2)$ are as follows.
      \begin{table}[h]
    	\centering
    	\begin{tabular}{c|cccccc}
    		$ $ & $H_0$ & $H_1$ & $H_2$ & $H_3$ & $H_4$ & $H_5$ \\
    		\hline
    		$p = 2$ & $\mathbb{Z}$ & $0$ & $(\mathbb{Z}/{2^r})^{2}$ & $0$ & $\mathbb{Z}/{2^r} \oplus (\mathbb{Z}/{2^{r+1}})^{2}$ & $(\mathbb{Z}/2)^{2} \oplus \mathbb{Z}/{2^r}$ \\
    		$p = 3$ & $\mathbb{Z}$ & $0$ & $(\mathbb{Z}/{3^r})^{2}$ & $0$ & $(\mathbb{Z}/{3^r})^{3}$ & $\mathbb{Z}/{3^r}$ \\
    		otherwise & $\mathbb{Z}$ & $0$ & $(\mathbb{Z}/{p^r})^{2}$ & $0$ & $(\mathbb{Z}/{p^r})^{3}$ & $\mathbb{Z}/{p^r}$ \\
    	\end{tabular}
      \end{table}
    
      \begin{table}[h]
    	\centering
    	\begin{tabular}{c|cc}
    		$ $ & $H_6$ & $H_7$ \\
    		\hline
    		$p = 2$ & $(\mathbb{Z}/{2^r})^4$ & $(\mathbb{Z}/2)^4 \oplus (\mathbb{Z}/{2^r})^2$ \\
    		$p = 3$ & $(\mathbb{Z}/{3^r})^{2} \oplus (\mathbb{Z}/{3^{r+1}})^{2}$ & $(\mathbb{Z}/3)^{2} \oplus (\mathbb{Z}/{3^r})^{2}$ \\
    		otherwise & $(\mathbb{Z}/{p^r})^{4}$ & $(\mathbb{Z}/{p^r})^{2}$ \\
    	\end{tabular}
      \end{table}

    \subsection{Computation of $\Omega^{\mathrm{Spin}}_6(K(H_2(M), 2))$}\label{bordismgroup}
    
      We use the Atiyah\textendash Hirzebruch spectral sequence to determine $\Omega^{\mathrm{Spin}}_6(K(H_2(M), 2))$ and its bordism invariants, which are, informally speaking, given by (the Poincar\'e dual of) the first Pontryagin class and the fundamental class. The low-dimensional spin cobordism groups $\Omega^{\mathrm{Spin}}_*$ are as follows \cite{Mil63-spin}.
      \begin{table}[h]
    	\centering
    	\begin{tabular}{c|ccccccccc}
    		$i$ & $0$ & $1$ & $2$ & $3$ & $4$ & $5$ & $6$ & $7$ & $8$ \\
    		\hline
    		$\Omega^{\mathrm{Spin}}_i$ & $\mathbb{Z}$ & $\mathbb{Z}/2$ & $\mathbb{Z}/2$ & $0$ & $\mathbb{Z}$ & $0$ & $0$ & $0$ & $\mathbb{Z} \oplus \mathbb{Z}$ \\
    	\end{tabular}
      \end{table}
      
      We shall consider two special cases, namely $H_2(M) \cong \mathbb{Z}^g$ and $H_2(M) \cong {(\mathbb{Z}/{p^r})}^2$, before dealing with the general case. We consider $H_2(M) \cong \mathbb{Z}^g$ first. Let $x_k \in H^2({(\mathbb{CP}^{\infty})}^g)$ be the Kronecker dual of the generator $[{\mathbb{CP}}^1]_k$ corresponding to the $k$-th ${\mathbb{CP}}^{\infty}$ factor. Then $H_6({({\mathbb{CP}}^{\infty})}^g)$ is a free abelian group of rank $g + 2{\binom{g}{2}} + {\binom{g}{3}}$, generated by the Kronecker duals \begin{align*}
      	& (x_1^3)^*, \cdots, (x_g^3)^*, \\ & ({x_1^2}{x_2})^*, \cdots, ({x_{g-1}^2}{x_g})^*, ({x_1}{x_2^2})^*, \cdots, ({x_{g-1}}{x_g^2})^*, \\
      	& ({x_1}{x_2}{x_3})^*, \cdots, ({x_{g-2}}{x_{g-1}}{x_g})^*.
      \end{align*} Let $H_g$ be the subgroup of $H_6({({\mathbb{CP}}^{\infty})}^g)$ of the same rank freely generated by \begin{align*}
      & (x_1^3)^*, \cdots, (x_g^3)^*, \\ & 2({x_1^2}{x_2})^*, \cdots, 2({x_{g-1}^2}{x_g})^*, ({x_1^2}{x_2})^* + ({x_1}{x_2^2})^*, \cdots, ({x_{g-1}^2}{x_g})^* + ({x_{g-1}}{x_g^2})^*, \\
      & ({x_1}{x_2}{x_3})^*, \cdots, ({x_{g-2}}{x_{g-1}}{x_g})^*.
      \end{align*} In \cite[Section 8]{Kre09} an explicit computation of $\Omega^{\mathrm{Spin}}_6(\mathbb{CP}^{\infty})$ using the Atiyah\textendash Hirzebruch spectral sequence was supplied, which states that there is an isomorphism \[\Omega^{\mathrm{Spin}}_6(\mathbb{CP}^{\infty}) \xrightarrow{\cong} \mathbb{Z} \oplus H_6(\mathbb{CP}^{\infty}) \cong \mathbb{Z} \oplus \mathbb{Z}\] given by \[[N, \nu] \mapsto \left(\left\langle\frac{-p_1(N) \cup {\nu}^*(x_1) + 4{\nu}^*(x_1^3)}{24}, [N]\right\rangle, {\nu}_*([N])\right).\] The first component of this isomorphism is called a \emph{twisted $\hat{A}$-genus} \cite[Section 2]{Kre09}. This computation easily generalizes to $g > 1$ and we obtain the following.
      
      \begin{Prop}\label{prop:bordismfree}
      	There is an isomorphism \[\Omega^{\mathrm{Spin}}_6({(\mathbb{CP}^{\infty})}^g) \xrightarrow{\cong} \mathbb{Z}^g \oplus H_g \cong \mathbb{Z}^{2g + 2{\binom{g}{2}} + {\binom{g}{3}}}\] given by \[[N, \nu] \mapsto (\hat{A}_1([N, \nu]), \cdots, \hat{A}_g([N, \nu]), {\nu}_*([N])),\] where $\hat{A}_k : \Omega^{\mathrm{Spin}}_6({(\mathbb{CP}^{\infty})}^g) \rightarrow \mathbb{Z}$ is defined as \[\hat{A}_k([N, \nu]) \coloneqq \left\langle\frac{-p_1(N) \cup {\nu}^*(x_k) + 4{\nu}^*({x_k^3})}{24}, [N]\right\rangle.\]
      \end{Prop}
      
      \begin{proof}
      	Making use of \cite[Lemma in page 751]{Tei93}, we see from the Atiyah\textendash Hirzebruch spectral sequence that there are short exact sequences \[0 \rightarrow H_2({(\mathbb{CP}^{\infty})}^g) \rightarrow A \rightarrow {(\mathbb{Z}/2)}^g \rightarrow 0\] and \[0 \rightarrow A \rightarrow \Omega^{\mathrm{Spin}}_6({(\mathbb{CP}^{\infty})}^g) \xrightarrow{h} H_g \rightarrow 0\] where $h$ denotes the Hurewicz homomorphism $[N, \nu] \mapsto \nu_*([N])$. An obvious generalization of Kreck's argument \cite[Section 8]{Kre09} shows that there is a commutative diagram of short exact sequences
      	\[\begin{tikzcd}
      		0 \arrow[r] & H_2({(\mathbb{CP}^{\infty})}^g) \arrow[r] \arrow[d, "="] & A \arrow[r] \arrow[d, "(\hat{A}_k)"] & {(\mathbb{Z}/2)}^g \arrow[r] \arrow[d] & 0 \\
      		0 \arrow[r] & \mathbb{Z}^g \arrow[r, "\times 2"] & \mathbb{Z}^g \arrow[r] & {(\mathbb{Z}/2)}^g \arrow[r] & 0
      	\end{tikzcd}\] where $(\hat{A}_k) \coloneqq (\hat{A}_1, \cdots, \hat{A}_g) : A \rightarrow \mathbb{Z}^g$ is surjective. Thus ${(\mathbb{Z}/2)}^g \rightarrow {(\mathbb{Z}/2)}^g$ is an isomorphism. By 5-lemma $(\hat{A}_k) : A \rightarrow \mathbb{Z}^g$ is an isomorphism.
      \end{proof}
      
      We consider the case $H_2(M) \cong {(\mathbb{Z}/{p^r})}^2$ next. If $p \neq 2$, the Atiyah\textendash Hirzebruch spectral sequence yields an exact sequence \[0 \rightarrow H_2(K({(\mathbb{Z}/{p^r})}^2, 2))/{\mathrm{im} d^4_{7,0}} \rightarrow \Omega^{\mathrm{Spin}}_6(K({(\mathbb{Z}/{p^r})}^2, 2)) \xrightarrow{h} H_6(K({(\mathbb{Z}/{p^r})}^2, 2)) \rightarrow 0,\] where $h$ is the Hurewicz homomorphism, and the homomorphism \[H_2(K({(\mathbb{Z}/{p^r})}^2, 2)) \rightarrow H_2(K({(\mathbb{Z}/{p^r})}^2, 2))/{\mathrm{im} d^4_{7,0}} \rightarrow \Omega^{\mathrm{Spin}}_6(K({(\mathbb{Z}/{p^r})}^2, 2))\] is described as follows. Identify $H_2(K({(\mathbb{Z}/{p^r})}^2, 2))$ with $\pi_2(K({(\mathbb{Z}/{p^r})}^2, 2)) = {(\mathbb{Z}/{p^r})}^2$ via the Hurewicz isomorphism. Let the maps ${\ell}_1$, ${\ell}_2 : S^2 \rightarrow K({(\mathbb{Z}/{p^r})}^2, 2)$ represent the generators $(1, 0)$, $(0, 1) \in {(\mathbb{Z}/{p^r})}^2$ respectively. Let $K$ be a K3 surface, define the maps \[{\phi}_i : S^2 \times K \rightarrow S^2 \xrightarrow{{\ell}_i} K({(\mathbb{Z}/{p^r})}^2, 2), \; i = 1, 2.\] The homomorphism $H_2(K({(\mathbb{Z}/{p^r})}^2, 2)) \rightarrow \Omega^{\mathrm{Spin}}_6(K({(\mathbb{Z}/{p^r})}^2, 2))$ maps ${{\ell}_i}_*[S^2]$ to the bordism class $[S^2 \times K, {\phi}_i]$ for $i = 1,2$.
      
      It is known that $\langle p_1(K), [K] \rangle = -48$, so $p_1(S^2 \times K) \cap [S^2 \times K] = -48[S^2] \in H_2(S^2 \times K)$, thus \begin{align*}
      	& {\phi_1}_*(p_1(S^2 \times K) \cap [S^2 \times K]) = (-48, 0) \in {(\mathbb{Z}/{p^r})}^2, \\
      	& {\phi_2}_*(p_1(S^2 \times K) \cap [S^2 \times K]) = (0, -48) \in {(\mathbb{Z}/{p^r})}^2.
      \end{align*} If $p \neq 2, 3$, then $-48$ and $p^r$ are coprime to each other, so there is a nonzero integer $a_{p,r}$ such that $a_{p,r}({\phi_1}_*(p_1(S^2 \times K) \cap [S^2 \times K])) = (1, 0)$ and $a_{p,r}({\phi_2}_*(p_1(S^2 \times K) \cap [S^2 \times K])) = (0, 1)$. Thus for $p \neq 2, 3$ the differential $d^4_{7,0} = 0$ and we have proved the following.
      
      \begin{Prop}\label{prop:bordismtorsion}
      	For $p \neq 2, 3$ there is an isomorphism \[\Omega^{\mathrm{Spin}}_6(K({(\mathbb{Z}/{p^r})}^2, 2)) \xrightarrow{\cong} H_2(K({(\mathbb{Z}/{p^r})}^2, 2)) \oplus H_6(K({(\mathbb{Z}/{p^r})}^2, 2)) \cong {(\mathbb{Z}/{p^r})}^6\] given by \[[N, \nu] \mapsto (a_{p,r}{\nu}_*(p_1(N) \cap [N]), {\nu}_*([N])).\]
      \end{Prop}
       
      Now we consider the general case. Let $M$ be a simply connected closed spin 5-manifold. By the classification of simply connected 5-manifolds \cite{Sma62, Bar65}, $H_2(M)$ is isomorphic to \[\mathbb{Z}^g \oplus {(\mathbb{Z}/{p_1^{r_1}})}^2 \oplus \cdots \oplus {(\mathbb{Z}/{p_s^{r_s}})}^2,\] where $p_1, \cdots, p_s$ are prime numbers. Therefore \[K(H_2(M), 2) = (\mathbb{CP}^{\infty})^g \times K({(\mathbb{Z}/{p_1^{r_1}})}^2, 2) \times \cdots \times K({(\mathbb{Z}/{p_s^{r_s}})}^2, 2).\] Let $\pi_{p_i, r_i} : K(H_2(M), 2) \rightarrow K({(\mathbb{Z}/{p_i^{r_i}})}^2, 2)$ be the proiection map, let $T$ be the torsion subgroup of $H_2(M)$, and let $H_M$ be the subgroup of $H_6(K(H_2(M), 2))$ given by 
      \begin{align}
      	H_M \coloneqq H_g & \oplus H_6(K(T, 2))  \tag{2}\label{fundamentalclass}\\ 
      	& \oplus (H_2({(\mathbb{CP}^{\infty})}^g) \otimes H_4(K(T, 2))) \nonumber\\
      	& \oplus (H_4({(\mathbb{CP}^{\infty})}^g) \otimes H_2(K(T, 2))). \nonumber
      \end{align}
      
      \begin{Prop}\label{prop:bordism}
      	Let $M$ be a simply connected closed spin 5-manifold.
      	\begin{itemize}
      		\item[(1)] Suppose $H_2(M)$ has no 2- and 3-torsion, in other words, suppose $p_1 \geq 5$. Then there is an isomorphism \[\Omega^{\mathrm{Spin}}_6(K(H_2(M), 2)) \xrightarrow{\cong} \mathbb{Z}^g \oplus H_2(K({(\mathbb{Z}/{{p_1}^{r_1}})}^2, 2)) \oplus \cdots \oplus H_2(K({(\mathbb{Z}/{{p_s}^{r_s}})}^2, 2)) \oplus H_M\] given by  \begin{align*}
      			[N, \nu] \mapsto ( & \hat{A}_1([N, \nu]), \cdots, \hat{A}_g([N, \nu]), \\
      			& a_{p_1, r_1}{\pi_{p_1, r_1}}_*{\nu}_*(p_1(N) \cap [N]), \cdots, a_{p_s, r_s}{\pi_{p_s, r_s}}_*{\nu}_*(p_1(N) \cap [N]), {\nu}_*([N])).
      		\end{align*}
      		
      		\item[(2)] Suppose $H_2(M)$ has no 2-torsion. Then there is a short exact sequence \[0 \rightarrow \mathbb{Z}^g \oplus (T/{\mathrm{im} d^4_{7,0}}) \rightarrow \Omega^{\mathrm{Spin}}_6(K(H_2(M), 2)) \xrightarrow{h} H_M \rightarrow 0.\]
      	\end{itemize}
      \end{Prop}
      
      \begin{proof}
      	Using \cite[Lemma in page 751]{Tei93} we obtain from Atiyah\textendash Hirzebruch spectral sequence the following short exact sequences \[0 \rightarrow H_2({(\mathbb{CP}^{\infty})}^g) \oplus (T/{\mathrm{im} d^4_{7,0}}) \rightarrow A_M \rightarrow {(\mathbb{Z}/2)}^g \rightarrow 0\] and \[0 \rightarrow A_M \rightarrow \Omega^{\mathrm{Spin}}_6(K(H_2(M), 2)) \xrightarrow{h} H_M \rightarrow 0.\] If $H_2(M)$ has no 2- and 3-torsion, let $\tau$ be the homomorphism $\Omega^{\mathrm{Spin}}_6(K(H_2(M), 2)) \rightarrow \mathbb{Z}^g \oplus T$ defined as \begin{align*}
      	[N, \nu] \mapsto ( & \hat{A}_1([N, \nu]), \cdots, \hat{A}_g([N, \nu]), \\
      	& a_{p_1, r_1}{\pi_{p_1, r_1}}_*{\nu}_*(p_1(N) \cap [N]), \cdots, a_{p_s, r_s}{\pi_{p_s, r_s}}_*{\nu}_*(p_1(N) \cap [N])).
      	\end{align*} Then there is a commutative diagram \[\begin{tikzcd}
      	H_2({(\mathbb{CP}^{\infty})}^g) \oplus (T/{\mathrm{im} d^4_{7,0}}) \arrow[r] & A_M \arrow[d, "\tau"] \\
      	\mathbb{Z}^g \oplus T \arrow[u, "{(id, q)}"] \arrow[r, "{(\times 2, id)}"] & \mathbb{Z}^g \oplus T
      	\end{tikzcd}\] which shows that the quotient homomorphism $q : T \rightarrow T/{\mathrm{im} d^4_{7,0}}$ is injective, hence the differential $d^4_{7,0}$ is trivial and we obtain a commutative diagram of short exact sequences
      	\[\begin{tikzcd}
      	0 \arrow[r] & H_2({(\mathbb{CP}^{\infty})}^g) \oplus T \arrow[r] \arrow[d, "="] & A_M \arrow[d, "\tau"] \arrow[r] & {(\mathbb{Z}/2)}^g \arrow[d] \arrow[r] & 0 \\
      	0 \arrow[r] & \mathbb{Z}^g \oplus T \arrow[r, "{(\times 2, id)}"] & \mathbb{Z}^g \oplus T \arrow[r] & {(\mathbb{Z}/2)}^g \arrow[r] & 0
      	\end{tikzcd}\] The proofs of Proposition \ref{prop:bordismfree} and Proposition \ref{prop:bordismtorsion} assemble to show that $\tau$ is surjective, thus ${(\mathbb{Z}/2)}^g \rightarrow {(\mathbb{Z}/2)}^g$ is an isomorphism. By 5-lemma $\tau : A_M \rightarrow \mathbb{Z}^g \oplus T$ is an isomorphism. Thus there is a commutative diagram \[\begin{tikzcd}
      	0 \arrow[r] & A_M \arrow[r] \arrow[dr, "\cong"'] & \Omega^{\mathrm{Spin}}_6(K(H_2(M), 2)) \arrow[r, "h"] \arrow[d, "\tau"] & H_M \arrow[r] & 0 \\
      	& & \mathbb{Z}^g \oplus T & &
      \end{tikzcd}\] which completes the proof of (1). For $H_2(M)$ with nontrivial 3-torsion we have a commutative diagram of short exact sequences \[\begin{tikzcd}
      0 \arrow[r] & H_2({(\mathbb{CP}^{\infty})}^g) \oplus (T/{\mathrm{im} d^4_{7,0}}) \arrow[r] \arrow[d, "{(id, 0)}"] & A_M \arrow[d, "{(\hat{A}_k)}"] \arrow[r] & {(\mathbb{Z}/2)}^g \arrow[d] \arrow[r] & 0 \\
      0 \arrow[r] & \mathbb{Z}^g \arrow[r, "{\times 2}"] & \mathbb{Z}^g \arrow[r] & {(\mathbb{Z}/2)}^g \arrow[r] & 0.
      \end{tikzcd}\] As in the proof of Proposition \ref{prop:bordismfree}, Kreck's argument \cite[Section 8]{Kre09} shows that $(\hat{A}_k)$ is surjective and thus $A_M$ contains $\mathbb{Z}^g$ as a direct summand. The exact sequence \[0 \rightarrow H_2({(\mathbb{CP}^{\infty})}^g) \oplus (T/{\mathrm{im} d^4_{7,0}}) \rightarrow A_M \rightarrow {(\mathbb{Z}/2)}^g \rightarrow 0,\] where $H_2({(\mathbb{CP}^{\infty})}^g)$ is mapped to the $\mathbb{Z}^g$ summand by $\times 2$, implies that $A_M \cong \mathbb{Z}^g \oplus (T/{\mathrm{im} d^4_{7,0}})$. This proves (2).
      \end{proof}
      
      Thus the bordism group $\Omega^{\mathrm{Spin}}_6(K(H_2(M), 2))$ is abstractly isomorphic to $H_2(M) \oplus H_M \cong H_2(M) \oplus H_6(K(H_2(M), 2))$ when $H_2(M)$ has no 2- and 3-torsion. This group can be calculated using K\"unneth theorem, so the isomorphism type of $\Omega^{\mathrm{Spin}}_6(K(H_2(M), 2))$ is explicitly determined (see \ref{5.3}). When $H_2(M)$ contains nontrivial 2- or 3-torsion, the Atiyah\textendash Hirzebruch spectral sequence is more complicated and the computation of $\Omega^{\mathrm{Spin}}_6(K(H_2(M), 2))$ will be considered in a subsequent paper.
      
      \begin{Cor}\label{cor:forgetful}
      	Let $M$ be a simply connected closed spin 5-manifold with no 2-torsion in homology, then $\mathrm{MCG}(M, D) = \mathrm{MCG}(M)$ and $\mathcal{I}(M, D) = \mathcal{I}(M)$.
      \end{Cor}
      
      \begin{proof}
      	The bordism group $\Omega^{\mathrm{Spin}}_6(K(H_2(M), 2))$ has no 2-torsion. So is $\mathcal{I}(M, D)$ since it can be viewed as a subgroup of $\Omega^{\mathrm{Spin}}_6(K(H_2(M), 2))$ via $\rho$. The corollary follows from Lemma \ref{lem:wallseq} since $\pi_1(\mathrm{SO}(5)) \cong \mathbb{Z}/2$.
      \end{proof}
      
    \subsection{The canonical quotient homomorphism}
      
      Let $H$ denotes the abelian group \[\mathbb{Z}^g \oplus {(\mathbb{Z}/{p_1^{r_1}})}^2 \oplus \cdots \oplus {(\mathbb{Z}/{p_s^{r_s}})}^2,\] where $p_1, \cdots, p_s$ are prime numbers such that $p_i \geq 3$ for each $i$. Let $\delta : \mathbb{Z}^{g + 2s} \rightarrow H$ be the canonical quotient homomorphism, in other words, the product of copies of the identity map $id : \mathbb{Z} \rightarrow \mathbb{Z}$ and the canonical quotient homomorphisms $\mathbb{Z} \rightarrow \mathbb{Z}/{p^r}$ given by $1 \mapsto 1$. The corresponding map between Eilenberg\textendash MacLane spaces, ${(\mathbb{CP}^{\infty})}^{g + 2s} \rightarrow K(H, 2)$, is unique up to homotopy and is still denoted by $\delta$. In this subsection we investigate the induced homomorphism \begin{align*}
      \delta_* : \Omega^{\mathrm{Spin}}_6({(\mathbb{CP}^{\infty})}^{g + 2s}) & \rightarrow \Omega^{\mathrm{Spin}}_6(K(H, 2)) \\
      [N, \nu] & \mapsto [N, \delta \circ \nu].
      \end{align*} This homomorphism will be used in Section \ref{sec:5}.
      
      First we consider $\delta : \mathbb{CP}^{\infty} \rightarrow K(\mathbb{Z}/{p^r}, 2)$ induced by the canonical map $\mathbb{Z} \rightarrow \mathbb{Z}/{p^r}$.
      
      \begin{Lem}\label{lem:cyclic}
      	Suppose $p \geq 3$. The induced homomorphisms \[\delta_* : H_4(\mathbb{CP}^{\infty}) \rightarrow H_4(K(\mathbb{Z}/{p^r}, 2)), \quad \delta_* : H_6(\mathbb{CP}^{\infty}) \rightarrow H_6(K(\mathbb{Z}/{p^r}, 2))\] are surjective maps between cyclic groups.
      \end{Lem}
      
      \begin{proof}
      	We shall first study \[\delta_* : H_4(\mathbb{CP}^{\infty}; \mathbb{Z}/p) \rightarrow H_4(K(\mathbb{Z}/{p^r}, 2); \mathbb{Z}/p), \quad \delta_* : H_6(\mathbb{CP}^{\infty}; \mathbb{Z}/p) \rightarrow H_6(K(\mathbb{Z}/{p^r}, 2); \mathbb{Z}/p)\] using mod $p$ cohomology ring structure. Let $x \in H^2(\mathbb{CP}^{\infty}; \mathbb{Z}/p)$ and $\ell \in H^2(K(\mathbb{Z}/{p^r}, 2); \mathbb{Z}/p)$ be the cohomology fundamental classes of Eilenberg\textendash MacLane spaces, by universal coefficient theorem we have $\delta^*(\ell) = x$. According to Cartan (\cite{Car54}, see also \cite[Theorem 10.3]{May70}), $H^4(K(\mathbb{Z}/{p^r}, 2); \mathbb{Z}/p) \cong \mathbb{Z}/p$ is generated by $\ell^2$, and $H^6(K(\mathbb{Z}/{p^r}, 2); \mathbb{Z}/p) \cong \mathbb{Z}/p$ is generated by $\ell^3$. So $\delta$ induces isomorphisms \[H_4(\mathbb{CP}^{\infty}; \mathbb{Z}/p) \cong H_4(K(\mathbb{Z}/{p^r}, 2); \mathbb{Z}/p), \quad H_6(\mathbb{CP}^{\infty}; \mathbb{Z}/p) \cong H_6(K(\mathbb{Z}/{p^r}, 2); \mathbb{Z}/p),\] mapping the Kronecker duals of $x^2$ and $x^3$ to the Kronecker duals of $\ell^2$ and $\ell^3$ respectively. The canonical map $\mathbb{Z} \rightarrow \mathbb{Z}/p$ between the coefficients induces a commutative diagram
      	\[\begin{tikzcd}
      		H_*(\mathbb{CP}^{\infty}) \arrow[r] \arrow[d, "\delta_*"] & H_*(\mathbb{CP}^{\infty}; \mathbb{Z}/p) \arrow[d, "\delta_*"] \\
      		H_*(K(\mathbb{Z}/{p^r}, 2)) \arrow[r] & H_*(K(\mathbb{Z}/{p^r}, 2); \mathbb{Z}/p)
      	\end{tikzcd}\] where $H_*(\mathbb{CP}^{\infty}; \mathbb{Z}/p) = H_*(\mathbb{CP}^{\infty}) \otimes \mathbb{Z}/p$ and $H_*(K(\mathbb{Z}/{p^r}, 2); \mathbb{Z}/p) = H_*(K(\mathbb{Z}/{p^r}, 2)) \otimes \mathbb{Z}/p$ for $* = 4,6$. The canonical generator of $H_4(\mathbb{CP}^{\infty})$ is mapped by $\delta_*$ to $n \in \mathbb{Z}/{p^r} = H_4(K(\mathbb{Z}/{p^r}, 2))$, which is mapped to the Kronecker dual of $\ell^2$ in $H_4(K(\mathbb{Z}/{p^r}, 2); \mathbb{Z}/p)$ by the commutative diagram. Thus $n = m + kp$ for some integer $k, m$ satisfying $1 \leq m \leq p-1$, so $n$ is coprime to $p^r$ and hence is a generator of $H_4(K(\mathbb{Z}/{p^r}, 2))$. Similarly we see that $\delta_* : H_6(\mathbb{CP}^{\infty}) \rightarrow H_6(K(\mathbb{Z}/{p^r}, 2))$ maps the canonical generator of $H_6(\mathbb{CP}^{\infty})$ to a generator of $H_6(K(\mathbb{Z}/{p^r}, 2))$. This completes the proof.
      \end{proof}
      
      In general $\delta : {(\mathbb{CP}^{\infty})}^{g + 2s} \rightarrow K(H, 2)$ is the product of $id : \mathbb{CP}^{\infty} \rightarrow \mathbb{CP}^{\infty}$ and the canonical maps $\mathbb{CP}^{\infty} \rightarrow K(\mathbb{Z}/{p^r}, 2)$. Using this fact we can prove the following.
      
      \begin{Prop}\label{prop:quotient}
      	The homomorphism \[\delta_* : \Omega^{\mathrm{Spin}}_6({(\mathbb{CP}^{\infty})}^{g + 2s}) \rightarrow \Omega^{\mathrm{Spin}}_6(K(H, 2))\] is surjective.
      \end{Prop}
      
      \begin{proof}
      	Let $T$ be the torsion subgroup of $H$, therefore $H = \mathbb{Z}^g \oplus T$. The map $\delta : {(\mathbb{CP}^{\infty})}^{g + 2s} \rightarrow K(H, 2)$ induces a morphism between the corresponding Atiyah\textendash Hirzebruch spectral sequences, which yields the following commutative diagrams \[\begin{tikzcd}
      	0 \arrow[r] & \mathbb{Z}^{g + 2s} \arrow[r, "{\times 2}"] \arrow[d, "\delta"] & \mathbb{Z}^{g + 2s} \arrow[r] \arrow[d, "\delta_*"] & {(\mathbb{Z}/2)}^{g + 2s} \arrow[r] \arrow[d, "\delta_*"] & 0 \\
      	0 \arrow[r] & \mathbb{Z}^g \oplus (T/{\mathrm{im} d^4_{7,0}}) \arrow[r, "{(\times 2, id)}"] & \mathbb{Z}^g \oplus (T/{\mathrm{im} d^4_{7,0}}) \arrow[r] & {(\mathbb{Z}/2)}^g \arrow[r] & 0
      	\end{tikzcd}\] and \[\begin{tikzcd}
      	0 \arrow[r] & \mathbb{Z}^{g + 2s} \arrow[r] \arrow[d, "\delta_*"] & \Omega^{\mathrm{Spin}}_6({(\mathbb{CP}^{\infty})}^{g + 2s}) \arrow[r, "h"] \arrow[d, "\delta_*"] & H_{g + 2s} \arrow[r] \arrow[d, "\delta_*"] & 0 \\
      	0 \arrow[r] & \mathbb{Z}^g \oplus (T/{\mathrm{im} d^4_{7,0}}) \arrow[r] & \Omega^{\mathrm{Spin}}_6(K(H, 2)) \arrow[r, "h"] & \overline{H} \arrow[r] & 0
      	\end{tikzcd}\] where $\overline{H}$ denotes the group \[H_g \oplus H_6(K(T, 2)) \oplus (H_2({(\mathbb{CP}^{\infty})}^g) \otimes H_4(K(T, 2))) \oplus (H_4({(\mathbb{CP}^{\infty})}^g) \otimes H_2(K(T, 2))).\] The homomorphism $\delta_* : {(\mathbb{Z}/2)}^{g + 2s} \rightarrow {(\mathbb{Z}/2)}^g$ is the projection onto the first $g$ components and thus is surjective. Hence the middle vertical map in the first diagram, which is the left vertical map in the second diagram, is surjective. By Lemma \ref{lem:cyclic} and K\"unneth theorem, the homomorphism $\delta_* : H_{g+2s} \rightarrow \overline{H}$ is surjective (recall that $H$ has no 2-torsion). The proof is completed by 5-lemma.
      \end{proof}

    \section{The Torelli group of $M_g$}\label{sec:4}
      
      Recall that $M_g$ denotes the $g$-fold connected sum ${\#}^g(S^2 \times S^3)$. In this section we show that the map $\rho : \mathcal{I}(M_g) \rightarrow \Omega^{\mathrm{Spin}}_6({(\mathbb{CP}^{\infty})}^g)$ given by Corollary \ref{cor:forgetful} is an isomorphism. More precisely, (the isotopy classes of) the Dehn twists listed in Theorem \ref{thm:2} are in $\mathcal{I}(M_g)$ by the following Lemma \ref{lem:homologyeffect}, and we show that these Dehn twists are mapped to a basis of $\Omega^{\mathrm{Spin}}_6({(\mathbb{CP}^{\infty})}^g)$, by calculating the bordism invariants of their mapping torus constructions by explicit geometric considerations. As a byproduct we determine the action of $\mathcal{I}(M_g)$ on $\pi_3(M_g)$.
      
      \begin{Lem}\label{lem:homologyeffect}
      	Let $S : S^2 \times D^3 \hookrightarrow M$ be an embedding in a simply connected closed 5-manifold $M$, and let $\beta \in \pi_3(\mathrm{SO}(3))$. Then the Dehn twist $\mathbf{t}_{S, \beta}$ acts trivially on $H_*(M)$.
      \end{Lem}
      
      \begin{proof}
      	Clearly $\mathbf{t}_{S, \beta}$ preserves the orientation of $M$. By Hurewicz theorem and Whitney's embedding theorem we may assume $H_2(M) = \pi_2(M)$ is generated by embeddings of $S^2$, which by transversality are disjoint with the image of $S$, and thus are fixed by $\mathbf{t}_{S, \beta}$. Hence $\mathbf{t}_{S, \beta}$ acts trivially on $H_2(M)$, and therefore acts trivially on $H_*(M)$ by Poincar\'e duality.
      \end{proof}
      
      In Section \ref{sec:3} we have seen that the bordism invariants \[\Omega^{\mathrm{Spin}}_6({(\mathbb{CP}^{\infty})}^g) \xrightarrow{\cong} \mathbb{Z}^{2g + 2{\binom{g}{2}} + {\binom{g}{3}}}\] are given by \begin{align*}
      	[N, \nu] \mapsto ( & \hat{A}_1([N, \nu]), \cdots, \hat{A}_g([N, \nu]), \\
      	& \langle {\nu}^*(x_1^3), [N] \rangle, \cdots, \langle {\nu}^*(x_g^3), [N] \rangle, \\
      	& \langle {\nu}^*({x_1}{x_2^2}), [N] \rangle, \cdots, \langle {\nu}^*({x_{g-1}}{x_g^2}), [N] \rangle, \\
      	& \langle \frac{{\nu}^*({x_1^2}{x_2}) - {\nu}^*({x_1}{x_2^2})}{2}, [N] \rangle, \cdots, \langle \frac{{\nu}^*({x_{g-1}^2}{x_g}) - {\nu}^*({x_{g-1}}{x_g^2})}{2}, [N] \rangle, \\
      	& \langle {\nu}^*({x_1}{x_2}{x_3}), [N] \rangle, \cdots, \langle {\nu}^*({x_{g-2}}{x_{g-1}}{x_g}), [N] \rangle).
      \end{align*} If $[N, \nu] = [{(M_g)}_f, {\nu}_f]$ where $f$ is a diffeomorphism whose isotopy class lies in $\mathcal{I}(M_g)$, then the cohomology class ${{\nu}_f}^*(x_k) \in H^2({(M_g)}_f)$ is the Kronecker dual of $i_*{S_k}_*([S^2])$, thus these cohomology classes form a basis of $H^2({(M_g)}_f)$. To calculate the bordism invariants it suffices to determine, with respect to this basis, the first Pontryagin class $p_1({(M_g)}_f)$ (or more precisely, the evaluation $\langle p_1({(M_g)}_f) \cup {{\nu}_f}^*(x_k), [{(M_g)}_f] \rangle$ for each $k$) and the cubic form \[H^2({(M_g)}_f) \times H^2({(M_g)}_f) \times H^2({(M_g)}_f) \rightarrow \mathbb{Z}\] defined by $(x, y, z) \rightarrow \langle x \cup y \cup z, [{(M_g)_f}] \rangle$.
      
    \subsection{The first Pontryagin class}
    
      First we calculate $\langle p_1({(M_g)}_f) \cup {{\nu}_f}^*(x_k), [{(M_g)}_f] \rangle$ when $f$ is a Dehn twist about an embedding of $S^2 \times D^3$. We shall need the following geometric observation which is an analogue of \cite[Lemma 7.11]{KS25}. The proof is evident.
      
      \begin{Lem}
      	Let $S : S^2 \times D^3 \hookrightarrow M$ be an embedding in a closed 5-manifold $M$, and let $\beta \in \pi_3(\mathrm{SO}(3))$. Let $\mathbf{t}_{S, \beta}$ be the Dehn twist about $S$ and $\beta$. Choose an embedding $\ell : I \hookrightarrow S^1$ and identity $S^2 \times D^4$ with the image of the embedding $S \times \ell : S^2 \times D^3 \times I \hookrightarrow M \times S^1$. Then:
      	\begin{itemize}
      		\item[(1)] We have \[M_{\mathbf{t}_{S, \beta}} = (M \times S^1 \setminus S^2 \times D^4) \cup_{\varphi} (S^2 \times D^4)\] where $\varphi : S^2 \times S^3 \rightarrow S^2 \times S^3$ is given by $(x, y) \mapsto (\beta(y) \cdot x, y)$.
      		
      		\item[(2)] Let $\xi_{\beta}$ be the 3-dimensional vector bundle over $S^4$ with clutching function $\beta$. Let $D(\xi_{\beta})$ denotes the total space of the associated unit disc bundle, and $\partial D(\xi_{\beta})$ denotes the total space of the associated unit sphere bundle. Then $\partial D(\xi_{\beta}) = (S^2 \times D^4) \cup_{\varphi} (S^2 \times D^4)$ and $M_{\mathbf{t}_{S, \beta}} = \partial ((M \times D^2) \cup_{S^2 \times D^4} D(\xi_{\beta}))$.
      	\end{itemize}
      \end{Lem}
      
      Denote the oriented manifold $(M_g \times D^2) \cup_{S^2 \times D^4} D(\xi_{\beta})$ by $W_{S, \beta}$ so that ${(M_g)}_{\mathbf{t}_{S, \beta}} = \partial W_{S, \beta}$. The Mayer\textendash Vietoris sequence implies that $H_4(W_{S, \beta}) \cong \mathbb{Z}$, and it is generated by the fundamental class $[S^4]$ of the base of the bundle $\xi_{\beta}$ (realized as the zero section of $\xi_{\beta}$, thus a submanifold of $D(\xi_{\beta})$), and there is an exact sequence \[0 \rightarrow H_3(M_g) \xrightarrow{\cong} H_3(W_{S, \beta}) \xrightarrow{0} H_2(S^2) \xrightarrow{S_*} H_2(M_g) \rightarrow H_2(W_{S, \beta}) \rightarrow 0\] if $S_*$ is nontrivial. By dimensional reason the Lefschetz duality yields a commutative diagram \[\begin{tikzcd}
      H^4(W_{S, \beta}) \arrow[d, "\cong"] \arrow[r, "j^*"] & H^4({(M_g)}_{\mathbf{t}_{S, \beta}}) \arrow[d, "\cong"] \arrow[r] & H^5(W_{S, \beta}, {(M_g)}_{\mathbf{t}_{S, \beta}}) \arrow[d, "\cong"] \\
      H_3(W_{S, \beta}, {(M_g)}_{\mathbf{t}_{S, \beta}}) \arrow[r] & H_2({(M_g)}_{\mathbf{t}_{S, \beta}}) \arrow[r] & H_2(W_{S, \beta}) \\
      & H_2(M_g) \arrow[u, "\cong"', "i_*"] \arrow[ur] &
      \end{tikzcd}\] where $j$ denotes the inclusion map. Note that $H^4(W_{S, \beta}) \cong \mathbb{Z}$ and is generated by ${[S^4]}^*$, the Kronecker dual of $[S^4]$, so the commutative diagram implies that the Poincar\'e dual of $i_*S_*([S^2])$ is $\pm j^*({[S^4]}^*)$. In fact the ambiguous sign ``$\pm$" can be removed.
      
      \begin{Lem}\label{lem:dual}
      	The Poincar\'e dual of $i_*S_*([S^2])$ is $j^*({[S^4]}^*)$.
      \end{Lem}
      
      We postpone the proof of this lemma to the end of this subsection and proceed. Suppose $[S] = d_1[S_1] + \cdots + d_g[S_g] \in \pi_2(M_g)$, then \[i_*S_*([S^2]) = \sum_{k = 1}^{g} d_ki_*{S_k}_*([S^2]).\] Recall that $\alpha \in \pi_3(\mathrm{SO}(3))$ is the preferred generator, suppose $\beta = n\alpha$, then \[\langle p_1(W_{S, \beta}), [S^4] \rangle = \langle p_1(D(\xi_{\beta})), [S^4] \rangle = \langle p_1(\xi_{\beta}), [S^4] \rangle = 4n\] since $H^*(W_{S, \beta})$ has no 2-torsion. Thus \[p_1({(M_g)}_{\mathbf{t}_{S, \beta}}) = j^*(p_1(W_{S, \beta})) = 4nj^*({[S^4]}^*).\] Since ${{\nu}_{\mathbf{t}_{S, \beta}}}^*(x_k)$ is the Kronecker dual of $i_*{S_k}_*([S^2])$, Lemma \ref{lem:dual} implies the following.
      
      \begin{Lem}\label{lem:pontryaginclass}
      	For $S$ and $\beta$ as above and for each $k$, we have \[\langle p_1({(M_g)}_{\mathbf{t}_{S, \beta}}) \cup {{\nu}_{\mathbf{t}_{S, \beta}}}^*(x_k), [{(M_g)}_{\mathbf{t}_{S, \beta}}] \rangle = 4nd_k.\]
      \end{Lem}
      
      Finally we prove Lemma \ref{lem:dual} and thus complete the calculation of the Pontryagin class.
      
      \begin{proof}[Proof of Lemma \ref{lem:dual}]
      	Still suppose $[S] = d_1[S_1] + \cdots + d_g[S_g]$. Note that the embedding $\ell : S^2 = S^2 \times \{*\} \hookrightarrow S^2 \times D^4 \hookrightarrow {(M_g)}_{\mathbf{t}_{S, \beta}}$ is isotopic to $i \circ S$, so we have ${\ell}_*([S^2]) = \sum_{k = 1}^{g} d_ki_*{S_k}_*([S^2])$. By a well-known result of Thom \cite[Theorem II. 27]{Tho54}, the Poincar\'e dual of ${{\nu}_{\mathbf{t}_{S, \beta}}}^*(x_k)$ can be represented by the fundamental class of a 4-dimensional closed submanifold $N_k \subset {(M_g)}_{\mathbf{t}_{S, \beta}}$ for each $k$. The intersection number of submanifolds $\ell(S^2)$ and $N_k$ is $d_k$. Note that $\ell$ extends to an embedding $\overline{\ell} : D^3 \hookrightarrow D(\xi_{\beta}) \hookrightarrow W_{S, \beta}$ of a fiber of the unit disc bundle. Using a collar neighborhood we may isotope $N_k$ into the interior of $W_{S, \beta}$ and obtain a closed submanifold $N_k' \subset W_{S, \beta}$, whose fundamental class represents $j_*([N_k]) \in H_4(W_{S, \beta})$. Then the intersection number of $\overline{\ell}(D^3)$ and $N_k'$ is $d_k$. The intersection number of $\overline{\ell}(D^3)$ and the base $S^4$ is 1, so ${\overline{\ell}}_*([D^3, S^2]) \in H_3(W_{S, \beta}, {(M_g)}_{\mathbf{t}_{S, \beta}})$ is the Lefschetz dual of ${[S^4]}^*$, thus $j_*([N_k]) = d_k[S^4]$. It follows that \begin{align*}
      		d_k = \langle {[S^4]}^*, j_*([N_k]) \rangle & = \langle {{\nu}_{\mathbf{t}_{S, \beta}}}^*(x_k) \cup j^*({[S^4]}^*), [{(M_g)}_{\mathbf{t}_{S, \beta}}] \rangle \\
      		& = \langle {{\nu}_{\mathbf{t}_{S, \beta}}}^*(x_k), j^*({[S^4]}^*) \cap [{(M_g)}_{\mathbf{t}_{S, \beta}}] \rangle,
      	\end{align*} thus \[j^*({[S^4]}^*) \cap [{(M_g)}_{\mathbf{t}_{S, \beta}}] = \sum_{k = 1}^{g} d_ki_*{S_k}_*([S^2]) = i_*S_*([S^2])\] by the fact that ${{\nu}_{\mathbf{t}_{S, \beta}}}^*(x_k)$ is the Kronecker dual of $i_*{S_k}_*([S^2])$.
      \end{proof}
    
    \subsection{The cubic form}\label{4.2}
    
      In this subsection we calculate the cubic form on $H^2({(M_g)}_{\mathbf{t}_{S, \alpha}})$ induced by cup product, where ${\mathbf{t}_{S, \alpha}}$ is one of the Dehn twists listed in Theorem \ref{thm:2}. We begin by sketching the strategy to make the calculation a bit more transparent. A diffeomorphism of $M_g$ is homotopic to a cellular map and we denote by $f$ its restriction on the 3-skeleton $\vee_g(S^2 \vee S^3)$ of $M_g$. The map $f$ acts trivially on $H_*(\vee_g(S^2 \vee S^3))$ if the original diffeomorphism acts trivially on $H_*(M_g)$. The CW structure of the mapping torus $(\vee_g(S^2 \vee S^3))_f$, which is determined by the action of $f$ on $\pi_3(\vee_g(S^2 \vee S^3))$, enables us to calculate the cup product \[H^2((\vee_g(S^2 \vee S^3))_f) \times H^2((\vee_g(S^2 \vee S^3))_f) \xrightarrow{\cup} H^4((\vee_g(S^2 \vee S^3))_f)\] using Hopf invariant. This determines the cubic form on the cohomology of the mapping torus of the original diffeomorphism.
      
    \subsubsection{Action of Dehn twists on $\pi_3(M_g)$}
      
      For $k = 1, \cdots, g$ let $T_k : S^3 \hookrightarrow M_g$ be the standard embedding of $S^3$ in the $k$-th $S^2 \times S^3$ summand. The homotopy group \[\pi_3(M_g) = \pi_3(\vee_g(S^2 \vee S^3)) \cong \mathbb{Z}^{2g + {\binom{g}{2}}}\] is freely generated by $[S_1 \circ \eta], \cdots, [S_g \circ \eta], [T_1], \cdots, [T_g], [S_1, S_2], \cdots, [S_{g_1}, S_g]$. Here $\eta$ denotes the Hopf map and $[S_k, S_l]$ denotes the Whitehead product of $[S_k]$ and $[S_l]$. Let $f$ be a self-map of $M_g$ acting trivially on homology, then $f \circ S_k \simeq S_k$ for each $k$, thus $f$ acts trivially on $[S_k \circ \eta]$ and $[S_k, S_l]$. Now we suppose $f$ is a Dehn twist about a 2-sphere. The action of such $f$ is given by the following lemma.
      
      \begin{Lem}\label{lem:homotopyeffect}
      	Let $M$ be a compact oriented 5-manifold, $S : S^2 \times D^3 \hookrightarrow M$ and $T : S^3 \hookrightarrow M$ be embeddings in the interior of $M$, and let $\mathbf{t}_{S, \beta}$ be the Dehn twist about $S$ and $\beta \in \pi_3(\mathrm{SO}(3))$. Suppose the intersection number of $S(S^2)$ and $T(S^3)$ is $n$. Then \[{\mathbf{t}_{S, \beta}}_*([S]) = [S], \quad {\mathbf{t}_{S, \beta}}_*([T]) = \frac{n}{4}\langle p_1(\xi_{\beta}), [S^4] \rangle \cdot [S \circ \eta] + [T],\] where $\xi_{\beta}$ is the 3-dimensional vector bundle over $S^4$ with clutching function $\beta$, and $\eta : S^3 \rightarrow S^2$ is the Hopf map.
      \end{Lem}
      
      \begin{proof}
      	Using Whitney trick we may assume $S(S^2)$ and $T(S^3)$ intersect at $n$ points transversally. The lemma follows from the facts that there is an isomorphism $\pi_3(\mathrm{SO}(3)) \rightarrow \mathbb{Z}$ given by $\beta \mapsto \frac{1}{4}\langle p_1(\xi_{\beta}), [S^4] \rangle$, and the effect of $\mathbf{t}_{S, \beta}$ on $S(S^2)$ is encoded by a map $\mathrm{SO}(3) \rightarrow S^2$ (defined in, for example, \cite[page 115]{Ste51}) which induces an isomorphism $\pi_3(\mathrm{SO}(3)) \rightarrow \pi_3(S^2)$ mapping the preferred generator $\alpha$ to $[\eta]$.
      \end{proof}
      
      Now we determine the action of the Dehn twists listed in Theorem \ref{thm:2} on $\pi_3(M_g)$. It suffices to determine their actions on $[T_1], \cdots, [T_g]$. By Lemma \ref{lem:homotopyeffect} the only nontrivial actions are as follows:
      \begin{align*}
      	\mathbf{t}_{S_k, \alpha} : \; & [T_k] \mapsto [S_k \circ \eta] + [T_k]; \\
      	\mathbf{t}_{S_k \# S_k, \alpha} : \; & [T_k] \mapsto 2[(S_k \# S_k) \circ \eta] + [T_k]; \\
      	\mathbf{t}_{S_k \# S_l, \alpha} : \; & [T_k] \mapsto [(S_k \# S_l) \circ \eta] + [T_k], \\
      	& [T_l] \mapsto [(S_k \# S_l) \circ \eta] + [T_l]; \\
      	\mathbf{t}_{S_k \# S_k \# S_l, \alpha} : \; & [T_k] \mapsto 2[(S_k \# S_k \# S_l) \circ \eta] + [T_k], \\
      	& [T_l] \mapsto [(S_k \# S_k \# S_l) \circ \eta] + [T_l]; \\
      	\mathbf{t}_{S_k \# S_l \# S_h, \alpha} : \; & [T_k] \mapsto [(S_k \# S_l \# S_h) \circ \eta] + [T_k], \\
      	& [T_l] \mapsto [(S_k \# S_l \# S_h) \circ \eta] + [T_l], \\
      	& [T_h] \mapsto [(S_k \# S_l \# S_h) \circ \eta] + [T_h].
      \end{align*}
      
      Consider the embedding $({\#}^{d_1} S_1) \# \cdots \# ({\#}^{d_g} S_g) : S^2 \hookrightarrow M_g$ where $d_k \neq 0$ for some $k$. We calculate the homotopy class $[\bigl(({\#}^{d_1} S_1) \# \cdots \# ({\#}^{d_g} S_g)\bigr) \circ \eta] \in \pi_3(M_g)$ using (the geometric interpretation of) Hopf invariant. Recall that the Hopf invariant $\pi_3(S^2) \rightarrow \mathbb{Z}$ is an isomorphism mapping $[\eta]$ to 1, and, given a map $f$, choose two regular values $x$, $y$ of a smooth approximation of $f$, let $X$ and $Y$ be the preimage of $x$ and $y$ respectively, then the Hopf invariant of $f$ is equal to the linking number of $X$ and $Y$.
      
      Note that for embeddings $S, \, S' : S^2 \hookrightarrow M_g$ we have $[S \# S'] = [S] + [S'] \in \pi_2(M_g)$, since we may fill in the ``neck" of the connected sum (which is formed by thickening a curve joining $S(S^2)$ and $S'(S^2)$ and taking the boundary) with a 3-disc and collapse it into a point, this defines a homotopy equivalence of $M_g$ which is homotopic to $id_{M_g}$, and the composition of $S \# S'$ and this homotopy equivalence is a map representing the homotopy class $[S] + [S']$. Thus it suffices to determine the homotopy class of the map \[S^3 \xrightarrow{\eta} S^2 \xrightarrow{\pi_{d_1 + \cdots + d_g}} {(\vee_{d_1} S^2)} \vee \cdots \vee {(\vee_{d_g} S^2)} \xrightarrow{{(\vee_{d_1} id)} \vee \cdots \vee {(\vee_{d_g} id)}} \vee_g S^2 \xrightarrow{S_1 \vee \cdots \vee S_g} M_g\] where $\pi_d$ denotes the quotient map $S^2 \rightarrow \vee_d S^2$ for each $d > 0$. A concrete construction of $\pi_d$ is given by choosing $d$ disjoint embedded 2-discs in $S^2$ and collapsing the complement of the interior of these discs into a point.
      
      We shall first determine the homotopy class $[\pi_d \circ \eta] \in \pi_3(\vee_d S^2)$. Let $i_k : S^2 \hookrightarrow \vee_d S^2$ be the inclusion of the $k$-th copy of $S^2$, then the homotopy group $\pi_3(\vee_d S^2)$ is freely generated by $[i_1 \circ \eta], \cdots, [i_d \circ \eta]$ and the Whitehead products $[i_1, i_2], \cdots, [i_{d-1}, i_d]$. Suppose \[[\pi_d \circ \eta] = \sum_{1 \leq k \leq d} a_k[i_k \circ \eta] + \sum_{1 \leq k < l \leq d} a_{k,l}[i_k, i_l].\] For each $k$ there is a commutative diagram \[\begin{tikzcd}
      	& & S^2 \\
      	S^3 \arrow[r, "\eta"] & S^2 \arrow[r, "\pi_d"] \arrow[ur, "p_k"] & \vee_d S^2 \arrow[u, "q_k"']
      \end{tikzcd}\] where $q_k$ is the map collapsing the complement of $i_k(S^2)$ into a point, and $p_k$ is the map collapsing the corresponding discs into a point. Note that $p_k$ is homotopic to $id_{S^2}$. Since the induced homomorphism ${q_k}_* : \pi_3(\vee_d S^2) \rightarrow \pi_3(S^2)$ maps $[i_k \circ \eta]$ to $[\eta]$ and maps other generators to zero, by the commutative diagram we have $a_k = 1$ for each $k$. We furthermore consider for each $k < l$ the following commutative diagram \[\begin{tikzcd}
      & & S^2 \vee S^2 \arrow[r, "id \vee id"] & S^2 \\
      S^3 \arrow[r, "\eta"] & S^2 \arrow[r, "\pi_d"] \arrow[ur, "p_{k,l}"] & \vee_d S^2 \arrow[u, "q_{k,l}"'] &
      \end{tikzcd}\] where $q_{k,l}$ is the map collapsing the complement of $i_k(S^2) \vee i_l(S^2)$ into a point, and $p_{k,l}$ is the map collapsing the corresponding discs into a point. We note that $p_{k,l}$ is homotopic to the quotient map $\pi_2 : S^2 \rightarrow S^2 \vee S^2$. The induced homomorphism ${q_{k,l}}_* : \pi_3(\vee_d S^2) \rightarrow \pi_3(S^2 \vee S^2)$ maps $[\pi_d \circ \eta]$ to $[i_k \circ \eta] + [i_l \circ \eta] + a_{k,l}[i_k, i_l]$, the latter homotopy class is mapped by ${(id \vee id)}_*$ to $2[\eta] + a_{k,l}[id, id] = (2 + 2a_{k,l})[\eta]$. Using the aforementioned geometric interpretation, we see that the Hopf invariant of $[(\vee_d id) \circ \pi_d \circ \eta]$ is $d^2$, thus $[(id \vee id) \circ \pi_2 \circ \eta] = 4[\eta]$. Therefore we have $a_{k,l} = 1$. Thus \[[\pi_d \circ \eta] = \sum_{1 \leq k \leq d} [i_k \circ \eta] + \sum_{1 \leq k < l \leq d} [i_k, i_l].\] It is then clear that the homomorphism \[({(\vee_{d_1} id)} \vee \cdots \vee {(\vee_{d_g} id)})_* : \pi_3({(\vee_{d_1} S^2)} \vee \cdots \vee {(\vee_{d_g} S^2)}) \rightarrow \pi_3(\vee_g S^2)\] maps $[\pi_{d_1 + \cdots + d_g} \circ \eta]$ to the homotopy class \[d_1[i_1 \circ \eta] + \cdots + d_g[i_g \circ \eta] + \binom{d_1}{2}[i_1, i_1] + \cdots + \binom{d_g}{2}[i_g, i_g] + d_1d_2[i_1, i_2] + \cdots + d_{g-1}d_g[i_{g-1}, i_g]\] which is \[\sum_{1 \leq k \leq g} d_k^2[i_k \circ \eta] + \sum_{1 \leq k < l \leq g} d_kd_l[i_k, i_l]\] since $[i_k, i_k] = 2[i_k \circ \eta]$. The obvious effect of the induced homomorphism of $S_1 \vee \cdots \vee S_g$ proves the following.
      
      \begin{Lem}\label{lem:homotopyclass}
      	We have \[[\bigl(({\#}^{d_1} S_1) \# \cdots \# ({\#}^{d_g} S_g)\bigr) \circ \eta] = \sum_{1 \leq k \leq g} d_k^2[S_k \circ \eta] + \sum_{1 \leq k < l \leq g} d_kd_l[S_k, S_l].\]
      \end{Lem}
      
      Combining Lemma \ref{lem:homotopyeffect} and Lemma \ref{lem:homotopyclass} one could determine the action of a Dehn twist about a 2-sphere on $\pi_3(M_g)$. In fact, we have seen in \ref{results} that such a Dehn twist is isotopic to a Dehn twist about $({\#}^{d_1} S_1) \# \cdots \# ({\#}^{d_g} S_g)$ and $\beta = n\alpha$ for some $d_k$ and $n$. The action of the latter Dehn twist is given by \[[T_k] \mapsto \frac{n}{4}\Bigl(\sum_{1 \leq k \leq g} d_k^3[S_k \circ \eta] + \sum_{1 \leq k < l \leq g} d_k^2d_l[S_k, S_l]\Bigr) + [T_k]\] for each $k$. In particular, the only nontrivial actions of the Dehn twists listed in Theorem \ref{thm:2} on $\pi_3(M_g)$ are as follows:
      \begin{align*}
      	\mathbf{t}_{S_k, \alpha} : \; & [T_k] \mapsto [S_k \circ \eta] + [T_k]; \\
      	\mathbf{t}_{S_k \# S_k, \alpha} : \; & [T_k] \mapsto 8[S_k \circ \eta] + [T_k]; \\
      	\mathbf{t}_{S_k \# S_l, \alpha} : \; & [T_k] \mapsto [S_k \circ \eta] + [S_l \circ \eta] + [S_k, S_l] + [T_k], \\
      	& [T_l] \mapsto [S_k \circ \eta] + [S_l \circ \eta] + [S_k, S_l] + [T_l]; \\
      	\mathbf{t}_{S_k \# S_k \# S_l, \alpha} : \; & [T_k] \mapsto 8[S_k \circ \eta] + 2[S_l \circ \eta] + 4[S_k, S_l] + [T_k], \\
      	& [T_l] \mapsto 4[S_k \circ \eta] + [S_l \circ \eta] + 2[S_k, S_l] + [T_l]; \\
      	\mathbf{t}_{S_k \# S_l \# S_h, \alpha} : \; & [T_k] \mapsto [S_k \circ \eta] + [S_l \circ \eta] + [S_h \circ \eta] + [S_k, S_l] + [S_k, S_h] + [S_l, S_h] + [T_k], \\
      	& [T_l] \mapsto [S_k \circ \eta] + [S_l \circ \eta] + [S_h \circ \eta] + [S_k, S_l] + [S_k, S_h] + [S_l, S_h] + [T_l], \\
      	& [T_h] \mapsto [S_k \circ \eta] + [S_l \circ \eta] + [S_h \circ \eta] + [S_k, S_l] + [S_k, S_h] + [S_l, S_h] + [T_h].
      \end{align*}
      
    \subsubsection{The cup product of $H^*((\vee_g(S^2 \vee S^3))_f)$}
      For brevity we denote $\vee_g(S^2 \vee S^3)$ by $V_g$. Let $i_k : S^2 \hookrightarrow V_g$ be the inclusion of the $k$-th copy of $S^2$, and let $j_k : S^3 \hookrightarrow V_g$ be the inclusion of the $k$-th copy of $S^3$. If we view $V_g$ as the 3-skeleton of $M_g$ then $i_k \simeq S_k$ and $j_k \simeq T_k$. Let $f : V_g \rightarrow V_g$ be a cellular map acting trivially on homology, then $f \circ i_k \simeq i_k$ and we have \[f = f \circ id_{V_g} = f \circ (i_1 \vee \cdots \vee i_g \vee j_1 \vee \cdots \vee j_g) \simeq i_1 \vee \cdots \vee i_g \vee (f \circ j_1) \vee \cdots \vee (f \circ j_g).\] Denote $i_k(S^2)$ by $S^2_k$, and denote $j_k(S^3)$ by $S^3_k$. The CW structure of the mapping torus ${(V_g)}_f$ is \[\Bigl(S^1 \vee S^2_1 \vee \cdots \vee S^2_g \vee S^3_1 \vee \cdots \vee S^3_g\Bigr) \cup_{[\ell, i_1]} e^3_1 \cup \cdots \cup_{[\ell, i_g]} e^3_g \cup_{\gamma_1} e^4_1 \cup \cdots \cup_{\gamma_g} e^4_g\] where $e^n_k$ is an $n$-cell for $n = 3, 4$ and each $k$, the map $\ell$ is the inclusion of the $S^1$, and the attaching map $\gamma_k$ is $[\ell, j_k] + [f \circ j_k] - [j_k]$. There is a ``Wang sequence" for mapping tori of general maps (see \cite[Example 2.48]{Hat02}), from which we compute $H_*({(V_g)}_f)$ and deduce that $H_2({(V_g)}_f) \cong \mathbb{Z}^g$ is freely generated by the cells $[S^2_k]$, and $H_4({(V_g)}_f) \cong \mathbb{Z}^g$ is freely generated by the cells $[e^4_k]$ (in fact the boundary maps in the cellular chain complex of ${(V_g)}_f$ are zero). For each $k$ let $\overline{x}_k \in H^2({(V_g)}_f)$ be the Kronecker dual of $[S^2_k]$, let $\overline{y}_k \in H^4({(V_g)}_f)$ be the Kronecker dual of $[e^4_k]$. Then $H^2({(V_g)}_f) \cong \mathbb{Z}^g$ is freely generated by $\overline{x}_k$'s and $H^4({(V_g)}_f) \cong \mathbb{Z}^g$ is freely generated by $\overline{y}_k$'s. To calculate the cup product \[H^2((V_g)_f) \times H^2((V_g)_f) \xrightarrow{\cup} H^4((V_g)_f)\] it suffices to calculate the cohomology classes $\overline{x}_k\overline{x}_l \in H^4((V_g)_f)$ for $k \leq l$. Thus we may replace ${(V_g)}_f$ by the quotient complex \[Q_f \coloneqq \Bigl(S^2_1 \vee \cdots \vee S^2_g\Bigr) \cup_{\sigma_1} e^4_1 \cup \cdots \cup_{\sigma_g} e^4_g\] (here $\sigma_k = [f \circ j_k] - [j_k]$) obtained from the subcomplex of ${(V_g)}_f$, \[\Bigl(S^1 \vee S^2_1 \vee \cdots \vee S^2_g \vee S^3_1 \vee \cdots \vee S^3_g\Bigr) \cup_{\gamma_1} e^4_1 \cup \cdots \cup_{\gamma_g} e^4_g,\] by collapsing $S^1 \vee S^3_1 \vee \cdots \vee S^3_g$ to a point, since their second and fourth (co)homology groups are identical respectively.
    
      The method is to consider certain quotient spaces of $Q_f$ whose cohomology ring can be calculated using Hopf invariant. These quotient spaces are obtained by collapsing irrelevant 2-cells of $Q_f$ to a point. Precisely, for each $k$ we suppose \[\sigma_k = \sum_{1 \leq l \leq g} d_{k,l}[i_l \circ \eta] + \sum_{1 \leq l < h \leq g} d_{k,l,h}[i_l, i_h].\] To calculate $\overline{x}_k^2$ we consider the quotient map \[q_k : Q_f \rightarrow Q_k \coloneqq S^2_k \cup_{d_{1,k} [\eta]} e^4_1 \cup \cdots \cup_{d_{g,k} [\eta]} e^4_g\] given by collapsing $S^2_1 \vee \cdots \vee S^2_{k-1} \vee S^2_{k+1} \vee \cdots \vee S^2_g$ to a point. The induced homomorphisms of $q_k$ on homology are \begin{align*}
      	& {q_k}_* : H_2(Q_f) \rightarrow H_2(Q_k), \quad [S^2_k] \mapsto [S^2_k], \quad [S^2_t] \mapsto 0, \quad t \neq k; \\
      	& {q_k}_* : H_4(Q_f) \rightarrow H_4(Q_k), \quad [e^4_t] \mapsto [e^4_t], \quad 1 \leq t \leq g,
      \end{align*} so ${q_k}^* : H^2(Q_k) \rightarrow H^2(Q_f)$ maps $x$, the Kronecker dual of $[S^2_k] \in H_2(Q_k)$, to $\overline{x}_k$, and $H^4(Q_k)$ and $H^2(Q_f)$ are identical via ${q_k}^*$. For each $t$ consider the subcomplex \[S^2_k \cup_{d_{t,k} [\eta]} e^4_t \subset Q_k.\] The induced homomorphism of the inclusion map $H^4(Q_k) \rightarrow H^4(S^2_k \cup_{d_{t,k} [\eta]} e^4_t)$ maps $\overline{y}_t$ to $y$, the Kronecker dual of $[e^4_t]$, and maps other generators to 0. Since the Hopf invariant of $d_{t,k} [\eta]$ is $d_{t,k}$, we have $x^2 = d_{t,k}y$, thus \[\overline{x}_k^2 = \sum_{1 \leq t \leq g} d_{t,k}\overline{y}_t.\] Similarly, to calculate $\overline{x}_k\overline{x}_l$ we consider the quotient map \[q_{k,l} : Q_f \rightarrow Q_{k,l} \coloneqq \Bigl(S^2_k \vee S^2_l\Bigr) \cup_{\sigma'_1} e^4_1 \cup \cdots \cup_{\sigma'_g} e^4_g\] where $\sigma'_t = d_{t,k}[i_k \circ \eta] + d_{t,l}[i_l \circ \eta] + d_{t,k,l}[i_k, i_l]$ for each $t$. We furthermore consider for each $t$ the subcomplex \[\Bigl(S^2_k \vee S^2_l\Bigr) \cup_{\sigma'_t} e^4_t \subset Q_{k,l}\] and the quotient map \[\Bigl(S^2_k \vee S^2_l\Bigr) \cup_{\sigma'_t} e^4_t \rightarrow S^2 \cup_{(d_{t,k} + d_{t,l} + 2d_{t,k,l})[\eta]} e^4_t\] given by mapping $S^2_k$ and $S^2_l$ to $S^2$ via $id_{S^2} \vee id_{S^2}$. Similar to the calculation of $\overline{x}_k^2$, by checking the effect of these maps on homology we deduce that \[{(\overline{x}_k + \overline{x}_l)}^2 = \sum_{1 \leq t \leq g} (d_{t,k} + d_{t,l} + 2d_{t,k,l})\overline{y}_t,\] thus \[\overline{x}_k\overline{x}_l = \sum_{1 \leq t \leq g} d_{t,k,l}\overline{y}_t.\]
      
      We summarize these calculations into a lemma.
      
      \begin{Lem}\label{lem:cup}
      	Let $f : V_g \rightarrow V_g$ be a cellular map acting trivially on homology. Suppose \[f_*([j_k]) = \sum_{1 \leq l \leq g} d_{k,l}[i_l \circ \eta] + \sum_{1 \leq l < h \leq g} d_{k,l,h}[i_l, i_h] + [j_k] \in \pi_3(V_g)\] for each $k$. Then the cup products of elements of $H^2({(V_g)}_f)$ are given by \[\overline{x}_k^2 = \sum_{1 \leq t \leq g} d_{t,k}\overline{y}_t, \quad \overline{x}_k\overline{x}_l = \sum_{1 \leq t \leq g} d_{t,k,l}\overline{y}_t.\]
      \end{Lem}
      
      In particular we write down the results for $f$ whose nontrivial action on $\pi_3(V_g)$ is of one of the following types:
      \begin{align*}
    	  \text{type I} : \; & [j_k] \mapsto [i_k \circ \eta] + [j_k]; \\
    	  \text{type II} : \; & [j_k] \mapsto 8[i_k \circ \eta] + [j_k]; \\
    	  \text{type III} : \; & [j_k] \mapsto [i_k \circ \eta] + [i_l \circ \eta] + [i_k, i_l] + [j_k], \\
    	  & [j_l] \mapsto [i_k \circ \eta] + [i_l \circ \eta] + [i_k, i_l] + [j_l]; \\
    	  \text{type IV} : \; & [j_k] \mapsto 8[i_k \circ \eta] + 2[i_l \circ \eta] + 4[i_k, i_l] + [j_k], \\
    	  & [j_l] \mapsto 4[i_k \circ \eta] + [i_l \circ \eta] + 2[i_k, i_l] + [j_l]; \\
    	  \text{type V} : \; & [j_k] \mapsto [i_k \circ \eta] + [i_l \circ \eta] + [i_h \circ \eta] + [i_k, i_l] + [i_k, i_h] + [i_l, i_h] + [j_k], \\
    	  & [j_l] \mapsto [i_k \circ \eta] + [i_l \circ \eta] + [i_h \circ \eta] + [i_k, i_l] + [i_k, i_h] + [i_l, i_h] + [j_l], \\
    	  & [j_h] \mapsto [i_k \circ \eta] + [i_l \circ \eta] + [i_h \circ \eta] + [i_k, i_l] + [i_k, i_h] + [i_l, i_h] + [j_h].
      \end{align*}
      For these maps the nonzero cup products of elements of $H^2({(V_g)}_f)$ are as follows:
      \begin{align*}
      	\text{type I} : \; & \overline{x}_k^2 = \overline{y}_k; \\
      	\text{type II} : \; & \overline{x}_k^2 = 8\overline{y}_k; \\
      	\text{type III} : \; & \overline{x}_k^2 = \overline{x}_l^2 = \overline{x}_k\overline{x}_l = \overline{y}_k + \overline{y}_l; \\
      	\text{type IV} : \; & \overline{x}_k^2 = 8\overline{y}_k + 4\overline{y}_l, \\
      	& \overline{x}_l^2 = 2\overline{y}_k + \overline{y}_l, \\
      	& \overline{x}_k\overline{x}_l = 4\overline{y}_k + 2\overline{y}_l; \\
      	\text{type V} : \; & \overline{x}_k^2 = \overline{x}_l^2 = \overline{x}_h^2 = \overline{x}_k\overline{x}_l = \overline{x}_k\overline{x}_h = \overline{x}_l\overline{x}_h = \overline{y}_k + \overline{y}_l + \overline{y}_h.
      \end{align*}
      
    \subsubsection{Determination of the cubic form}
    
      Let $f'$ be a diffeomorphism of $M_g$ acting trivially on homology, and let $f$ be a cellular map homotopic to $f'$. Then the mapping tori ${(M_g)}_{f'}$ and ${(M_g)}_f$ are homotopy equivalent relative to the subspace $M_g \times \{\frac{1}{2}\}$ (see for example \cite[Proposition 0.18]{Hat02}). Such a homotopy equivalence commutes with the inclusion maps $M_g = M_g \times \{\frac{1}{2}\} \hookrightarrow {(M_g)}_{f'}$ and $M_g = M_g \times \{\frac{1}{2}\} \hookrightarrow {(M_g)}_f$. We identify the (co)homology groups of ${(M_g)}_{f'}$ and ${(M_g)}_f$ via the induced isomorphisms of such a homotopy equivalence.
      
      Recall that $V_g$ denotes the 3-skeleton of $M_g$, and we still denote the restriction of $f$ on $V_g$ by $f : V_g \rightarrow V_g$. The mapping torus ${(V_g)}_f$ is a subcomplex of ${(M_g)}_f$. Using the Wang sequence for mapping tori \cite[Example 2.48]{Hat02}, we see that the inclusion map ${(V_g)}_f \hookrightarrow {(M_g)}_f$ induces isomorphisms on (co)homology groups of dimension up to 4, and moreover, it maps the homology class ${i_k}_*([S^2]) = [S^2_k] \in H_2({(V_g)}_f)$ to $i_*{S_k}_*([S^2]) \in H_2({(M_g)}_f)$ for each $k$. Identifying these (co)homology groups via the induced isomorphisms of the inclusion map, we obtain an isomorphism \[H^2({(M_g)}_{f'}) \rightarrow H^2({(V_g)}_f)\] mapping ${\nu_{f'}}^*(x_k)$ to $\overline{x}_k$ for each $k$, and an isomorphism \[H^4({(M_g)}_{f'}) \rightarrow H^4({(V_g)}_f)\] mapping the Poincar\'e dual of $i_*{S_k}_*([S^2])$ to $\overline{y}_k$ for each $k$. The effect of the latter isomorphism could be seen as follows. Note that it suffices to show that the dual isomorphism $H_4({(V_g)}_f) \rightarrow H_4({(M_g)}_{f'})$ maps $[e^4_k]$ to the Poincar\'e dual of ${\nu_{f'}}^*(x_k)$, which by \cite[Theorem II. 27]{Tho54} can be represented by the fundamental class of a 4-dimensional closed submanifold $N_k \subset {(M_g)}_{f'}$. The naturality of the Wang sequence yields a commutative diagram \[\begin{tikzcd}
      H_4({(V_g)}_f) \arrow[r, "\partial"] \arrow[d] & H_3(V_g) \arrow[d] \\
      H_4({(M_g)}_{f'}) \arrow[r, "\partial"] & H_3(M_g).
      \end{tikzcd}\] By the construction of Wang sequence we see that $\partial [e^4_k] = [S^3_k] \in H_3(V_g) = H_3(M_g)$. The homology class $\partial [N_k] \in H_3(M_g)$ is represented by the following 3-dimensional submanifold: there is a 4-dimensional submanifold of $M_g \times I$ with boundary, denoted by $\overline{N}_k$, which is mapped to $N_k$ by the quotient map $M_g \times I \rightarrow {(M_g)}_{f'}$. Since $N_k$ represents the Poincar\'e dual of ${\nu_{f'}}^*(x_k)$, the intersection number of $N_k$ and $S^3_k$ is 1, thus $N_k \cap (M_g \times \{t\})$ is nonempty for each $t$. Denote the boundary component $\overline{N}_k \cap (M_g \times \{0\})$ by $L_k$, which is a 3-dimensional closed submanifold of $M_g$. The fundamental class of $L_k$ is $\partial [N_k]$. The intersection number of $N_k$ and $S^3_k$ is 1 and the intersection number of $N_k$ and $S^3_l$ is zero for $l \neq k$, the same assertion holds for $L_k$ and $S^3_l$ ($1 \leq l \leq g$) considered as submanifolds of $M_g$. Thus $L_k$ is homologous to $S^3_k$, and by the commutative diagram we conclude that $[e^4_k]$ is mapped to $[N_k]$.
      
      These isomorphisms are induced by maps between spaces and thus preserve the cup product structure. Therefore the cubic form on $H^2({(M_g)}_f)$, where $f$ is a Dehn twist listed in Theorem \ref{thm:2}, is determined by Lemma \ref{lem:homotopyclass} and Lemma \ref{lem:cup} (more presicely, by the cup products of mapping tori of maps of type I - type V, which were listed below Lemma \ref{lem:cup}). We summarize the results as follows, where $[{(M_g)}_f]^*$ denotes the Kronecker dual of $[{(M_g)}_f]$, and $\overline{x}_k$ stands for ${\nu_f}^*(x_k)$.
      
      \begin{Lem}\label{lem:cubic}
      	The nonzero values of the cubic form on $H^2({(M_g)}_f)$, where $f$ is a Dehn twist listed in Theorem \ref{thm:2}, are as follows:
      	\begin{align*}
      		\mathbf{t}_{S_k, \alpha} : \; & \overline{x}_k^3 = [{(M_g)}_f]^*; \\
      		\mathbf{t}_{S_k \# S_k, \alpha} : \; & \overline{x}_k^3 = 8[{(M_g)}_f]^*; \\
      		\mathbf{t}_{S_k \# S_l, \alpha} : \; & \overline{x}_k^3 = \overline{x}_l^3 = \overline{x}_k^2\overline{x}_l = \overline{x}_k\overline{x}_l^2 = [{(M_g)}_f]^*; \\
      		\mathbf{t}_{S_k \# S_k \# S_l, \alpha} : \; & \overline{x}_k^3 = 8[{(M_g)}_f]^*, \\
      		& \overline{x}_l^3 = [{(M_g)}_f]^*, \\
      		& \overline{x}_k^2\overline{x}_l = 4[{(M_g)}_f]^*, \\
      		& \overline{x}_k\overline{x}_l^2 = 2[{(M_g)}_f]^*; \\
      		\mathbf{t}_{S_k \# S_l \# S_h, \alpha} : \; & \overline{x}_k^3 = \overline{x}_l^3 = \overline{x}_h^3 = \overline{x}_k^2\overline{x}_l = \overline{x}_k\overline{x}_l^2 = \overline{x}_k^2\overline{x}_h = \overline{x}_k\overline{x}_h^2 = \overline{x}_l^2\overline{x}_h = \overline{x}_l\overline{x}_h^2 = \overline{x}_k\overline{x}_l\overline{x}_h = [{(M_g)}_f]^*.
      	\end{align*}
      \end{Lem}

    \subsection{The group $\mathcal{I}(M_g)$}
    
      The bordism invariants of the mapping tori of the Dehn twists listed in Theorem \ref{thm:2} are determined by Lemma \ref{lem:pontryaginclass} and Lemma \ref{lem:cubic}. The result shows that the mapping torus constructions of these Dehn twists form a basis of ${\Omega}^{\mathrm{Spin}}_6((\mathbb{CP}^{\infty})^g)$, thus $\rho : \mathcal{I}(M_g) \rightarrow {\Omega}^{\mathrm{Spin}}_6((\mathbb{CP}^{\infty})^g)$ is an isomorphism and these Dehn twists form a basis of $\mathcal{I}(M_g)$. We list the values of the bordism invariants for $g \leq 3$ below, and summarize the results of this section into the following theorem. To simplify the notation, we write $\overline{x}_k^3$ for the evaluation of ${\nu_f}^*(x_k^3)$ on the fundamental class of the mapping torus (and the other notations should be interpreted similarly),  and $(\overline{x}_k^2\overline{x}_l)'$ stands for $\frac{1}{2}(\overline{x}_k^2\overline{x}_l - \overline{x}_k\overline{x}_l^2)$.
      \begin{table}[h]
      	\centering
      	\begin{tabular}{c|cc}
      		& $\hat{A}_1$ & $\overline{x}_1^3$ \\
      		\hline
      		$\mathbf{t}_{S_1, \alpha}$ & 0 & 1 \\
      		$\mathbf{t}_{S_1 \# S_1, \alpha}$ & 1 & 8
      	\end{tabular}
      \end{table}
      \begin{table}[h]
      	\centering
      	\begin{tabular}{c|cccccc}
      		& $\hat{A}_1$ & $\hat{A}_2$ & $\overline{x}_1^3$ & $\overline{x}_2^3$ & $\overline{x}_1\overline{x}_2^2$ & $(\overline{x}_1^2\overline{x}_2)'$\\
      		\hline
      		$\mathbf{t}_{S_1, \alpha}$ & 0 & 0 & 1 & 0 & 0 & 0 \\
      		$\mathbf{t}_{S_2, \alpha}$ & 0 & 0 & 0 & 1 & 0 & 0 \\
      		$\mathbf{t}_{S_1 \# S_1, \alpha}$ & 1 & 0 & 8 & 0 & 0 & 0 \\
      		$\mathbf{t}_{S_2 \# S_2, \alpha}$ & 0 & 1 & 0 & 8 & 0 & 0 \\
      		$\mathbf{t}_{S_1 \# S_2, \alpha}$ & 0 & 0 & 1 & 1 & 1 & 0 \\
      		$\mathbf{t}_{S_1 \# S_1 \# S_2, \alpha}$ & 1 & 0 & 8 & 1 & 2 & 1
      	\end{tabular}
      \end{table}
      \begin{table}[h]
      	\centering
      	\begin{tabular}{c|ccccccccccccc}
      		& $\hat{A}_1$ & $\hat{A}_2$ & $\hat{A}_3$ & $\overline{x}_1^3$ & $\overline{x}_2^3$ & $\overline{x}_3^3$ & $\overline{x}_1\overline{x}_2^2$ & $\overline{x}_1\overline{x}_3^2$ & $\overline{x}_2\overline{x}_3^2$ & $(\overline{x}_1^2\overline{x}_2)'$ & $(\overline{x}_1^2\overline{x}_3)'$ & $(\overline{x}_2^2\overline{x}_3)'$ & $\overline{x}_1\overline{x}_2\overline{x}_3$ \\
      		\hline
      		$\mathbf{t}_{S_1, \alpha}$ & 0 & 0 & 0 & 1 & 0 & 0 & 0 & 0 & 0 & 0 & 0 & 0 & 0 \\
      		$\mathbf{t}_{S_2, \alpha}$ & 0 & 0 & 0 & 0 & 1 & 0 & 0 & 0 & 0 & 0 & 0 & 0 & 0 \\
      		$\mathbf{t}_{S_3, \alpha}$ & 0 & 0 & 0 & 0 & 0 & 1 & 0 & 0 & 0 & 0 & 0 & 0 & 0 \\
      		$\mathbf{t}_{S_1 \# S_1, \alpha}$ & 1 & 0 & 0 & 8 & 0 & 0 & 0 & 0 & 0 & 0 & 0 & 0 & 0 \\
      		$\mathbf{t}_{S_2 \# S_2, \alpha}$ & 0 & 1 & 0 & 0 & 8 & 0 & 0 & 0 & 0 & 0 & 0 & 0 & 0 \\
      		$\mathbf{t}_{S_3 \# S_3, \alpha}$ & 0 & 0 & 1 & 0 & 0 & 8 & 0 & 0 & 0 & 0 & 0 & 0 & 0 \\
      		$\mathbf{t}_{S_1 \# S_2, \alpha}$ & 0 & 0 & 0 & 1 & 1 & 0 & 1 & 0 & 0 & 0 & 0 & 0 & 0 \\
      		$\mathbf{t}_{S_1 \# S_3, \alpha}$ & 0 & 0 & 0 & 1 & 0 & 1 & 0 & 1 & 0 & 0 & 0 & 0 & 0 \\
      		$\mathbf{t}_{S_2 \# S_3, \alpha}$ & 0 & 0 & 0 & 0 & 1 & 1 & 0 & 0 & 1 & 0 & 0 & 0 & 0 \\
      		$\mathbf{t}_{S_1 \# S_1 \# S_2, \alpha}$ & 1 & 0 & 0 & 8 & 1 & 0 & 2 & 0 & 0 & 1 & 0 & 0 & 0 \\
      		$\mathbf{t}_{S_1 \# S_1 \# S_3, \alpha}$ & 1 & 0 & 0 & 8 & 0 & 1 & 0 & 2 & 0 & 0 & 1 & 0 & 0 \\
      		$\mathbf{t}_{S_2 \# S_2 \# S_3, \alpha}$ & 0 & 1 & 0 & 0 & 8 & 1 & 0 & 0 & 2 & 0 & 0 & 1 & 0 \\
      		$\mathbf{t}_{S_1 \# S_2 \# S_3, \alpha}$ & 0 & 0 & 0 & 1 & 1 & 1 & 1 & 1 & 1 & 0 & 0 & 0 & 1
      	\end{tabular}
      \end{table}
    
      \begin{Thm}\label{thm:connectedsumofproduct}
      	For $g \geq 1$ the mapping torus construction $\rho : \mathcal{I}(M_g) \rightarrow {\Omega}^{\mathrm{Spin}}_6((\mathbb{CP}^{\infty})^g)$ is an isomorphism, thus $\mathcal{I}(M_g)$ is isomorphic to $\mathbb{Z}^{2g + 2{\binom{g}{2}} + {\binom{g}{3}}}$. Moreover:
      	\begin{itemize}
      		\item[(1)] A basis of $\mathcal{I}(M_g)$ is given by the following Dehn twists:
      		\begin{align*}
      			& \mathbf{t}_{S_k, \alpha}, & 1 \leq k \leq g; \\
      			& \mathbf{t}_{S_k \# S_k, \alpha}, & 1 \leq k \leq g; \\
      			& \mathbf{t}_{S_k \# S_l, \alpha}, & 1 \leq k < l \leq g; \\
      			& \mathbf{t}_{S_k \# S_k \# S_l, \alpha}, & 1 \leq k < l \leq g; \\
      			& \mathbf{t}_{S_k \# S_l \# S_h, \alpha}, & 1 \leq k < l < h \leq g.
      		\end{align*}
      		
      		\item[(2)] A ``standard" basis of $\mathcal{I}(M_g)$, mapping to the standard basis of $\mathbb{Z}^{2g + 2{\binom{g}{2}} + {\binom{g}{3}}}$ via the composition of $\rho$ and the bordism invariants described below Lemma \ref{lem:homologyeffect}, is given by the following diffeomorphisms:
      		\begin{align*}
      			& \mathbf{t}_{S_k \# S_k, \alpha} - 8\mathbf{t}_{S_k, \alpha}, & 1 \leq k \leq g; \\
      			& \mathbf{t}_{S_k, \alpha}, & 1 \leq k \leq g; \\
      			& \mathbf{t}_{S_k \# S_l, \alpha} - (\mathbf{t}_{S_k, \alpha} + \mathbf{t}_{S_l, \alpha}), & 1 \leq k < l \leq g; \\
      			& \mathbf{t}_{S_k \# S_k \# S_l, \alpha} + 2\mathbf{t}_{S_k, \alpha} + \mathbf{t}_{S_l, \alpha} - (\mathbf{t}_{S_k \# S_k, \alpha} + 2\mathbf{t}_{S_k \# S_l, \alpha}), & 1 \leq k < l \leq g; \\
      			& \mathbf{t}_{S_k \# S_l \# S_h, \alpha} + \mathbf{t}_{S_k, \alpha} + \mathbf{t}_{S_l, \alpha} + \mathbf{t}_{S_h, \alpha} - (\mathbf{t}_{S_k \# S_l, \alpha} + \mathbf{t}_{S_k \# S_h, \alpha} + \mathbf{t}_{S_l \# S_h, \alpha}), & 1 \leq k < l < h \leq g.
      		\end{align*}
      	\end{itemize}
      \end{Thm}
      
      The action of $\mathcal{I}(M_g)$ on $\pi_3(M_g)$ is determined since the actions of the Dehn twists, which form a basis of $\mathcal{I}(M_g)$, are determined in Lemma \ref{lem:homotopyclass}. For instance we describe the action of $\mathcal{I}(S^2 \times S^3) = \mathbb{Z}[\mathbf{t}_{S_1, \alpha}] \oplus \mathbb{Z}[\mathbf{t}_{S_1 \# S_1, \alpha} - 8\mathbf{t}_{S_1, \alpha}]$ on $\pi_3(S^2 \times S^3) = \mathbb{Z}[S_1 \circ \eta] \oplus \mathbb{Z}[T_1]$. The direct summand $\mathbb{Z}[\mathbf{t}_{S_1 \# S_1, \alpha} - 8\mathbf{t}_{S_1, \alpha}]$ acts trivially on $\pi_3(S^2 \times S^3)$, and the direct summand $\mathbb{Z}[\mathbf{t}_{S_1, \alpha}]$ acts on $\pi_3(S^2 \times S^3)$ by \[{(n\mathbf{t}_{S_1, \alpha})}_*[S_1 \circ \eta] = [S_1 \circ \eta], \quad {(n\mathbf{t}_{S_1, \alpha})}_*[T_1] = n[S_1 \circ \eta] + [T_1].\] The result for general $g$ is more complicated, so we omit the explicit description. However we note the following.
      
      \begin{Cor}
      	Let $\overline{K}(M_g)$ be the rank $g$ free abelian group generated by $\mathbf{t}_{S_k \# S_k, \alpha} - 8\mathbf{t}_{S_k, \alpha}$, where $1 \leq k \leq g$. Then $\overline{K}(M_g)$ is the subgroup of $\mathcal{I}(M_g)$ consisting of all diffeomorphisms that act trivially on $\pi_3(M_g)$.
      \end{Cor}
      
      \begin{proof}
      	Clearly $\overline{K}(M_g)$ acts trivially on $\pi_3(M_g)$. Now suppose $f$ is a diffeomorphism acting trivially on $H_*(M_g)$ and $\pi_3(M_g)$. By Lemma \ref{lem:cup} the cubic form on $H^2({(M_g)}_f)$ is zero, thus ${\nu_f}_*([{(M_g)}_f]) = 0$ and $f$ lies in $\overline{K}(M_g)$.
      \end{proof}

    \section{Torelli groups of spin 5-manifolds}\label{sec:5}
    
      Throughout this section $M$ will denote a simply connected closed spin 5-manifold unless otherwise stated. We compute the Torelli group of $M$ when $H_2(M)$ has no 2-torsion, determine the generators, and prove Theorem \ref{thm:1}, using Kreck's modified surgery theory.
    
    \subsection{The mapping class group of $({\natural}^g(S^2 \times D^3), \partial)$}
    
      As the first step we compute the relative mapping class group $\mathrm{MCG}({\natural}^g(S^2 \times D^3), \partial)$ of the $g$-fold boundary connected sum of $S^2 \times D^3$.
      
      Note that $M_g = {\natural}^g(S^2 \times D^3) \cup_{id} \overline{{\natural}^g(S^2 \times D^3)}$ is a twisted double, thus there is a well-defined homomorphism \[\mathrm{MCG}({\natural}^g(S^2 \times D^3), \partial) \rightarrow \mathcal{I}(M_g)\] given by extending a diffeomorphism of ${\natural}^g(S^2 \times D^3)$ by the identity map on the second copy of ${\natural}^g(S^2 \times D^3)$. This homomorphism is surjective, since by Theorem \ref{thm:connectedsumofproduct} $\mathcal{I}(M_g)$ is generated by Dehn twists about certain embeddings of $S^2$ which, up to isotopy, are supported in ${\natural}^g(S^2 \times D^3)$. This proves the following.
      
      \begin{Lem}\label{lem:firstsurjection}
      	Extending a diffeomorphism by identity defines a surjective homomorphism \[\mathrm{MCG}({\natural}^g(S^2 \times D^3), \partial) \rightarrow \mathcal{I}(M_g).\]
      \end{Lem}
      
      We shall now construct a surjective homomorphism ${\Omega}^{\mathrm{Spin}}_6((\mathbb{CP}^{\infty})^g) \rightarrow \mathrm{MCG}({\natural}^g(S^2 \times D^3), \partial)$ as follows. The normal 2-type of ${\natural}^g(S^2 \times D^3)$ is the same as that of $M_g$, which is \[B \coloneqq {({\mathbb{CP}}^{\infty})}^g \times B\mathrm{Spin} \rightarrow B\mathrm{Spin} \rightarrow B\mathrm{SO},\] and a normal 2-smoothing $\overline{\nu} : {\natural}^g(S^2 \times D^3) \rightarrow B$ is given by the restriction of the normal 2-smoothing of $M_g$ defined in Section \ref{sec:2}. This defines a 6-dimensional normal $B$-bordism class $[{\natural}^g(S^2 \times D^3) \times I, \overline{\nu} \times I]$. Let $[N, \nu] \in {\Omega}^{\mathrm{Spin}}_6((\mathbb{CP}^{\infty})^g)$, then there is a 6-dimensional normal $B$-bordism class \[[W, \nu] \coloneqq [{\natural}^g(S^2 \times D^3) \times I, \overline{\nu} \times I] + [N, \nu]\] which is a normal $B$-bordism between $({\natural}^g(S^2 \times D^3), \overline{\nu})$ and itself.
      
      \begin{Lem}\label{lem:surgery}
      	The bordism class $[W, \nu]$ can be represented by an $h$-cobordism of $({\natural}^g(S^2 \times D^3), \overline{\nu})$ to itself, thus yields an orientation preserving diffeomorphism of ${\natural}^g(S^2 \times D^3)$. This defines a well-defined map \[\Phi : {\Omega}^{\mathrm{Spin}}_6((\mathbb{CP}^{\infty})^g) \rightarrow \mathrm{MCG}({\natural}^g(S^2 \times D^3), \partial).\]
      \end{Lem}
      
      \begin{proof}
      	First we show that the normal $B$-manifold $(W, \nu) = ({\natural}^g(S^2 \times D^3) \times I, \overline{\nu} \times I) \sqcup (N, \nu)$ is normal $B$-bordant, relative to boundary, to an $h$-cobordism of $({\natural}^g(S^2 \times D^3), \overline{\nu})$ to itself. By \cite[Theorem 3]{Kre99}, this is true if and only if the modified surgery obstruction $\theta(W, \nu)$ is elementary (see \cite[Page 730]{Kre99} for definition), which lies in the monoid $l_6^{\sim}(e)$ since the image of the universal Stiefel\textendash Whitney class $w_4 \in H^4(B\mathrm{SO}; \mathbb{Z}/2)$ in $H^4(B\mathrm{Spin}; \mathbb{Z}/2)$ is nonzero by \cite{Tho62}. We show that $\theta(W, \nu)$ vanishes, thus is elementary. In fact, by \cite[Theorem 5.2 (b)]{Kre}, $\theta(W, \nu)$ lies in a certain abelian group $L_6^{s, \tau}(e, \mathbb{Z}e)$. By \cite[Lemma 4.1]{Kre} and the fact that the Whitehead group of the trivial group is trivial, there is a surjective homomorphism $L_6(\mathbb{Z}) \rightarrow L_6^{s, \tau}(e, \mathbb{Z}e)$, so $L_6^{s, \tau}(e, \mathbb{Z}e)$ is isomorphic to $0$ or $\mathbb{Z}/2$. If $L_6^{s, \tau}(e, \mathbb{Z}e)$ is nontrivial, then it is isomorphic to $L_6(\mathbb{Z})$ and the nontrivial element is detected by the Arf\textendash Kervaire invariant \cite{Wal99}. Since $S^3 \times S^3$ admits a framing with nontrivial Kervaire invariant, we may form a connected sum of $W$ and $S^3 \times S^3$ (the trace of this connected sum operation is a normal $B$-bordism relative to boundary) which has vanishing modified surgery obstruction, if $\theta(W, \nu) \neq 0$. Since $\theta(W, \nu)$ is a bordism invariant it must vanish.
      	
      	Therefore $(W, \nu)$ is normal $B$-bordant, relative to boundary, to an $h$-cobordism of $({\natural}^g(S^2 \times D^3), \overline{\nu})$ to itself. Such an $h$-cobordism yields an orientation preserving diffeomorphism of ${\natural}^g(S^2 \times D^3)$ which restricts to the identity on the boundary. This diffeomorphism defines an element of $\mathrm{MCG}({\natural}^g(S^2 \times D^3), \partial)$. To show that this defines a well-defined map from ${\Omega}^{\mathrm{Spin}}_6((\mathbb{CP}^{\infty})^g)$ to $\mathrm{MCG}({\natural}^g(S^2 \times D^3), \partial)$ we need to show that the diffeomorphisms given by different representatives of the same bordism class in ${\Omega}^{\mathrm{Spin}}_6((\mathbb{CP}^{\infty})^g)$ are isotopic relative to boundary.
      	
      	Let $(N, \nu)$ and $(N', \nu')$ be normal $B$-manifolds such that $[N, \nu] = [N', \nu'] \in {\Omega}^{\mathrm{Spin}}_6((\mathbb{CP}^{\infty})^g)$. Choose a normal $B$-bordism between $(N, \nu)$ and $(N', \nu')$, the disjoint union of this bordism with $({\natural}^g(S^2 \times D^3) \times I \times I, \overline{\nu} \times I \times I)$ form a normal $B$-bordism between $(W, \nu) = ({\natural}^g(S^2 \times D^3) \times I, \overline{\nu} \times I) \sqcup (N, \nu)$ and $(W', \nu') = ({\natural}^g(S^2 \times D^3) \times I, \overline{\nu} \times I) \sqcup (N', \nu')$. The corresponding $h$-cobordisms are thus normal $B$-bordant and we denote the normal $B$-bordism between them by $(V, \tilde{\nu})$. After smoothing the corners we may view $V$ as a compact 7-manifold with boundary given by gluing the two $h$-cobordisms (which are both homotopy equivalent to ${\natural}^g(S^2 \times D^3)$) along the boundary via the identity map, and by \cite[Corollary 1]{Kre99} we may assume that the normal $B$-structures of these $h$-cobordisms are 3-connected. By \cite[Theorem 5]{Kre99} we can assume that the modified surgery obstruction $\theta(V, \tilde{\nu}) \in l_7^{\sim}(e)$ is elementary, thus by \cite[Theorem 4]{Kre99} we can assume that $(V, \tilde{\nu})$ is an $h$-cobordism relative to boundary. The $h$-cobordism theorem then implies that the diffeomorphisms given by $(N, \nu)$ and $(N', \nu')$ are pseudoisotopic relative to boundary, and hence are isotopic relative to boundary by Cerf's pseudoisotopy theorem. Therefore our construction yields a well-defined map $\Phi : {\Omega}^{\mathrm{Spin}}_6((\mathbb{CP}^{\infty})^g) \rightarrow \mathrm{MCG}({\natural}^g(S^2 \times D^3), \partial)$.
      \end{proof}
      
      \begin{Rmk}
      	\rm In the proof of Lemma \ref{lem:surgery} we used the unpublished work of Kreck \cite{Kre} to show that the surgery obstruction $\theta(W, \nu)$ vanishes. In fact it suffices to show that $\theta(W, \nu)$ is elementary, and this can be proved directly without using \cite{Kre}. We present such a proof in Appendix \ref{a:obstruction}.
      \end{Rmk}
      
      \begin{Lem}\label{lem:secondsurjection}
      	The map $\Phi : {\Omega}^{\mathrm{Spin}}_6((\mathbb{CP}^{\infty})^g) \rightarrow \mathrm{MCG}({\natural}^g(S^2 \times D^3), \partial)$ is a surjective homomorphism.
      \end{Lem}
      
      \begin{proof}
      	That the map $\Phi$ is a homomorphism follows from the observation that \begin{align*}
      		& [{\natural}^g(S^2 \times D^3) \times I, \overline{\nu} \times I] + [N, \nu] + [N', \nu'] \\
      		& = \Bigl[\bigl(({\natural}^g(S^2 \times D^3) \times [0, 1], \overline{\nu} \times [0, 1]) + (N, \nu)\bigr) \\
      		& \quad \bigcup_{{\natural}^g(S^2 \times D^3) \times \{1\}} \bigl(({\natural}^g(S^2 \times D^3) \times [1, 2], \overline{\nu} \times [1, 2]) + (N', \nu')\bigr)\Bigr]
      	\end{align*} is normal $B$-bordant, relative to boundary, to the result of gluing the two $h$-cobordisms along ${\natural}^g(S^2 \times D^3) \times \{1\}$ via the identity map. The latter $h$-cobordism yields the composition of the diffeomorphisms given by $[N, \nu]$ and $[N', \nu']$.
      	
      	Now we show that $\Phi$ is surjective. Let $f$ be a diffeomorphism representing an element of $\mathrm{MCG}({\natural}^g(S^2 \times D^3), \partial)$. By obstruction theory there is a map ${\natural}^g(S^2 \times D^3) \times I \rightarrow B\mathrm{SO}$ extending \[({\natural}^g(S^2 \times D^3) \times \{0\}) \cup (\partial ({\natural}^g(S^2 \times D^3)) \times I) \cup ({\natural}^g(S^2 \times D^3) \times \{1\}) \xrightarrow{\nu \cup (\nu \times I) \cup (\nu \circ f)} B\mathrm{SO}.\] This is a normal Gauss map because it is homotopic to $\nu \times I : {\natural}^g(S^2 \times D^3) \times I \rightarrow B\mathrm{SO}$. Since ${\natural}^g(S^2 \times D^3) \times I$ has a unique spin structure, there is a map ${\natural}^g(S^2 \times D^3) \times I \rightarrow B\mathrm{Spin}$ lifting this normal Gauss map. Again by obstruction theory the map \[({\natural}^g(S^2 \times D^3) \times \{0\}) \cup (\partial ({\natural}^g(S^2 \times D^3)) \times I) \cup ({\natural}^g(S^2 \times D^3) \times \{1\}) \xrightarrow{\overline{\nu} \cup (\overline{\nu} \times I) \cup (\overline{\nu} \circ f)} {(\mathbb{CP}^{\infty})}^g\] extends to a map ${\natural}^g(S^2 \times D^3) \times I \rightarrow {(\mathbb{CP}^{\infty})}^g$. These maps assemble into a normal $B$-structure $H : {\natural}^g(S^2 \times D^3) \times I \rightarrow B$ extending \[({\natural}^g(S^2 \times D^3) \times \{0\}) \cup (\partial ({\natural}^g(S^2 \times D^3)) \times I) \cup ({\natural}^g(S^2 \times D^3) \times \{1\}) \xrightarrow{\overline{\nu} \cup (\overline{\nu} \times I) \cup (\overline{\nu} \circ f)} B.\] Extending $f$ by identity yields a diffeomorphism of $M_g$ which we still denote by $f$. The normal $B$-structure $H$ extends by $\overline{\nu} \times I$ to a map $M_g \times I \rightarrow B$, which induces a normal $B$-structure $\tilde{H} : {(M_g)}_f \rightarrow B$. This defines a normal $B$-bordism class $[{(M_g)}_f, \tilde{H}] \in {\Omega}^{\mathrm{Spin}}_6((\mathbb{CP}^{\infty})^g)$ (which equals to $\rho([f]) = [{(M_g)}_f, \nu_f]$ since the spin structures and the effects on the second homology groups of these two normal $B$-structures coincide). Now the normal $B$-manifold \begin{align*}
      		& \Bigl({\natural}^g(S^2 \times D^3) \times [0, 1] \times I, \overline{\nu} \times I \times I\Bigr) \\
      		& \bigcup_F \Bigl({\natural}^g(S^2 \times D^3) \times [1, 2] \times I, \bigl(\overline{\nu} \times I \times I\bigr) \cup \bigl(H \times I\bigr) \cup \bigl((\overline{\nu} \circ f) \times I \times I\bigr)\Bigr) \\
      		& \bigcup_{id} \Bigl({\natural}^g(S^2 \times D^3) \times D^2, \overline{\nu} \times D^2\Bigr),
      	\end{align*} where the first gluing map $F$ is the disjoint union of \begin{align*}
      	f \times id : & \; {\natural}^g(S^2 \times D^3) \times \{1\} \times [0, 1/3] \rightarrow {\natural}^g(S^2 \times D^3) \times \{1\} \times [0, 1/3], \\
      	id : & \; {\natural}^g(S^2 \times D^3) \times \{1\} \times [2/3, 1] \rightarrow {\natural}^g(S^2 \times D^3) \times \{1\} \times [2/3, 1]
      	\end{align*} and the second gluing map is \[id : \partial ({\natural}^g(S^2 \times D^3)) \times S^1 \rightarrow \partial ({\natural}^g(S^2 \times D^3)) \times S^1,\] is a normal $B$-coboundary between \[\Bigl({\natural}^g(S^2 \times D^3) \times I, \overline{\nu} \times I\Bigr) + \Bigl({(M_g)}_f, \tilde{H}\Bigr)\] and \[\Bigl(\bigl({\natural}^g(S^2 \times D^3) \times I\bigr) \cup_f \bigl({\natural}^g(S^2 \times D^3) \times I\bigr), (\overline{\nu} \times I) \cup \bigl((\overline{\nu} \circ f) \times I\bigr)\Bigr).\] Thus $\Phi([{(M_g)}_f, \tilde{H}]) = [f]$. Hence $\Phi$ is surjective.
      \end{proof}
      
      By Lemma \ref{lem:firstsurjection} and Lemma \ref{lem:secondsurjection} the composition \[{\Omega}^{\mathrm{Spin}}_6((\mathbb{CP}^{\infty})^g) \xrightarrow{\Phi} \mathrm{MCG}({\natural}^g(S^2 \times D^3), \partial) \rightarrow \mathcal{I}(M_g)\] is a surjective homomorphism \[\mathbb{Z}^{2g + 2{\binom{g}{2}} + {\binom{g}{3}}} \rightarrow \mathbb{Z}^{2g + 2{\binom{g}{2}} + {\binom{g}{3}}},\] which is necessarily bijective. This proves the following.
      
      \begin{Thm}\label{thm:boundaryconnectedsum}
      	Extending a diffeomorphism by identity induces an isomorphism \[\mathrm{MCG}({\natural}^g(S^2 \times D^3), \partial) \xrightarrow{\cong} \mathcal{I}(M_g).\]
      \end{Thm}

    \subsection{Extending diffeomorphisms by identity}
    
      According to Barden's construction \cite[Theorem 2.3]{Bar65}, a simply connected closed spin 5-manifold $M$ is a twisted double of a boundary connected sum of $S^2 \times D^3$. Precisely, we may assume \[H_2(M) \cong \mathbb{Z}^g \oplus (\mathbb{Z}/{{p_1}^{r_1}})^2 \oplus \cdots \oplus (\mathbb{Z}/{{p_s}^{r_s}})^2,\] where $p_i$ are prime numbers, then $M$ is a twisted double of ${\natural}^{g+2s}(S^2 \times D^3)$. This implies that there is a well-defined homomorphism \[\mathrm{MCG}({\natural}^{g+2s}(S^2 \times D^3), \partial) \rightarrow \mathcal{I}(M)\] given by extending a diffeomorphism by identity.
      
      From now on $M$ is assumed to contain no 2-torsion in homology. By Corollary \ref{cor:forgetful} the mapping torus construction $\rho$ can be viewed as an injective homomorphism from $\mathcal{I}(M)$ to the bordism group ${\Omega}^{\mathrm{Spin}}_6(K(H_2(M), 2))$. We identify $H_2(M)$ with $\mathbb{Z}^g \oplus (\mathbb{Z}/{{p_1}^{r_1}})^2 \oplus \cdots \oplus (\mathbb{Z}/{{p_s}^{r_s}})^2$ via the standard basis given by the fundamental classes of the embedded spheres $S_k(S^2)$, where $S_k : S^2 = S^2 \times \{*\} \hookrightarrow {\natural}^{g+2s}(S^2 \times D^3)$ is the standard embedding in the $k$-th boundary connected summand. So the canonical quotient homomorphism $\delta : \mathbb{Z}^{g+2s} \rightarrow H_2(M)$ is well-defined.
      
      The main result of this section is the following.
    
      \begin{Thm}\label{thm:torelli}
    	  The homomorphism $\mathrm{MCG}({\natural}^{g+2s}(S^2 \times D^3), \partial) \rightarrow \mathcal{I}(M)$ given by extending diffeomorphisms by identity fits into a commutative diagram \[\begin{tikzcd}
    		  \mathrm{MCG}({\natural}^{g+2s}(S^2 \times D^3), \partial) \arrow[d, "\cong"] \arrow[dr] & \\
    		  \mathcal{I}(M_{g+2s}) \arrow[d, "\cong", "\rho"'] & \mathcal{I}(M) \arrow[d, "\rho"'] \\
    		  {\Omega}^{\mathrm{Spin}}_6((\mathbb{CP}^{\infty})^{g+2s}) \arrow[r, "\delta_*"] & {\Omega}^{\mathrm{Spin}}_6(K(H_2(M), 2)).
    	  \end{tikzcd}\]
    	  It follows that the mapping torus construction \[\rho : \mathcal{I}(M) \rightarrow {\Omega}^{\mathrm{Spin}}_6(K(H_2(M), 2))\] is an isomorphism, and the homomorphism $\mathrm{MCG}({\natural}^{g+2s}(S^2 \times D^3), \partial) \rightarrow \mathcal{I}(M)$ is surjective.
      \end{Thm}
      
      \begin{proof}
      	By Proposition \ref{prop:quotient} it suffices to show that the diagram commutes. Namely, let $f$ be a diffeomorphism representing an element of $\mathrm{MCG}({\natural}^{g+2s}(S^2 \times D^3), \partial)$, still denote its image in $\mathcal{I}(M_{g+2s})$ by $f$ and denote its image in $\mathcal{I}(M)$ by $f'$, and denote their mapping torus constructions by $[{(M_{g+2s})}_f, \nu_f]$ and $[M_{f'}, \nu_{f'}]$ respectively, we will show that \[[{(M_{g+2s})}_f, \delta \circ \nu_f] = [M_{f'}, \nu_{f'}] \in {\Omega}^{\mathrm{Spin}}_6(K(H_2(M), 2)).\] It follows from the proof of Lemma \ref{lem:secondsurjection} that \begin{align*}
      		& \Bigl[{\natural}^{g+2s}(S^2 \times D^3) \times I, (\delta \circ \overline{\nu}) \times I\Bigr] + \Bigl[{(M_{g+2s})}_f, \delta \circ \nu_f\Bigr] \\
      		& = \Bigl[\bigl({\natural}^{g+2s}(S^2 \times D^3) \times I\bigr) \cup_f \bigl({\natural}^{g+2s}(S^2 \times D^3) \times I\bigr), \bigl((\delta \circ \overline{\nu}) \times I\bigr) \cup \bigl((\delta \circ \overline{\nu} \circ f) \times I\bigr)\Bigr]
      	\end{align*} so it suffices to show that \begin{align*}
      	& \Bigl[{\natural}^{g+2s}(S^2 \times D^3) \times I, (\delta \circ \overline{\nu}) \times I\Bigr] + \Bigl[M_{f'}, \nu_{f'}\Bigr] \\
      	& = \Bigl[\bigl({\natural}^{g+2s}(S^2 \times D^3) \times I\bigr) \cup_f \bigl({\natural}^{g+2s}(S^2 \times D^3) \times I\bigr), \bigl((\delta \circ \overline{\nu}) \times I\bigr) \cup \bigl((\delta \circ \overline{\nu} \circ f) \times I\bigr)\Bigr]
      	\end{align*} since we can glue the two normal $B$-bordisms along their common boundary to obtain a normal $B$-bordism between $[{(M_{g+2s})}_f, \delta \circ \nu_f]$ and $[M_{f'}, \nu_{f'}]$, here $B$ denotes $K(H_2(M), 2) \times B\mathrm{Spin}$.
      	
      	Similar to the proof of Lemma \ref{lem:secondsurjection} we construct a normal $B$-bordism explicitly. Denote the gluing diffeomorphism of the twisted double $M$ by $f_M : {\#}^{g+2s}(S^2 \times S^2) \rightarrow {\#}^{g+2s}(S^2 \times S^2)$. Then we have \[M_{f'} = {({\natural}^{g+2s}(S^2 \times D^3))}_f \cup_{f_M \times S^1} ({\natural}^{g+2s}(S^2 \times D^3) \times S^1).\] We first construct a normal $B$-structure $\tilde{\nu} : {\natural}^{g+2s}(S^2 \times D^3) \times D^2 \rightarrow B$ whose restriction on $\partial ({\natural}^{g+2s}(S^2 \times D^3)) \times \partial D^2$ is $\nu_{f'} \circ (f_M \times S^1)$. Note that the maps $\nu_{f'} : {\#}^{g+2s}(S^2 \times S^2) \rightarrow B$ and $\nu_{f'} \circ f_M : {\#}^{g+2s}(S^2 \times S^2) \rightarrow B$ are homotopic via a homotopy lifting the normal Gauss map of ${\#}^{g+2s}(S^2 \times S^2) \times I$, since the spin structures and the effects on the second homology groups of the two maps coincide (the effect of $f_M$ on homology is given by \cite[Section 1]{Bar65}). We choose such a homotopy and use it to define a normal $B$-structure on $\partial({\natural}^{g+2s}(S^2 \times D^3)) \times \partial D^2 \times I \times I$, a product neighborhood of the corner of ${\natural}^{g+2s}(S^2 \times D^3) \times D^2$, which restricts to $\nu_{f'} \circ f_M$ on $\partial({\natural}^{g+2s}(S^2 \times D^3)) \times \partial D^2$ and restricts to $\nu_{f'} \times \partial D^2 \times (\{1\} \times I \cup I \times \{1\})$ on $\partial({\natural}^{g+2s}(S^2 \times D^3)) \times \partial D^2 \times (\{1\} \times I \cup I \times \{1\})$. We extend this normal $B$-structure to the entire ${\natural}^{g+2s}(S^2 \times D^3) \times D^2$ by $\nu_{f'} \times D^2$. Let $\tilde{\nu} : {\natural}^{g+2s}(S^2 \times D^3) \times D^2 \rightarrow B$ be the resulting normal $B$-structure. As in the proof of Lemma \ref{lem:secondsurjection} we have a normal $B$-structure $H' : {\natural}^{g+2s}(S^2 \times D^3) \times I \rightarrow B$ extending the map $\nu_{f'} \cup (\nu_{f'} \times I) \cup (\nu_{f'} \circ f)$ on the subspace \[({\natural}^{g+2s}(S^2 \times D^3) \times \{0\}) \cup (\partial ({\natural}^{g+2s}(S^2 \times D^3)) \times I) \cup ({\natural}^{g+2s}(S^2 \times D^3) \times \{1\}).\] The normal $B$-bordism class $[M_{f'}, \tilde{H}']$, where $\tilde{H}'$ is given by extending $H'$ to $M_{f'}$ by $\nu_{f'} \times S^1$, is equal to $[M_{f'}, \nu_{f'}]$. The normal $B$-manifold \begin{align*}
      		& \Bigl({\natural}^{g+2s}(S^2 \times D^3) \times [0, 1] \times I, \nu_{f'} \times I \times I\Bigr) \\
      		& \bigcup_{F'} \Bigl({\natural}^{g+2s}(S^2 \times D^3) \times [1, 2] \times I, \bigl(\nu_{f'} \times I \times I\bigr) \cup \bigl(H' \times I\bigr) \cup \bigl((\nu_{f'} \circ f) \times I \times I\bigr)\Bigr) \\
      		& \bigcup_{f_M \times S^1} \Bigl({\natural}^{g+2s}(S^2 \times D^3) \times D^2, \tilde{\nu}\Bigr),
      	\end{align*} where the first gluing map $F'$ is the disjoint union of \begin{align*}
      	f \times id : & \; {\natural}^g(S^2 \times D^3) \times \{1\} \times [0, 1/3] \rightarrow {\natural}^g(S^2 \times D^3) \times \{1\} \times [0, 1/3], \\
      	id : & \; {\natural}^g(S^2 \times D^3) \times \{1\} \times [2/3, 1] \rightarrow {\natural}^g(S^2 \times D^3) \times \{1\} \times [2/3, 1],
      	\end{align*} is a normal $B$-coboundary between \[\Bigl({\natural}^{g+2s}(S^2 \times D^3) \times I, \nu_{f'} \times I\Bigr) + \Bigl(M_{f'}, \tilde{H}'\Bigr)\] and \[\Bigl(\bigl({\natural}^{g+2s}(S^2 \times D^3) \times I\bigr) \cup_f \bigl({\natural}^{g+2s}(S^2 \times D^3) \times I\bigr), (\nu_{f'} \times I) \cup \bigl((\nu_{f'} \circ f) \times I\bigr)\Bigr).\] The proof is completed by noticing that the maps $\delta \circ \overline{\nu} : {\natural}^{g+2s}(S^2 \times D^3) \rightarrow B$ and $\nu_{f'} : {\natural}^{g+2s}(S^2 \times D^3) \rightarrow B$ are homotopic via a homotopy lifting the normal Gauss map, since the spin structures and the effects on the second homology groups of the two maps coincide.
      \end{proof}

    \subsection{Generators of $\mathcal{I}(M)$}\label{5.3}
    
      It follows from Theorem \ref{thm:torelli} that, for a simply connected closed spin 5-manifold $M$ with no 2-torsion in homology, $\mathcal{I}(M)$ is generated by certain Dehn twists. More precisely, suppose $H_2(M)$ is isomorphic to \[\mathbb{Z}^g \oplus \left({(\mathbb{Z}/{p_1^{r_{1,1}}})}^2 \oplus \cdots \oplus {(\mathbb{Z}/{p_1^{r_{1,s_1}}})}^2\right) \oplus \cdots \oplus \left({(\mathbb{Z}/{p_t^{r_{t,1}}})}^2 \oplus \cdots \oplus {(\mathbb{Z}/{p_t^{r_{t,s_t}}})}^2\right), \tag{3}\label{isotype}\] where $g$, $t$, $s_1, \cdots, s_t$ are non-negative integers such that $g + 2(s_1 + \cdots + s_t) > 0$; $p_1, \cdots, p_t$ are prime numbers such that $3 \leq p_1 < \cdots < p_t$; and $r_{i,1}, \cdots, r_{i,s_i}$ are positive integers such that $r_{i,1} \leq \cdots \leq r_{i,s_i}$ for each $i$. Then $M$ is a twisted double of ${\natural}^{g+2(s_1 + \cdots + s_t)}(S^2 \times D^3)$. For $1 \leq k \leq g+2(s_1 + \cdots + s_t)$ let $S_k : S^2 \hookrightarrow {\natural}^{g+2(s_1 + \cdots + s_t)}(S^2 \times D^3) \hookrightarrow M$ be the standard embedding in the first copy of ${\natural}^{g+2(s_1 + \cdots + s_t)}(S^2 \times D^3)$. Then the Dehn twists \[\mathbf{t}_{S_k, \alpha}, \quad \mathbf{t}_{S_k \# S_k, \alpha}, \quad \mathbf{t}_{S_k \# S_l, \alpha}, \quad \mathbf{t}_{S_k \# S_k \# S_l, \alpha}, \quad \mathbf{t}_{S_k \# S_l \# S_h, \alpha}\] generate $\mathcal{I}(M)$. Another set of generators are given by diffeomorphisms of the form listed in the part (2) of Theorem \ref{thm:connectedsumofproduct}, and we denote them by \[\mathbf{f}_{k,k}, \quad \mathbf{f}_k, \quad \mathbf{f}_{k,l}, \quad \mathbf{f}_{k,k,l}, \quad \mathbf{f}_{k,l,h}\] for brevity. Here the correspondence is via the subscript, for instance, $\mathbf{f}_{k,k}$ stands for the diffeomorphism $\mathbf{t}_{S_k \# S_k, \alpha} - 8\mathbf{t}_{S_k, \alpha}$.
      
      The group $H_M$ defined by (\ref{fundamentalclass}) can be calculated using K\"unneth theorem. When $p_1 \geq 5$ it is isomorphic to \begin{align*}
      	\mathbb{Z}^{g + 2{\binom{g}{2}} + {\binom{g}{3}}} & \oplus {(\mathbb{Z}/{p_1^{r_{1,1}}})}^{4s_1^2 + 4g(s_1-1) + 5g + 2{\binom{g}{2}}} \oplus {(\mathbb{Z}/{p_1^{r_{1,2}}})}^{4(s_1-1)^2 + 4g(s_1-2) + 5g + 2{\binom{g}{2}}} \\
      	& \oplus \cdots \oplus {(\mathbb{Z}/{p_1^{r_{1,s_1-1}}})}^{16 + 9g + 2{\binom{g}{2}}} \oplus {(\mathbb{Z}/{p_1^{r_{1,s_1}}})}^{4 + 5g + 2{\binom{g}{2}}} \\
      	& \oplus {(\mathbb{Z}/{p_t^{r_{t,1}}})}^{4s_t^2 + 4g(s_t-1) + 5g + 2{\binom{g}{2}}} \oplus {(\mathbb{Z}/{p_t^{r_{t,2}}})}^{4(s_t-1)^2 + 4g(s_t-2) + 5g + 2{\binom{g}{2}}} \\
      	& \oplus \cdots \oplus {(\mathbb{Z}/{p_t^{r_{t,s_t-1}}})}^{16 + 9g + 2{\binom{g}{2}}} \oplus {(\mathbb{Z}/{p_t^{r_{t,s_t}}})}^{4 + 5g + 2{\binom{g}{2}}}.
      \end{align*} So the isomorphism type of $\mathcal{I}(M)$ is determined for $M$ with no 2- and 3-torsion in homology. Furthermore, the commutative diagrams in Theorem \ref{thm:torelli} and in the proof of Proposition \ref{prop:quotient} determine, for $M$ with no 2- and 3-torsion in homology, the order of each generator and which generators form a basis of $\mathcal{I}(M)$. We summarize the results into the following theorem, whose proof is a straightforward verification.
      
      \begin{Thm}
      	Let $M$ be a simply connected closed spin 5-manifold with $H_2(M)$ isomorphic to the group described in (\ref{isotype}). Assume in addition that $p_1 \geq 5$. For brevity, suppose an integer $k$ satisfies $g + 2(s_1 + \cdots + s_{i-1}) +1 \leq k \leq g + 2(s_1 + \cdots + s_i)$ for some $i$, denote \[[k] \coloneqq \left\lceil \frac{k - (g + 2(s_1 + \cdots + s_{i-1}))}{2} \right\rceil,\] and write $k \in \Delta_i$. If $1 \leq k \leq g$, write $k \in \Delta_{\infty}$. In the following we always assume $k < l < h$.
      	\begin{itemize}
      		\item[(1)] The following diffeomorphisms form a basis of $\mathcal{I}(M)$.
      		\begin{table}[h]
      			\centering
      			\begin{tabular}{c|c|c}
      				& \rm{range of subscript} & \rm{order} \\
      				\hline
      				$\mathbf{f}_k, \; \mathbf{f}_{k,k}$ & $k \in \Delta_{\infty}$ & $\infty$ \\ [1ex]
      				& $k \in \Delta_i$ & $p_i^{r_{i, [k]}}$ \\ [1ex]
      				\hline
      				$\mathbf{f}_{k,l}, \; \mathbf{f}_{k,k,l}$ & $k,l \in \Delta_{\infty}$ & $\infty$ \\ [1ex]
      				& $k \in \Delta_{\infty}, \quad l \in \Delta_i$ & $p_i^{r_{i, [l]}}$ \\ [1ex]
      				& $k,l \in \Delta_i$ & $p_i^{r_{i, [k]}}$ \\ [1ex]
      				\hline
      				$\mathbf{f}_{k,l,h}$ & $k,l,h \in \Delta_{\infty}$ & $\infty$ \\ [1ex]
      				& $\quad k,l \in \Delta_{\infty}, \quad h \in \Delta_i \quad$ & $p_i^{r_{i, [h]}}$ \\ [1ex]
      				& $k \in \Delta_{\infty}, \quad l,h \in \Delta_i$ & $p_i^{r_{i, [l]}}$ \\ [1ex]
      				& $k,l,h \in \Delta_i$ & $\qquad p_i^{r_{i, [k]}} \qquad$ \\ [1ex]
      			\end{tabular}
      		\end{table}
      		
      		All the diffeomorphisms $\mathbf{f}_*$ with other subscripts are isotopic to $id_M$.
      		
      		\item[(2)] The first of the following two tables provides a basis of $\mathcal{I}(M)$ consisting of Dehn twists, and the isotopy classes of other Dehn twists are given by the second table.
      		\begin{table}[h]
      			\centering
      			\begin{tabular}{c|c|c}
      				& \rm{range of subscript} & \rm{order} \\
      				\hline
      				$\mathbf{t}_{S_k, \alpha}, \; \mathbf{t}_{S_k \# S_k, \alpha}$ & $k \in \Delta_{\infty}$ & $\infty$ \\ [1ex]
      				& $k \in \Delta_i$ & $p_i^{r_{i, [k]}}$ \\ [1ex]
      				\hline
      				$\mathbf{t}_{S_k \# S_l, \alpha}, \; \mathbf{t}_{S_k \# S_k \# S_l, \alpha}$ & $k,l \in \Delta_{\infty}$ & $\infty$ \\ [1ex]
      				& $k \in \Delta_{\infty}, \quad l \in \Delta_i$ & $\infty$ \\ [1ex]
      				& $k,l \in \Delta_i$ & $p_i^{r_{i, [l]}}$ \\ [1ex]
      				\hline
      				$\mathbf{t}_{S_k \# S_l \# S_h, \alpha}$ & $k,l,h \in \Delta_{\infty}$ & $\infty$ \\ [1ex]
      				& $\quad k,l \in \Delta_{\infty}, \quad h \in \Delta_i \quad$ & $\infty$ \\ [1ex]
      				& $k \in \Delta_{\infty}, \quad l,h \in \Delta_i$ & $\infty$ \\ [1ex]
      				& $k,l,h \in \Delta_i$ & $\qquad p_i^{r_{i, [h]}} \qquad$ \\ [1ex]
      			\end{tabular}
      		\end{table}
      		
      		\begin{table}[h]
      			\centering
      			\begin{tabular}{c|c|c}
      				& \rm{range of subscript} & \rm{isotopy class} \\
      				\hline
      				$\mathbf{t}_{S_k \# S_l, \alpha}$ & $k \in \Delta_i, \quad l \in \Delta_j, \quad i < j$ & $\mathbf{t}_{S_k, \alpha} + \mathbf{t}_{S_l, \alpha}$ \\ [1ex]
      				\hline
      				$\mathbf{t}_{S_k \# S_k \# S_l, \alpha}$ & $k \in \Delta_i, \quad l \in \Delta_j, \quad i < j$ & $\mathbf{t}_{S_k \# S_k, \alpha} + \mathbf{t}_{S_l, \alpha}$ \\ [1ex]
      				\hline
      				$\mathbf{t}_{S_k \# S_l \# S_h, \alpha}$ & $k \in \Delta_{\infty}, \quad l \in \Delta_i, \quad h \in \Delta_j, \quad i < j$ & $\ \mathbf{t}_{S_k \# S_l, \alpha} + \mathbf{t}_{S_k \# S_h, \alpha} - \mathbf{t}_{S_k, \alpha} \ $ \\ [1ex]
      				& $\ k \in \Delta_i, \quad l \in \Delta_j, \quad h \in \Delta_n, \quad i < j < n \ $ & $\mathbf{t}_{S_k, \alpha} + \mathbf{t}_{S_l, \alpha} + \mathbf{t}_{S_h, \alpha}$ \\ [1ex]
      				& $k,l \in \Delta_i, \quad h \in \Delta_j, \quad i < j$ & $\mathbf{t}_{S_k \# S_l, \alpha} + \mathbf{t}_{S_h, \alpha}$ \\ [1ex]
      				& $k \in \Delta_i, \quad l,h \in \Delta_j, \quad i < j$ & $\mathbf{t}_{S_k, \alpha} + \mathbf{t}_{S_l \# S_h, \alpha}$ \\ [1ex]
      			\end{tabular}
      		\end{table}
      	\end{itemize}
      \end{Thm}

    \subsection{The group extension and the action of $\mathrm{Aut}(H_2(M), L_M)$ on $\mathcal{I}(M)$}
    
      We are at the position to complete the proof of Theorem \ref{thm:1}.
      
      Let $M$ be a simply connected closed 5-manifold (which is not required to be spin). Let $\mathrm{Aut}(H_2(M), L_M, w_2(M))$ be the group of automorphisms of the group $H_2(M)$ that preserve the linking form $L_M$ and the second Stiefel\textendash Whitney class $w_2(M)$ (viewed as a homomorphism $H_2(M) \rightarrow \mathbb{Z}/2$). There is a homomorphism $\phi : \mathrm{MCG}(M) \rightarrow \mathrm{Aut}(H_2(M), L_M, w_2(M))$ given by $[f] \mapsto f_*$, and $\mathrm{Aut}(H_2(M), L_M, w_2(M))$ is $\mathrm{Aut}(H_2(M), L_M)$ when $M$ is spin. An immediate consequence of \cite[Theorem 2.2]{Bar65} is the following.
      
      \begin{Thm}[Barden]\label{thm:barden}
      	There is a short exact sequence of groups \[1 \rightarrow \mathcal{I}(M) \rightarrow \mathrm{MCG}(M) \xrightarrow{\phi} \mathrm{Aut}(H_2(M), L_M, w_2(M)) \rightarrow 1.\]
      \end{Thm}
      
      The first part of Theorem \ref{thm:1} is a special case of this theorem. To study this extension problem the following action of $\mathrm{Aut}(H_2(M), L_M, w_2(M))$ on $\mathcal{I}(M)$ is crucial (see for example \cite[Chapter IV]{Bro82}). Let $h \in \mathrm{Aut}(H_2(M), L_M, w_2(M))$ be an automorphism of $H_2(M)$, and let $f_h$ be an orientation preserving diffeomorphism such that ${f_h}_* = h$. The action of $\mathrm{Aut}(H_2(M), L_M, w_2(M))$ on $\mathcal{I}(M)$ is defined by \[h \cdot [f] \coloneqq [f_h \circ f \circ {f_h}^{-1}].\]
      
      In the remaining part of this subsection let $M$ be a simply connected closed spin 5-manifold with no 2- and 3-torsion in homology. For such $M$ isotopy invariants of diffeomorphisms in $\mathcal{I}(M)$ are defined, and the action of $\mathrm{Aut}(H_2(M), L_M)$ on $\mathcal{I}(M)$ is thus calculable.
      
      The Torelli group $\mathcal{I}(M)$ is isomorphic to the bordism group $\Omega^{\mathrm{Spin}}_6(K(H_2(M), 2))$. Recall that in \ref{results} we defined the following \emph{standard action} of $\mathrm{Aut}(H_2(M), L_M)$ on $\Omega^{\mathrm{Spin}}_6(K(H_2(M), 2))$. Given $h \in \mathrm{Aut}(H_2(M), L_M)$, let $\delta_h : K(H_2(M), 2) \rightarrow K(H_2(M), 2)$ be the map inducing $h$ on the second homology, which is unique up to homotopy. Then $[N, \delta_h \circ \nu]$ is a well-defined bordism class in $\Omega^{\mathrm{Spin}}_6(K(H_2(M), 2))$ and the standard action is defined by $h \cdot [N, \nu] \coloneqq [N, \delta_h \circ \nu]$.
      
      \begin{Thm}\label{thm:action}
      	Let $M$ be a simply connected closed spin 5-manifold with no 2- and 3-torsion in $H_2(M)$. The action of $\mathrm{Aut}(H_2(M), L_M)$ on $\mathcal{I}(M)$ and the standard action of $\mathrm{Aut}(H_2(M), L_M)$ on $\Omega^{\mathrm{Spin}}_6(K(H_2(M), 2))$ coincide via the mapping torus construction $\rho$.
      \end{Thm}
      
      \begin{proof}
      	It suffices to show that $\rho(h \cdot [f]) = h \cdot \rho([f])$, namely \[[M_{f_h \circ f \circ {f_h}^{-1}}, \nu_{f_h \circ f \circ {f_h}^{-1}}] = [M_f, \delta_h \circ \nu_f].\] We prove this by calculating their bordism invariants. For brevity we write $h \cdot f$ for $f_h \circ f \circ {f_h}^{-1}$ in the following. Note that \begin{align*}
      		M_f = \biggl(M \times \Bigl([0, 1/5] \sqcup [2/5, 3/5] \sqcup [4/5, 1]\Bigr) \biggr) \Big/ & (x, 0) \sim (f(x), 1), \\
      		& (x, 1/5) \sim (x, 2/5), \\
      		& (x, 3/5) \sim (x, 4/5); \\
      		M_{h \cdot f} = \biggl(M \times \Bigl([0, 1/5] \sqcup [2/5, 3/5] \sqcup [4/5, 1]\Bigr) \biggr) \Big/ & (x, 0) \sim (f(x), 1), \\
      		& ({f_h}^{-1}(x), 1/5) \sim (x, 2/5), \\
      		& ({f_h}(x), 3/5) \sim (x, 4/5).
      	\end{align*} There is an orientation preserving diffeomorphism \[F = id_{M \times [0, 1/5]} \cup (f_h \times id_{[2/5, 3/5]}) \cup id_{M \times [4/5, 1]} : M_f \rightarrow M_{h \cdot f}.\] Let $i_f$ and $i_{h \cdot f}$ be the inclusion of $M = M \times \{\frac{1}{2}\}$ in $M_f$ and $M_{h \cdot f}$ respectively. Then we have $F \circ i_f = i_{h \cdot f} \circ f_h$. It follows that \[{\nu_{h \cdot f}}_*(p_1(M_{h \cdot f}) \cap [M_{h \cdot f}]) = h \circ {\nu_f}_*(p_1(M_f) \cap [M_f]) = {(\delta_h \circ {\nu_f})}_*(p_1(M_f) \cap [M_f])\] and \[{\nu_{h \cdot f}}_*([M_{h \cdot f}]) = {\nu_{h \cdot f}}_*F_*([M_f]) = h \circ {\nu_f}_*([M_f]) = {(\delta_h \circ {\nu_f})}_*([M_f]).\] Thus $[M_{h \cdot f}, \nu_{h \cdot f}] = [M_f, \delta_h \circ \nu_f]$.
      \end{proof}
      
      The proof of Theorem \ref{thm:1} is thus completed.
      
      \begin{Rmk}
      	\rm In \cite[Section 4]{Wal63-form} Wall gave a complete algebraic description of the linking form $L_M$. Wall's result reduces the determination of $\mathrm{Aut}(H_2(M), L_M)$ to linear algebra. In general, suppose \[H_2(M) \cong \mathbb{Z}^g \oplus {(\mathbb{Z}/{p_1^{r_1}})^2} \oplus \cdots \oplus {(\mathbb{Z}/{p_s^{r_s}})^2},\] then $L_M$ is a skew-symmetric non-singular $\mathbb{Q}/{\mathbb{Z}}$-valued bilinear form over the finite abelian group ${(\mathbb{Z}/{p_1^{r_1}})^2} \oplus \cdots \oplus {(\mathbb{Z}/{p_s^{r_s}})^2}$ which is represented by the block diagonal matrix \[\begin{pmatrix}
      	0 & \frac{1}{{p_1^{r_1}}} & & & \\
      	-\frac{1}{{p_1^{r_1}}} & 0 & & & \\
      	& & \ddots & & \\
      	& & & 0 & \frac{1}{{p_s^{r_s}}} \\
      	& & & -\frac{1}{{p_s^{r_s}}} & 0 \\
      	\end{pmatrix}\] under the standard basis of ${(\mathbb{Z}/{p_1^{r_1}})^2} \oplus \cdots \oplus {(\mathbb{Z}/{p_s^{r_s}})^2}$. Then $\mathrm{Aut}(H_2(M), L_M)$ is the group of automorphisms of the abelian group $H_2(M)$ preserving this matrix. For instance, when $H_2(M) \cong {(\mathbb{Z}/{p^r})}^2$, the group $\mathrm{Aut}(H_2(M), L_M)$ is isomorphic to $\mathrm{SL}_2(\mathbb{Z}/{p^r})$.
      \end{Rmk}

    \subsection{Diffeomorphisms acting nontrivially on homology are not Dehn twists}\label{5.5}
    
      Let $M$ be a simply connected closed spin 5-manifold with no 2-torsion in homology. We show that a diffeomorphism of $M$ acting nontrivially on $H_*(M)$ is not isotopic to any Dehn twist. This is achieved by determining all possible Dehn twists of $M$.
      
      The normal bundle of an embedding $S^k \hookrightarrow M$ is classified by its clutching function which is an element of $\pi_{k-1}(\mathrm{SO}(5-k))$, so it is trivial (when $k = 2$, use the fact that $M$ is spin). Thus for $1 \leq k \leq 4$, $S^k \times D^{5-k}$ embeds in $M$ if $S^k$ embeds in $M$. Assume that this is the case and let $S : S^k \times D^{5-k} \hookrightarrow M$ be an embedding. To define a Dehn twist about $S$ one needs to choose an element of $\pi_{5-k}(\mathrm{SO}(k+1))$, which is nontrivial only if $k = 2, 4$ ($\pi_3(\mathrm{SO}(3)) \cong \mathbb{Z}$ and $\pi_1(\mathrm{SO}(5)) \cong \mathbb{Z}/2$), so we have proved the following.
      
      \begin{Prop}\label{homologyeffect2}
      	A Dehn twist about an embedded 1-sphere or 3-sphere of a closed 5-manifold is isotopic to the identity.
      \end{Prop}
      
      By Lemma \ref{lem:homologyeffect} a Dehn twist about an embedded 2-sphere lies in the Torelli group of $M$. It remains to consider Dehn twists about embedded 4-spheres, which are of order 1 or 2 in $\mathrm{MCG}(M)$ since $\pi_1(\mathrm{SO}(5)) \cong \mathbb{Z}/2$.
      
      Let $k = 4$, then $S$ is an embedding of $S^4 \times D^1$ in $M$. A natural way to produce such an embedding is to realize it as the thickened meridian sphere of a 0-surgery (for example, the ``neck" of a connected sum). In fact an embedding of $S^4$ in $M$ always arises in this way.
      
      \begin{Prop}
      	Let $M$ be a simply connected closed spin 5-manifold, and let $S : S^4 \hookrightarrow M$ be an embedding. Then there are two simply connected closed spin 5-manifolds, such that $M$ is the connected sum of them, and $S(S^4)$ is the meridian sphere of this connected sum operation. In particular, if $H_2(M)$ is isomorphic to $\mathbb{Z}$ or ${(\mathbb{Z}/{p^r})}^2$, where $p$ is a prime number, then $S(S^4)$ is the boundary of an embedded 5-disc.
      \end{Prop}
      
      \begin{proof}
      	Since $\pi_3(\mathrm{SO}(1)) = 0$ we may extend $S$ to an embedding $S^4 \times D^1 \hookrightarrow M$. We perform a surgery on this embedding and denote the resulting manifold by $M'$. The homology groups of $M'$ can be calculated using homology exact sequence of pairs and are as follows.
      	\begin{table}[h]
      		\centering
      		\begin{tabular}{c|cccccc}
      			$i$ & $0$ & $1$ & $2$ & $3$ & $4$ & $5$ \\
      			\hline
      			$H_i(M')$ & $\mathbb{Z}^2$ & $0$ & $H_2(M)$ & $H_3(M)$ & $0$ & $\mathbb{Z}^2$ \\
      		\end{tabular}
      	\end{table}
      	
      	Thus $M'$ has two path components. By Seifert\textendash van Kampen theorem each component is a simply connected closed spin 5-manifold. Now $M$ is the connected sum of these two manifolds and $S(S^4)$ is the meridian sphere of this connected sum operation. If $H_2(M)$ is isomorphic to $\mathbb{Z}$ or ${(\mathbb{Z}/{p^r})}^2$, by the classification of simply connected 5-manifolds \cite{Sma62, Bar65}, $M'$ is diffeomorphic to $M \sqcup S^5$.
      \end{proof}
      
      \begin{Cor}
      	Only the trivial element of $\pi_4(S^2 \times S^3) \cong {(\mathbb{Z}/2)}^2$ can be represented by an embedding of $S^4$.
      \end{Cor}
      
      \begin{Cor}
      	Let $M$ be a simply connected closed spin 5-manifold, and let $S : S^4 \times D^1 \hookrightarrow M$ be an embedding.
      	\begin{itemize}
      		\item[(1)] The Dehn twist about $S$ acts trivially on $H_*(M)$.
      		
      		\item[(2)] If $H_2(M)$ has no 2-torsion, then the Dehn twist about $S$ is isotopic to $id_M$.
      		
      		\item[(3)] If $H_2(M)$ is isomorphic to $\mathbb{Z}$ or ${(\mathbb{Z}/{p^r})}^2$, where $p$ is a prime number, then the Dehn twist about $S$ is isotopic to $id_M$.
      	\end{itemize}
      \end{Cor}
      
      \begin{proof}
      	(1) and (2) are clear. To prove (3), note that this Dehn twist is supported in an embedded 5-disc, and thus its isotopy class lies in the image of $\mathrm{MCG}(D^5, \partial) \rightarrow \mathrm{MCG}(M)$, where $\mathrm{MCG}(D^5, \partial) = \Theta_6 = 0$.
      \end{proof}
      
      \begin{Cor}
      	All the Dehn twists of a simply connected closed spin 5-manifold $M$ act trivially on $H_*(M)$. If $H_2(M)$ has no 2-torsion or is isomorphic to ${(\mathbb{Z}/{2^r})}^2$, then Dehn twists about embedded $k$-spheres with $k \neq 2$ are isotopic to $id_M$.
      \end{Cor}
      
      We have seen in \ref{results} that the isotopy classes of Dehn twists about embedded 2-spheres are determined by elements of $\pi_2(M)$ and $\pi_3(\mathrm{SO}(3))$. Thus we have determined all possible Dehn twists of $M$ when $H_2(M)$ has no 2-torsion or is isomorphic to ${(\mathbb{Z}/{2^r})}^2$.
      
      \begin{Rmk}
      	\rm The results in this subsection are formulated for spin manifolds. They remain valid in the nonspin case after mild modification of the statements.
      \end{Rmk}

    \section{The mapping class group of $M_g$ and its abelianization}\label{sec:6}
    
      In this section we focus on $M_g$. We show that $\mathrm{MCG}(M_g)$ is a semi-direct product of $\mathcal{I}(M_g)$ and $\mathrm{GL}_g(\mathbb{Z})$, and compute the abelianization of $\mathrm{MCG}(M_g)$. We also determine $\mathrm{MCG}^{\pm}(S^2 \times S^3)$, the \emph{full} mapping class group of $S^2 \times S^3$.
    
    \subsection{Splitting the group extension for $M_g$}
    
      Recall that, given a short exact sequence of groups \[0 \rightarrow A \rightarrow E \xrightarrow{\phi} Q \rightarrow 1\] with $A$ abelian, the group $Q$ acts on $A$ by $g \cdot a = s(g)as(g)^{-1}$, where $s : Q \rightarrow E$ is a set theoretic cross section. This action is independent to the choice of $s$ and is used to define the group structure of the semi-direct product $A \rtimes Q$ (see \cite[Chapter IV]{Bro82}). We note the following.
      
      \begin{Lem}\label{lem:split}
      	If there is a map $\chi : E \rightarrow A$ which restricts to identity on $A$, and \[\chi(g_1g_2) = \chi(g_1) + g_1\chi(g_2)g_1^{-1}\] for $g_1, g_2 \in E$, then the short exact sequence splits.
      \end{Lem}
      
      \begin{proof}
      	The map $(\chi, \phi) : E \rightarrow A \rtimes Q$ is a homomorphism which fits into the following commutative diagram \[\begin{tikzcd}
      		0 \arrow[r] & A \arrow[r] \arrow[d, "="] & E \arrow[r, "\phi"] \arrow[d, "{(\chi, \phi)}"] & Q \arrow[r] \arrow[d, "="] & 1 \\
      		0 \arrow[r] & A \arrow[r] & A \rtimes Q \arrow[r] & Q \arrow[r] & 1
      	\end{tikzcd}\] and therefore is an isomorphism.
      \end{proof}
      
      We call such a map $\chi$ a \emph{crossed splitting}.
      
      For $M = M_g$ the short exact sequence in Theorem \ref{thm:barden} takes the form \[0 \rightarrow \mathcal{I}(M_g) \rightarrow \mathrm{MCG}(M_g) \rightarrow \mathrm{GL}_g(\mathbb{Z}) \rightarrow 1\] if we identify $H_2(M_g)$ with $\mathbb{Z}^g$ via the standard basis ${S_k}_*([S^2])$. The action of $\mathrm{GL}_g(\mathbb{Z})$ on $\mathcal{I}(M_g)$ is described by Theorem \ref{thm:action}. We now construct a crossed splitting $\chi : \mathrm{MCG}(M_g) \rightarrow \mathcal{I}(M_g)$. We identify $\mathcal{I}(M_g)$ with ${\Omega}^{\mathrm{Spin}}_6((\mathbb{CP}^{\infty})^g)$ via the mapping torus construction $\rho$.
      
      Let $f$ be an orientation preserving diffeomorphism of $M_g$. Consider the twisted double \[X_f \coloneqq {\natural}^g(S^2 \times D^4) \cup_f \overline{{\natural}^g(S^2 \times D^4)}.\] This is a simply connected closed oriented 6-manifold. Using the homology exact sequence we see that the inclusion maps induce isomorphisms $H_2(M_g) \cong H_2({\natural}^g(S^2 \times D^4)) \cong H_2(X_f) \cong \mathbb{Z}^g$. This implies that $X_f$ is a spin manifold with unique spin structure, and we could find a map $\tilde{\nu}_f : X_f \rightarrow {(\mathbb{CP}^{\infty})}^g$ such that the diagram \[\begin{tikzcd}
      	{H_2}({X}_f) \arrow[r, "{\tilde{\nu}_{f *}}"] & {H_2}({(\mathbb{CP}^{\infty})}^g) \\
      	{H_2}({M}_g) \arrow[u] \arrow[r, "{i_* \circ f_*}"] & {H_2}({({M}_g)}_f) \arrow[u, "{{\nu}_{f *}}"]
      \end{tikzcd}\] commutes, where the left vertical map is induced by the inclusion map  which factors through the inclusion $M_g = \partial ({\natural}^g(S^2 \times D^4)) \hookrightarrow {\natural}^g(S^2 \times D^4)$ into the first copy of ${\natural}^g(S^2 \times D^4)$. This determines a bordism class $[X_f, \tilde{\nu}_f] \in {\Omega}^{\mathrm{Spin}}_6((\mathbb{CP}^{\infty})^g)$. The map \begin{align*}
      \chi : \mathrm{MCG}(M_g) & \rightarrow {\Omega}^{\mathrm{Spin}}_6((\mathbb{CP}^{\infty})^g) \\
      [f] & \mapsto [X_f, \tilde{\nu}_f]
      \end{align*} is clearly well-defined.
      
      \begin{Lem}\label{lem:double1}
      	If $[f] \in \mathcal{I}(M_g)$, then $\chi([f]) = \rho([f])$.
      \end{Lem}
      
      \begin{proof}
      	Let $B \coloneqq {(\mathbb{CP}^{\infty})}^g \times B\mathrm{Spin}$. It suffices to construct a normal $B$-bordism between $(X_f, \tilde{\nu}_f)$ and $({(M_g)}_f, \nu_f)$. Choose a collar neighborhood $M_g \times [0, 1] \subset {\natural}^g(S^2 \times D^4)$. The manifold \[V \coloneqq {\natural}^g(S^2 \times D^4) \times I/{\sim},\] obtained by identifying $M_g \times [0, 1] \times \{0\}$ with $M_g \times [0, 1] \times \{1\}$ via the gluing map $f \times id_{[0, 1]} : M_g \times [0, 1] \times \{0\} \rightarrow M_g \times [0, 1] \times \{1\}$, is an oriented bordism between $X_f$ and ${(M_g)}_f$. We show that $V$ can be made into a normal $B$-bordism. Using Mayer\textendash Vietoris sequences we see that $V$ is a spin manifold having two spin structures which restrict to the two spin structures of ${(M_g)}_f$ respectively. Since $X_f$ has a unique spin structure, we can choose the spin structure of $V$ such that $V$ is a spin bordism. Moreover, from Mayer\textendash Vietoris sequences we see that the inclusion maps $X_f \hookrightarrow V$ and ${(M_g)}_f \hookrightarrow V$ induce isomorphisms \[H_2(X_f) \cong H_2(V) \cong H_2({(M_g)}_f),\] consequently by universal coefficient theorem there are isomorphisms \[\begin{tikzcd}
      	\mathrm{Hom}_{\mathbb{Z}}(H_2(X_f), \mathbb{Z}^g) \arrow[d, "="] & \mathrm{Hom}_{\mathbb{Z}}(H_2(V), \mathbb{Z}^g) \arrow[l, "\cong"'] \arrow[r, "\cong"] \arrow[d, "="] & \mathrm{Hom}_{\mathbb{Z}}(H_2({(M_g)}_f), \mathbb{Z}^g) \arrow[d, "="] \\
      	{[{X_f}, {(\mathbb{CP}^{\infty})}^g]} & {[V, {(\mathbb{CP}^{\infty})}^g] \arrow[l, "\cong"']} \arrow[r, "\cong"] & {[{{(M_g)}_f}, {(\mathbb{CP}^{\infty})}^g]}.
      	\end{tikzcd}\] Since $f$ lies in the Torelli group, by our definition of $\tilde{\nu}_f$ there is a map $\overline{\nu}_f : V \rightarrow {(\mathbb{CP}^{\infty})}^g$ extends $\tilde{\nu}_f \sqcup \nu_f$ up to homotopy. By gluing collar neighborhoods to $V$ if needed, we obtain the desired normal $B$-bordism between $(X_f, \tilde{\nu}_f)$ and $({(M_g)}_f, \nu_f)$.
      \end{proof}
      
      \begin{Lem}\label{lem:double2}
      	The map $\chi : \mathrm{MCG}(M_g) \rightarrow {\Omega}^{\mathrm{Spin}}_6((\mathbb{CP}^{\infty})^g)$ is a crossed splitting.
      \end{Lem}
      
      \begin{proof}
      	By Lemma \ref{lem:double1} it suffices to show that $\chi([f \circ f']) = \chi([f]) + [f \circ \chi([f']) \circ f^{-1}]$. Namely, let $h \coloneqq f_* \in \mathrm{GL}_g(\mathbb{Z})$, let $\delta_h : (\mathbb{CP}^{\infty})^g \rightarrow (\mathbb{CP}^{\infty})^g$ be the map inducing $h$ on $\mathbb{Z}^g$, we show that \[[X_{f \circ f'}, \tilde{\nu}_{f \circ f'}] = [X_f, \tilde{\nu}_f] + [X_{f'}, \delta_h \circ \tilde{\nu}_{f'}].\]
      	
      	We construct a spin bordism $Y$ between $X_f \sqcup X_{f'}$ and $X_{f \circ f'}$ as follows. Let \begin{align*}
      		& Y_{f'} \coloneqq \bigl({\natural}^g(S^2 \times D^4) \times [0, 2]\bigr) \cup_{f' \times id_{[0, 1]}} \bigl({\natural}^g(S^2 \times D^4) \times [0, 1]\bigr), \\
      		& Y_f \coloneqq \bigl({\natural}^g(S^2 \times D^4) \times [0, 1]\bigr) \cup_{f \times id_{[0, 1]}} \bigl({\natural}^g(S^2 \times D^4) \times [0, 2]\bigr).
      	\end{align*} Let $Z$ be the manifold obtained by gluing $M_g \times [1, 2]$ and ${\natural}^g(S^2 \times D^4)$ via the identification $M_g \times \{1\} = \partial ({\natural}^g(S^2 \times D^4))$. Clearly $Z$ is diffeomorphic to ${\natural}^g(S^2 \times D^4)$. Let \[Y \coloneqq Y_{f'} \cup_{F'} (Z \times I) \cup_F Y_f\] where $F'$ is the map from $(\partial ({\natural}^g(S^2 \times D^4)) \times [1, 2]) \cup_{f'} ({\natural}^g(S^2 \times D^4) \times \{1\}) \subset \partial Y_{f'}$ to $Z \times \{0\}$ given by the union of \begin{align*}
      	f' \times id_{[1, 2]} & : \partial ({\natural}^g(S^2 \times D^4)) \times [1, 2] \rightarrow M_g \times [1, 2] \times \{0\}, \\
      	id & : {\natural}^g(S^2 \times D^4) \times \{1\} \rightarrow {\natural}^g(S^2 \times D^4) \times \{0\},
      	\end{align*} and $F$ is the map from $Z \times \{1\}$ to $({\natural}^g(S^2 \times D^4) \times \{1\}) \cup_f (\partial ({\natural}^g(S^2 \times D^4)) \times [1, 2]) \subset \partial Y_f$ given by the union of \begin{align*}
      	f \times id_{[1, 2]} & : M_g \times [1, 2] \times \{1\} \rightarrow \partial ({\natural}^g(S^2 \times D^4)) \times [1, 2], \\
      	id & : {\natural}^g(S^2 \times D^4) \times \{1\} \rightarrow {\natural}^g(S^2 \times D^4) \times \{1\}.
      	\end{align*} Clearly $Y$ is a simply connected bordism between $X_f \sqcup X_{f'}$ and $X_{f \circ f'}$. Using Mayer\textendash Vietoris sequence we see that $Y$ is spin. Then the spin structure of $Y$ is unique and thus $Y$ is a spin bordism.
      	
      	We shall construct a map $\tilde{\nu} : Y \rightarrow {(\mathbb{CP}^{\infty})}^g$ extending $\tilde{\nu}_{f \circ f'} \sqcup \tilde{\nu}_f \sqcup \delta_h \circ \tilde{\nu}_{f'}$, thus making $(Y, \tilde{\nu})$ into a normal $B$-bordism (where $B \coloneqq {(\mathbb{CP}^{\infty})}^g \times B\mathrm{Spin}$) and as a consequence $[X_{f \circ f'}, \tilde{\nu}_{f \circ f'}] = [X_f, \tilde{\nu}_f] + [X_{f'}, \delta_h \circ \tilde{\nu}_{f'}]$. Denote the restriction of $\tilde{\nu}_f$ on the first and on the second copy of ${\natural}^g(S^2 \times D^4)$ by ${(\tilde{\nu}_f)}_1$ and ${(\tilde{\nu}_f)}_2$ respectively. Similar notations apply to $\tilde{\nu}_{f'}$ and $\tilde{\nu}_{f \circ f'}$. Note that by our definition of $\tilde{\nu}_f$ we have \[\delta_h \circ {{(\tilde{\nu}_{f'})}_2}_* = {{(\tilde{\nu}_f)}_1}_* : H_2({\natural}^g(S^2 \times D^4)) \rightarrow H_2({(\mathbb{CP}^{\infty})}^g),\] thus $\delta_h \circ {(\tilde{\nu}_{f'})}_2$ is homotopic to ${(\tilde{\nu}_f)}_1$. Choose a homotopy $H : {\natural}^g(S^2 \times D^4) \times I \rightarrow {(\mathbb{CP}^{\infty})}^g$ from $\delta_h \circ {(\tilde{\nu}_{f'})}_2$ to ${(\tilde{\nu}_f)}_1$, and denote the restriction of $H$ on $\partial ({\natural}^g(S^2 \times D^4)) \times I$ by $\partial H$. The map $\tilde{\nu} : Y \rightarrow {(\mathbb{CP}^{\infty})}^g$ is defined as the union of the following maps \begin{align*}
      		\bigl((\delta_h \circ {(\tilde{\nu}_{f'})}_1) \times [0, 2]\bigr) \cup \bigl((\delta_h \circ {(\tilde{\nu}_{f'})}_2) \times [0, 1]\bigr) & : Y_{f'} \rightarrow {(\mathbb{CP}^{\infty})}^g, \\
      		H \cup \bigl(\partial H \times [1, 2]\bigr) & : Z \times I \rightarrow {(\mathbb{CP}^{\infty})}^g, \\
      		\bigl({(\tilde{\nu}_f)}_1 \times [0, 1]\bigr) \cup \bigl({(\tilde{\nu}_f)}_2 \times [0, 2]\bigr) & : Y_f \rightarrow {(\mathbb{CP}^{\infty})}^g.
      	\end{align*} Since the effects of $\delta_h \circ {(\tilde{\nu}_{f'})}_1$ and ${(\tilde{\nu}_{f \circ f'})}_1$ on the second homology groups coincide, $(X_{f \circ f'}, \tilde{\nu}_{f \circ f'})$ and $(X_{f \circ f'}, \tilde{\nu})$ are normal $B$-bordant, so $(Y, \tilde{\nu})$ can be viewed as a normal $B$-bordism between $(X_{f \circ f'}, \tilde{\nu}_{f \circ f'})$ and $(X_f, \tilde{\nu}_f) \sqcup (X_{f'}, \delta_h \circ \tilde{\nu}_{f'})$. This proves the lemma.
      \end{proof}
      
      Therefore $\mathrm{MCG}(M_g) \cong \mathcal{I}(M_g) \rtimes \mathrm{GL}_g(\mathbb{Z})$. Moreover, let $\mathrm{Top}^{+}(M)$ be the topological group of orientation preserving homeomorphisms of an oriented manifold $M$, and let $G^{+}(M)$ be the topological monoid of orientation preserving homotopy equivalences of $M$, both equipped with compact-open topology. The \emph{topological mapping class group} of $M$, \[\mathrm{MCG}^{\mathrm{Top}}(M) \coloneqq \pi_0\mathrm{Top}^{+}(M),\] is the group of topological isotopy classes of orientation preserving homeomorphisms of $M$, and the \emph{homotopy automorphism group} of $M$, \[\mathrm{haut}(M) \coloneqq \pi_0G^{+}(M),\] is the group of homotopy classes of orientation preserving homotopy equivalences of $M$. Recall that we have defined in \ref{results} the \emph{homotopy Torelli group} $\mathcal{I}^h(M)$ of $M$ consisting of homotopy classes of homotopy equivalences acting trivially on $H_*(M)$. Similarly the \emph{topological Torelli group} $\mathcal{I}^{\mathrm{Top}}(M)$ of $M$ is the subgroup of $\mathrm{MCG}^{\mathrm{Top}}(M)$ consisting of homeomorphisms acting trivially on $H_*(M)$. If $M$ is a simply connected closed 5-manifold, Barden's result \cite[Theorem 2.2]{Bar65} implies that there is a commutative diagram of short exact sequences \[\begin{tikzcd}
      0 \arrow[r] & \mathcal{I}(M) \arrow[r] \arrow[d] & \mathrm{MCG}(M) \arrow[r] \arrow[d] & \mathrm{Aut}(H_2(M), L_M, w_2(M)) \arrow[r] \arrow[d, "="] & 1 \\
      1 \arrow[r] & \mathcal{I}^{\mathrm{Top}}(M) \arrow[r] \arrow[d] & \mathrm{MCG}^{\mathrm{Top}}(M) \arrow[r] \arrow[d] & \mathrm{Aut}(H_2(M), L_M, w_2(M)) \arrow[r] \arrow[d, "="] & 1 \\
      1 \arrow[r] & \mathcal{I}^h(M) \arrow[r] & \mathrm{haut}(M) \arrow[r] & \mathrm{Aut}(H_2(M), L_M, w_2(M)) \arrow[r] & 1
      \end{tikzcd}\] where the vertical maps are forgetful homomorphisms. Combining this with Lemma \ref{lem:split} and Lemma \ref{lem:double2} we have the following.
      
      \begin{Thm}\label{thm:mcg}
      	There are splitting short exact sequences \[0 \rightarrow \mathcal{I}(M_g) \rightarrow \mathrm{MCG}(M_g) \rightarrow \mathrm{GL}_g(\mathbb{Z}) \rightarrow 1,\] \[1 \rightarrow \mathcal{I}^{\mathrm{Top}}(M_g) \rightarrow \mathrm{MCG}^{\mathrm{Top}}(M_g) \rightarrow \mathrm{GL}_g(\mathbb{Z}) \rightarrow 1,\] \[1 \rightarrow \mathcal{I}^h(M_g) \rightarrow \mathrm{haut}(M_g) \rightarrow \mathrm{GL}_g(\mathbb{Z}) \rightarrow 1.\]
      \end{Thm}
      
      Here for a short exact sequence \[0 \rightarrow A \rightarrow E \xrightarrow{\phi} Q \rightarrow 1\] with $A$ not necessarily an abelian group, we say that this short exact sequence \emph{splits} if there is a group theoretic cross section, namely a homomorphism $s : Q \rightarrow E$ such that $\phi \circ s = id_Q$.
      
      By Theorem \ref{thm:mcg} the group structure of $\mathrm{MCG}(M_g)$ is determined. This completes the proof of Theorem \ref{thm:2}.
      
      \begin{Rmk}
      	\rm The determination of the group structures of $\mathrm{MCG}^{\mathrm{Top}}(M_g)$ and $\mathrm{haut}(M_g)$ leaves open. (However, see \cite{BB96} for various results on $\mathrm{haut}(M_g)$.)
      \end{Rmk}
      
      \begin{Rmk}
      	\rm Given an orientation preserving diffeomorphism $f$ of $M_g$ one may also consider the simply connected closed spin 6-manifolds \[{\natural}^g(S^3 \times D^3) \cup_f \overline{{\natural}^g(S^3 \times D^3)}, \quad {\natural}^g(S^2 \times D^4) \cup_f \overline{{\natural}^g(S^3 \times D^3)}.\] By computing their homology groups and applying \cite[Theorem 1]{Wal66-classification5} and \cite{KM63}, one sees that the first manifold is diffeomorphic to ${\#}^g(S^3 \times S^3)$, and the second manifold is diffeomorphic to $S^6$.
      \end{Rmk}

    \subsection{The abelianization of $\mathrm{MCG}(M_g)$}
    
      The Hochschild\textendash Serre spectral sequence of the exact sequence \[0 \rightarrow \mathcal{I}(M_g) \rightarrow \mathrm{MCG}(M_g) \xrightarrow{\phi} \mathrm{GL}_g(\mathbb{Z}) \rightarrow 1\] with trivial coefficient module $\mathbb{Z}$ yields the following exact sequence \[H_2(\mathrm{MCG}(M_g)) \xrightarrow{\phi_*} H_2(\mathrm{GL}_g(\mathbb{Z})) \rightarrow {\mathcal{I}(M_g)}_{\mathrm{GL}_g(\mathbb{Z})} \rightarrow H_1(\mathrm{MCG}(M_g)) \xrightarrow{\phi_*} H_1(\mathrm{GL}_g(\mathbb{Z})) \rightarrow 0.\] Here ${\mathcal{I}(M_g)}_{\mathrm{GL}_g(\mathbb{Z})}$ denotes the coinvariant, which is the quotient group of $\mathcal{I}(M_g)$ modulo the additive subgroup generated by \[\{h \cdot [f] - [f] \; | \; h \in \mathrm{GL}_g(\mathbb{Z}), [f] \in \mathcal{I}(M_g)\}.\] It follows from Theorem \ref{thm:mcg} that $\phi_* : H_2(\mathrm{MCG}(M_g)) \rightarrow H_2(\mathrm{GL}_g(\mathbb{Z}))$ is surjective, and there is a splitting short exact sequence \[0 \rightarrow {\mathcal{I}(M_g)}_{\mathrm{GL}_g(\mathbb{Z})} \rightarrow H_1(\mathrm{MCG}(M_g)) \xrightarrow{\phi_*} H_1(\mathrm{GL}_g(\mathbb{Z})) \rightarrow 0.\] It is a well-known fact that
      \[H_1(\mathrm{GL}_g(\mathbb{Z})) \cong \begin{cases}
      	{(\mathbb{Z}/2)}^2, & g = 2; \\
      	\mathbb{Z}/2, & g \neq 2.
      \end{cases}\]
      Thus, to compute the abelianization $H_1(\mathrm{MCG}(M_g))$ of $\mathrm{MCG}(M_g)$, it suffices to compute the coinvariant ${\mathcal{I}(M_g)}_{\mathrm{GL}_g(\mathbb{Z})}$.
      
      \begin{Lem}\label{lem:coinvariant}
      	We have
      	\[{\mathcal{I}(M_g)}_{\mathrm{GL}_g(\mathbb{Z})} \cong \begin{cases}
      		{(\mathbb{Z}/2)}^2, & g = 1; \\
      		\mathbb{Z}/2, & g = 2, 3; \\
      		0, & g \geq 4.
      	\end{cases}\]
      \end{Lem}
      
      \begin{proof}
      	The proof is a straightforward calculation using the action of $\mathrm{GL}_g(\mathbb{Z})$ on $\mathcal{I}(M_g)$ determined in Theorem \ref{thm:action}. For $g = 1$ the result is clear. For $g \geq 2$, since $\mathcal{I}(M_g) \cong \mathbb{Z}^g \oplus H_g$ and $\mathrm{GL}_g(\mathbb{Z})$ acts on the $\mathbb{Z}^g$ summand in the standard way (this follows from the definition of the bordism invariants), we have ${\mathcal{I}(M_g)}_{\mathrm{GL}_g(\mathbb{Z})} = {(H_g)}_{\mathrm{GL}_g(\mathbb{Z})}$. Recall that the general linear group $\mathrm{GL}_g(\mathbb{Z})$ is generated by the (block) diagonal matrices (where the empty entries are 1 on the diagonal and 0 otherwise) \[\begin{pmatrix}
      		1 & 1 & 0 & & \\
      		0 & 1 & 0 & & \\
      		0 & 0 & 1 & & \\
      		& & & \ddots & \\
      		& & & & 1 \\
      	\end{pmatrix}, \;
      	\begin{pmatrix}
      		-1 & 0 & 0 & & \\
      		0 & 1 & 0 & & \\
      		0 & 0 & 1 & & \\
      		& & & \ddots & \\
      		& & & & 1 \\
      	\end{pmatrix}, \;
      	\begin{pmatrix}
      		0 & 1 & 0 & & \\
      		1 & 0 & 0 & & \\
      		0 & 0 & 1 & & \\
      		& & & \ddots & \\
      		& & & & 1 \\
      	\end{pmatrix}, \cdots,
      	\begin{pmatrix}
      		1 & & & & \\
      		& \ddots & & & \\
      		& & 1 & 0 & 0 \\
      		& & 0 & 0 & 1 \\
      		& & 0 & 1 & 0 \\
      	\end{pmatrix}.\] It suffices to check the action of these matrices on the generators of $H_g$, which can be calculated using the universal coefficient theorem.
      	
      	For $g = 2, 3$ we see that the generators $2(x_k^2x_l)^*$ are zero in the coinvariant, and the remaining generators are mapped to the same element, which is of order at most 2, in the coinvariant. Thus the coinvariant for $g = 2, 3$ is $0$ or $\mathbb{Z}/2$. Consider the homomorphism $H_g \rightarrow \mathbb{Z}/2$ defined by adding the coordinates except those correspond to the generators $2(x_k^2x_l)^*$. By checking on the generators of $\mathrm{GL}_g(\mathbb{Z})$ we see that this is a surjective homomorphism which is invariant under the action of $\mathrm{GL}_g(\mathbb{Z})$, thus it induces a surjection from the coinvariant to $\mathbb{Z}/2$, it follows that the coinvariant is $\mathbb{Z}/2$ for $g = 2, 3$. For $g \geq 4$ the aforementioned effects on the generators of $H_g$ are still true, namely, $2(x_k^2x_l)^*$ are zero in the coinvariant, and the remaining generators (including $(x_kx_lx_h)^*$) are mapped to the same element, which is of order at most 2, in the coinvariant. Moreover, the generators $(x_kx_lx_h)^*$ are mapped to zero, which force all generators of $H_g$ to be zero in the coinvariant. Thus the coinvariant is trivial for $g \geq 4$.
      \end{proof}
      
      \begin{proof}[Proof of Corollary \ref{cor:abelian}]
      	Part (1) is evident. Part (2) follows from the well-known fact that given a topological group $G$, one has $\pi_1(BG) \cong \pi_0(G)$. Therefore $H_1(BG) \cong H_1(\pi_0(G))$.
      \end{proof}

    \subsection{The full mapping class group of $S^2 \times S^3$}\label{6.3}
    
      Barden's theorem \cite[Theorem 2.2]{Bar65} implies that for each simply connected closed 5-manifold $M$ there is an exact sequence \[1 \rightarrow \mathrm{MCG}(M) \rightarrow \mathrm{MCG}^{\pm}(M) \rightarrow \mathbb{Z}/2 \rightarrow 1,\] where $\mathrm{MCG}^{\pm}(M) \rightarrow \mathbb{Z}/2$ maps an isotopy class to its degree. Whether this exact sequence splits or not is unknown in general. For $M = S^2 \times S^3$ a (group theoretic) cross section $\mathbb{Z}/2 \rightarrow \mathrm{MCG}^{\pm}(S^2 \times S^3)$ is provided by the orientation reversing diffeomorphism $r : S^2 \times S^3 \rightarrow S^2 \times S^3$ defined as $(x, y) \mapsto (-x, y)$, where $-1$ denotes the antipodal map. As a generalization of Corollary \ref{cor:product} we have the following computation of $\mathrm{MCG}^{\pm}(S^2 \times S^3)$.
      
      \begin{Thm}
      	The full mapping class group $\mathrm{MCG}^{\pm}(S^2 \times S^3)$ has the following presentation: \[\left\langle a, b, v, w | ab = ba, v^2 = 1, vav = a^{-1}, vbv = b^{-1}, w^2 = 1, vw = wv, aw = wa, bw = wb \right\rangle,\] where $a$, $b$, $v$, $w$ can be represented by $\mathbf{t}_{S_1, \alpha}$, $\mathbf{t}_{S_1 \# S_1, \alpha}$, $R$ and $r$ respectively.
      \end{Thm}
      
      \begin{proof}
      	We only need to prove that, given a diffeomorphism $f$ whose isotopy class lies in the Torelli group $\mathcal{I}(S^2 \times S^3)$, we have $r \circ f \circ r$ isotopic to $f$. The proof is similar to that of Theorem \ref{thm:action}. In the following we write $r \cdot f$ for $r \circ f \circ r$. Note that \begin{align*}
      		(S^2 \times S^3)_f = \bigl((S^2 \times S^3 \times I) \sqcup (S^2 \times S^3)\bigr) \Big/ & (x, 0) \sim f(x), \\
      		& (x, 1) \sim x; \\
      		(S^2 \times S^3)_{r \cdot f} = \bigl((S^2 \times S^3 \times I) \sqcup (S^2 \times S^3)\bigr) \Big/ & (x, 0) \sim f \circ r(x), \\
      		& (x, 1) \sim r(x).
      	\end{align*} There is an orientation reversing diffeomorphism \[F = (r \times id_I) \cup id_{S^2 \times S^3} : (S^2 \times S^3)_f \rightarrow (S^2 \times S^3)_{r \cdot f}.\] We have $F_*([(S^2 \times S^3)_f]) = -[(S^2 \times S^3)_{r \cdot f}]$ and $p_1((S^2 \times S^3)_f) = F^*(p_1((S^2 \times S^3)_{r \cdot f}))$. By calculating the bordism invariants from these facts we see that $r \cdot f$ is isotopic to $f$.
      \end{proof}
      
      Using the geometric interpretation of the Hopf invariant, one easily calculates the induced homomorphisms of $R$ and $r$ on $\pi_3(S^2 \times S^3)$, and obtain that $R_*([S_1 \circ \eta]) = [S_1 \circ \eta]$ and $R_*([T_1]) = -[T_1]$, and $r_* = id$. Thus the action of $\mathrm{MCG}^{\pm}(S^2 \times S^3)$ on $\pi_3(S^2 \times S^3)$ is completely determined. This will be used in Section \ref{sec:9}.

    \section{Stabilization of $\mathrm{MCG}(M)$}\label{sec:7}
    
      In this section $M$ denotes a simply connected closed spin manifold with no 2-torsion in homology. We study the stabilization of $\mathrm{MCG}(M)$ with respect to connected sum with $S^2 \times S^3$. Let $D \subset M$ be an embedded 5-disc, and let $W \coloneqq M \setminus \mathring{D}$ be the compact 5-manifold obtained by digging out the interior of $D$ from $M$. Corollary \ref{cor:forgetful} implies that \[\mathrm{MCG}(M) = \mathrm{MCG}(M, D) = \mathrm{MCG}(W, \partial)\] and \[\mathcal{I}(M) = \mathcal{I}(M, D) = \mathcal{I}(W, \partial).\] Fix two disjoint embedded 5-discs $D_1, D_2 \subset S^2 \times S^3$. The manifold $(S^2 \times S^3) \setminus (\mathring{D_1} \sqcup \mathring{D_2})$ has two boundary components, each of which is diffeomorphic to $S^4$. Identify one of the boundary components of $(S^2 \times S^3) \setminus (\mathring{D_1} \sqcup \mathring{D_2})$ with the boundary of $W$, and glue $(S^2 \times S^3) \setminus (\mathring{D_1} \sqcup \mathring{D_2})$ with $W$ via this identification, we obtain a manifold diffeomorphic to $(M \# (S^2 \times S^3)) \setminus \mathring{D}$. Let $f$ be a diffeomorphism in $\mathrm{MCG}(W, \partial)$, extending $f$ to $(M \# (S^2 \times S^3)) \setminus \mathring{D}$ by $id_{(S^2 \times S^3) \setminus (\mathring{D_1} \sqcup \mathring{D_2})}$ yields a diffeomorphism in $\mathrm{MCG}((M \# (S^2 \times S^3)) \setminus \mathring{D}, \partial)$. This induces a well-defined homomorphism \[\mathrm{MCG}(W, \partial) \rightarrow \mathrm{MCG}((M \# (S^2 \times S^3)) \setminus \mathring{D}, \partial),\] which can be viewed as a homomorphism \[\mathrm{MCG}(M) \rightarrow \mathrm{MCG}((M \# (S^2 \times S^3)))\] mapping an isotopy class $[f]$ to the well-defined isotopy class $[f \# id_{S^2 \times S^3}]$, independent to the choice of $D \subset M$ by the disc theorem. Clearly it restricts to a homomorphism \[\mathcal{I}(M) \rightarrow \mathcal{I}(M \# (S^2 \times S^3)).\] We call these homomorphisms the \emph{stabilization maps}. Two orientation preserving diffeomorphisms of $M$ are \emph{stably isotopic} if their isotopy classes are mapped to the same element by a composition of stabilization maps. Denote $\mathrm{MCG}(M \# M_g)$ and $\mathcal{I}(M \# M_g)$ by $\mathrm{MCG}_g(M)$ and $\mathcal{I}_g(M)$ respectively. The stabilization maps fit into the following commutative diagram of short exact sequence \[\begin{tikzcd}
      0 \arrow[r] & \mathcal{I}(M) \arrow[r] \arrow[d] & \mathrm{MCG}(M) \arrow[r, "\phi"] \arrow[d] & \mathrm{Aut}(H_2(M), L_M) \arrow[r] \arrow[d] & 1 \\
      0 \arrow[r] & \mathcal{I}_1(M) \arrow[r] \arrow[d] & \mathrm{MCG}_1(M) \arrow[r, "\phi"] \arrow[d] & \mathrm{Aut}(H_2(M) \oplus \mathbb{Z}, L_M) \arrow[r] \arrow[d] & 1 \\
      & \vdots \arrow[d] & \vdots \arrow[d] & \vdots \arrow[d] & \\
      0 \arrow[r] & \mathcal{I}_g(M) \arrow[r] \arrow[d] & \mathrm{MCG}_g(M) \arrow[r, "\phi"] \arrow[d] & \mathrm{Aut}(H_2(M) \oplus \mathbb{Z}^g, L_M) \arrow[r] \arrow[d] & 1 \\
      0 \arrow[r] & \mathcal{I}_{g+1}(M) \arrow[r] \arrow[d] & \mathrm{MCG}_{g+1}(M) \arrow[r, "\phi"] \arrow[d] & \mathrm{Aut}(H_2(M) \oplus \mathbb{Z}^{g+1}, L_M) \arrow[r] \arrow[d] & 1 \\
      & \vdots & \vdots & \vdots &
      \end{tikzcd}\] where $\mathrm{Aut}(H_2(M) \oplus \mathbb{Z}^g, L_M) \rightarrow \mathrm{Aut}(H_2(M) \oplus \mathbb{Z}^{g+1}, L_M)$ is given by taking direct sum with $id : \mathbb{Z} \rightarrow \mathbb{Z}$. This is clearly injective for each $g$. By the computation of $\mathcal{I}(M)$ we see that the stabilization map $\mathcal{I}_g(M) \rightarrow \mathcal{I}_{g+1}(M)$ is injective for each $g$, thus the stabilization map $\mathrm{MCG}_g(M) \rightarrow \mathrm{MCG}_{g+1}(M)$ is injective for each $g$. In other words, two orientation preserving diffeomorphisms of $M$ are not stably isotopic if they are not isotopic. This proves Corollary \ref{cor:stable}.
      
      In particular, the stabilization map for $M_g$ fits into the following commutative diagram \[\begin{tikzcd}
      	\mathrm{MCG}(M_g) \arrow[r, "\phi"] \arrow[d] & \mathrm{GL}_g(\mathbb{Z}) \arrow[d] \\
      	\mathrm{MCG}(M_{g+1}) \arrow[r, "\phi"] & \mathrm{GL}_{g+1}(\mathbb{Z}).
      \end{tikzcd}\] By Lemma \ref{lem:coinvariant}, $\phi_* : H_1(\mathrm{MCG}(M_g)) \rightarrow H_1(\mathrm{GL}_g(\mathbb{Z}))$ is an isomorphism for $g \geq 4$. By the homological stability of $\mathrm{GL}_g(\mathbb{Z})$ \cite[Theorem 3.2]{Cha80}, the stabilization map induces an isomorphism $H_1(\mathrm{GL}_g(\mathbb{Z})) \cong H_1(\mathrm{GL}_{g+1}(\mathbb{Z}))$ for $g \geq 3$. Thus, for $g \geq 4$, the stabilization map induces an isomorphism \[H_1(\mathrm{MCG}(M_g)) \xrightarrow{\cong} H_1(\mathrm{MCG}(M_{g+1})).\] This proves Corollary \ref{cor:homologystable}.
      
      \begin{Rmk}
      	\rm Let $W_g \coloneqq M_g \setminus \mathring{D}$. Then one can define the stabilization map for relative diffeomorphism groups \[\mathrm{Diff}(W_g, \partial) \rightarrow \mathrm{Diff}(W_{g+1}, \partial)\] which induces a stabilization map on the classifying spaces \[B\mathrm{Diff}(W_g, \partial) \rightarrow B\mathrm{Diff}(W_{g+1}, \partial).\] The homological stability of $B\mathrm{Diff}(W_g, \partial)$ under the stabilization map is an important topic in geometric topology (see for example \cite{MW07, GRW18-stability1, GRW17-stability2, Per16-product, Per16-odd, ER24}). For dimension 5 the homological stability result has not been established. Our result shows that $H_1(B\mathrm{Diff}(W_g, \partial))$ stabilizes under the stabilization map when $g \geq 4$.
      \end{Rmk}
      
      Finally we consider the colimits of the direct system of mapping class groups formed by the stabilization map. Denote \begin{align*}
      	\mathrm{MCG}_{\infty}(M) & \coloneqq \mathrm{colim}_{g \to \infty} \mathrm{MCG}_g(M), \\
      	\mathcal{I}_{\infty}(M) & \coloneqq \mathrm{colim}_{g \to \infty} \mathcal{I}_g(M), \\
      	\mathrm{Aut}_{\infty}(H_2(M), L_M) & \coloneqq \mathrm{colim}_{g \to \infty} \mathrm{Aut}(H_2(M) \oplus \mathbb{Z}^g, L_M).
      \end{align*} We call $\mathrm{MCG}_{\infty}(M)$ the \emph{stable mapping class group} of $M$. Similarly $\mathcal{I}_{\infty}(M)$ is called the \emph{stable Torelli group} of $M$. The following result is evident.
      
      \begin{Prop}
      	There is a short exact sequence \[0 \rightarrow \mathcal{I}_{\infty}(M) \rightarrow \mathrm{MCG}_{\infty}(M) \rightarrow \mathrm{Aut}_{\infty}(H_2(M), L_M) \rightarrow 1.\] The stable Torelli group $\mathcal{I}_{\infty}(M)$ is an infinitely generated abelian group.
      \end{Prop}
      
      For example, denote $\mathrm{MCG}_{\infty}(S^5)$ and $\mathcal{I}_{\infty}(S^5)$ by $\mathrm{MCG}_{\infty}$ and $\mathcal{I}_{\infty}$ respectively, then the short exact sequence takes the form \[0 \rightarrow \mathcal{I}_{\infty} \rightarrow \mathrm{MCG}_{\infty} \rightarrow \mathrm{GL}_{\infty}(\mathbb{Z}) \rightarrow 1.\] The stable Torelli group $\mathcal{I}_{\infty} = \mathrm{colim}_{g \to \infty} \mathcal{I}(M_g)$ is isomorphic to $\mathbb{Z}^{\infty}$.

    \section{Homotopy equivalences of $M$}\label{sec:8}
    
      In this section $M$ denotes a simply connected closed spin 5-manifold with no 2-torsion in homology, unless otherwise specified. We compare diffeomorphisms and homotopy equivalences of $M$ using the forgetful homomorphism $\psi : \mathcal{I}(M) \rightarrow \mathcal{I}^h(M)$. To formulate the main result of this section we need to recall some previous works.
      
      The homotopy Torelli group has a subgroup $\mathrm{haut}(M | \dot{M})$ consisting of elements which can be represented by homotopy equivalences restricting to the identity map on the complement of the interior of a fixed 5-cell. This group was completely computed by Baues and Buth in \cite[Theorem B]{BB96}. Moreover, in \cite{BB96} Baues\textendash Buth made an extensive investigation of the (unoriented) homotopy automorphism group of $M$. In particular, in \cite[page 581]{BB96} they defined a homomorphism \[t : \mathcal{I}^h(M) \rightarrow H_6(K(H_2(M), 2)),\] which is a homotopic version of the mapping torus construction $\rho$ defined in Section \ref{sec:2}, and proved that (for $M$ satisfying our condition) the kernel of $t$ is $\mathrm{haut}(M | \dot{M})$ \cite[Theorem D]{BB96}. We denote the image of $t$ by $H_M^h$.
      
      We also recall that St\"ocker has proved that there is a bijection $\mathcal{S}(M) \rightarrow H_3(M; \mathbb{Z}/2)$ \cite[Theorem 1.2]{Stö82-mfds}, where $\mathcal{S}(M)$ is the surgery structure set of $M$.
      
      From now on we suppose $H_2(M)$ has the isomorphism type of (\ref{isotype}). For such $M$ the aforementioned results of Baues\textendash Buth and St\"ocker can be stated as follows.
      
      \begin{Thm}[Baues\textendash Buth]\label{thm:BB}
      	There is a short exact sequence \[0 \rightarrow \mathrm{haut}(M | \dot{M}) \rightarrow \mathcal{I}^h(M) \xrightarrow{t} H_M^h \rightarrow 0.\] If $H_2(M)$ has no 3-torsion, then $\mathrm{haut}(M | \dot{M}) \cong {(\mathbb{Z}/2)}^{2g}$. If $H_2(M)$ has nontrivial 3-torsion, then $\mathrm{haut}(M | \dot{M}) \cong {(\mathbb{Z}/2)}^{2g} \oplus {(\mathbb{Z}/3)^{2s_1}}$.
      \end{Thm}
      
      \begin{Thm}[St\"ocker]
      	There is a bijection $\mathcal{S}(M) \cong {(\mathbb{Z}/2)}^g$.
      \end{Thm}
      
      \begin{Rmk}
      	\rm For $M = M_g$ the existence of the bijection $\mathcal{S}(M_g) \cong {(\mathbb{Z}/2)}^g$ can be proved without using St\"ocker's work. By surgery exact sequence \cite{Wal99} and the facts that $L_5(\mathbb{Z}) = 0$ and $\Theta_5 = 0$, we have a bijection \[\mathcal{S}(M) \cong [M, G/\mathrm{O}].\] The homotopy exact sequence of the homotopy fibration $G/\mathrm{O} \rightarrow B\mathrm{O} \rightarrow BG$ takes the form \[\cdots \rightarrow \pi_n(\mathrm{O}) \xrightarrow{J_n} \pi_n^S \rightarrow \pi_n(G/\mathrm{O}) \rightarrow \pi_{n-1}(\mathrm{O}) \rightarrow \cdots\] where $J_n : \pi_n(\mathrm{O}) \rightarrow \pi_n^S$ is the stable $J$-homomorphism. Using Adams' work on $J_n$ \cite{Ada66-J4} and the fact that the Adams conjecture is true \cite{Qui71-Adams}, one could calculate the homotopy groups of $G/\mathrm{O}$, and it follows from the Puppe sequence (and the CW structure of $M_g$) that there is a bijection \[\mathcal{S}(M_g) \cong [\vee_gS^2, G/\mathrm{O}] \cong {(\pi_2(G/\mathrm{O}))}^g \cong {(\mathbb{Z}/2)}^g.\]
      \end{Rmk}
      
      Now we fix some notations. Our computation of $\mathcal{I}(M)$ shows that there is a surjective homomorphism $\mathcal{I}(M) \rightarrow H_M$, given by $[f] \mapsto {\nu_f}_*([M_f])$. By the definition of $t$ there is a commutative diagram \[\begin{tikzcd}
      \mathcal{I}(M) \arrow[r, "\psi"] \arrow[d] & \mathcal{I}^h(M) \arrow[d, "t"] \\
      H_M \arrow[r, "\subset"] & H_M^h.
      \end{tikzcd}\] We denote the kernel of $\mathcal{I}(M) \rightarrow H_M$ by $\overline{K}(M)$. According to \ref{5.3}, $\overline{K}(M)$ is generated by \[\mathbf{t}_{S_1 \# S_1, \alpha} - 8\mathbf{t}_{S_1, \alpha}, \; \cdots, \; \mathbf{t}_{S_{g+2(s_1 + \cdots + s_t)} \# S_{g+2(s_1 + \cdots + s_t)}, \alpha} - 8\mathbf{t}_{S_{g+2(s_1 + \cdots + s_t)}, \alpha}.\] When $H_2(M)$ has no 3-torsion, $\overline{K}(M)$ is isomorphic to $H_2(M)$ and these generators form a basis; when $H_2(M)$ has nontrivial 3-torsion, $\overline{K}(M)$ is isomorphic to $H_2(M)$ modulo a subgroup of its 3-primary subgroup. Let $K(M)$ be the subgroup of $\overline{K}(M)$ generated by \begin{align*}
      & 2(\mathbf{t}_{S_1 \# S_1, \alpha} - 8\mathbf{t}_{S_1, \alpha}), \; \cdots, \; 2(\mathbf{t}_{S_g \# S_g, \alpha} - 8\mathbf{t}_{S_g, \alpha}), \\
      & \mathbf{t}_{S_{g+1} \# S_{g+1}, \alpha} - 8\mathbf{t}_{S_{g+1}, \alpha}, \; \cdots, \; \mathbf{t}_{S_{g+2(s_1 + \cdots + s_t)} \# S_{g+2(s_1 + \cdots + s_t)}, \alpha} - 8\mathbf{t}_{S_{g+2(s_1 + \cdots + s_t)}, \alpha}.
      \end{align*} It has the same isomorphism type as $\overline{K}(M)$.
      
      The main result of this section is the following.
      
      \begin{Thm}\label{thm:diagram}
      	Let $M$ be a simply connected closed spin 5-manifold with no 2- and 3-torsion in homology, then there is a commutative diagram of exact sequences (of based sets at $\mathcal{S}(M)$, of groups at other terms) \[\begin{tikzcd}
      		& & 0 \arrow[d] & 0 \arrow[d] & & \\
      		0 \arrow[r] & K(M) \arrow[r, "\subset"] \arrow[d, "="] & \overline{K}(M) \arrow[r, "\psi"] \arrow[d] & \mathrm{haut}(M | \dot{M}) \arrow[r] \arrow[d] & \mathcal{S}(M) \arrow[d, "="] \arrow[r] & 0 \\
      		0 \arrow[r] & K(M) \arrow[r] & \mathcal{I}(M) \arrow[r, "\psi"] \arrow[d] & \mathcal{I}^h(M) \arrow[r] \arrow[d] & \mathcal{S}(M) \arrow[r] & 0 \\
      		& & H_M \arrow[r, "="] \arrow[d] & H_M^h \arrow[d] & & \\
      		& & 0 & 0 & &
      	\end{tikzcd}\] where the based map $\mathcal{I}^h(M) \rightarrow \mathcal{S}(M)$ maps a homotopy class $[f]$ to the homotopy smoothing $[f : M \rightarrow M]$.
      \end{Thm}
      
      Theorem \ref{thm:3} follows from Theorem \ref{thm:diagram} as a special case. Before presenting the proof of Theorem \ref{thm:diagram} we elaborate this theorem by some remarks (in which $M$ is assumed to satisfy the conditions in Theorem \ref{thm:diagram}).
      
      \begin{Rmk}
      	\rm By Theorem \ref{thm:diagram} $K(M)$ is the group of isotopy classes of diffeomorphisms of $M$ that are homotopic to $id_M$, which is (abstractly) isomorphic to $H_2(M)$ and we have given explicit constructions of its generators. A result of Browder \cite[Theorem 6]{Bro67} implies that such group of a simply connected closed 5-manifold with positive second Betti number contains a free cyclic subgroup. Our computation of $K(M)$ can be viewed as a supplement of Browder's result in dimension 5.
      \end{Rmk}
      
      \begin{Rmk}
      	\rm By our computation of $\mathcal{I}(M)$, the left vertical short exact sequence in the diagram splits, thus so does the right vertical short exact sequence. Thus the homotopy Torelli group $\mathcal{I}^h(M)$ is a semi-direct product of $\mathrm{haut}(M | \dot{M})$ and $H_M$. The action of $H_M$ on the elements of $\mathrm{haut}(M | \dot{M})$ that cannot be represented by diffeomorphisms is unknown, so the group structure of $\mathcal{I}^h(M)$ (for example, whether it is abelian) is unknown. For $M = S^2 \times S^3$, the homotopy Torelli group $\mathcal{I}^h(S^2 \times S^3)$ is not abelian (\cite{Saw75}, see Theorem \ref{thm:product}).
      \end{Rmk}
      
      \begin{Rmk}\label{rmk:homotopyconstruction}
      	\rm According to Baues (see \cite[page 580]{BB96}) there is a surjective homomorphism \[\pi_5(M) \rightarrow \mathrm{haut}(M | \dot{M})\] mapping a homotopy class, represented by a map $\gamma : S^5 \rightarrow M$, to the homotopy class of the homotopy equivalence \[M \rightarrow M \vee S^5 \xrightarrow{id_M \vee {\gamma}} M,\] where $M \rightarrow M \vee S^5$ is the quotient map collapsing the boundary of the 5-cell to a point. In fact $\gamma$ can be chosen to factor through the 4-skeleton of $M$. Combining with Theorem \ref{thm:diagram}, this shows that each of the diffeomorphisms \[\mathbf{t}_{S_1 \# S_1, \alpha} - 8\mathbf{t}_{S_1, \alpha}, \; \cdots, \; \mathbf{t}_{S_g \# S_g, \alpha} - 8\mathbf{t}_{S_g, \alpha}\] is homotopic to such a map. It is unknown that which homotopy classes do these diffeomorphisms correspond to (except for $S^2 \times S^3$, see below).
      \end{Rmk}
      
      Now we prove Theorem \ref{thm:diagram}.
      
      \begin{proof}[Proof of Theorem \ref{thm:diagram}]
      	Clearly the diagram of the vertical short exact sequences commutes. By the definition of surgery structure set, the sequence \[0 \rightarrow \ker(\psi) \rightarrow \mathcal{I}(M) \xrightarrow{\psi} \mathcal{I}^h(M) \rightarrow \mathcal{S}(M)\] is exact, and the induced map $\mathcal{I}^h(M)/{\mathrm{im}(\psi)} \rightarrow \mathcal{S}(M)$ is injective. This induces a commutative diagram of exact sequences (note that $\ker(\psi)$ is contained in $\overline{K}(M)$ since $H_M \rightarrow H_M^h$ is injective) \[\begin{tikzcd}
      		0 \arrow[r] & \ker(\psi) \arrow[r, "\subset"] \arrow[d, "="] & \overline{K}(M) \arrow[r, "\psi"] \arrow[d] & \mathrm{haut}(M | \dot{M}) \arrow[r] \arrow[d] & \mathcal{S}(M) \arrow[d, "="] \\
      		0 \arrow[r] & \ker(\psi) \arrow[r] & \mathcal{I}(M) \arrow[r, "\psi"] & \mathcal{I}^h(M) \arrow[r] & \mathcal{S}(M)
      	\end{tikzcd}\] and the map $\mathrm{haut}(M | \dot{M})/{\mathrm{im}(\psi)} \rightarrow \mathcal{S}(M)$ is also injective. By Baues\textendash Buth's result (Theorem \ref{thm:BB}) $\mathrm{haut}(M | \dot{M}) \cong {(\mathbb{Z}/2)}^{2g}$, so the torsion subgroup of $\overline{K}(M)$ lies in $\ker(\psi)$. Thus the upper horizontal exact sequence induces the following exact sequence \[0 \rightarrow \ker(\psi)/{\text{torsion}} \rightarrow \mathbb{Z}^g \xrightarrow{\psi} {(\mathbb{Z}/2)}^{2g} \rightarrow {(\mathbb{Z}/2)}^g.\] It follows that $\ker(\psi)/{\text{torsion}} \cong \mathbb{Z}^g$, containing the isotopy classes of \[2(\mathbf{t}_{S_1 \# S_1, \alpha} - 8\mathbf{t}_{S_1, \alpha}), \; \cdots, \; 2(\mathbf{t}_{S_g \# S_g, \alpha} - 8\mathbf{t}_{S_g, \alpha}).\] This implies that ${(\mathbb{Z}/2)}^{2g}/{\psi(\mathbb{Z}^g)}$ has at least $2^g$ elements. Since ${(\mathbb{Z}/2)}^{2g}/{\psi(\mathbb{Z}^g)}$ admits an injection to ${(\mathbb{Z}/2)}^g$, it contains precisely $2^g$ elements, and $\mathrm{haut}(M | \dot{M}) \rightarrow \mathcal{S}(M)$ is surjective. Consequently we have $\ker(\psi) = K(M)$, and $H_M$ is equal to $H_M^h$. The proof is completed.
      \end{proof}
      
      \begin{Rmk}
      	\rm If $M$ has nontrivial 3-torsion in homology, we still have the commutative diagram in the proof of Theorem \ref{thm:diagram}. It follows that $\overline{K}(M)$ has nontrivial 3-torsion. For instance, if $H_2(M) \cong {(\mathbb{Z}/3)}^2$, then $\overline{K}(M)$ is isomorphic to ${(\mathbb{Z}/3)}^2$. This implies that the differential $d^4_{7,0}$ of the Atiyah\textendash Hirzebruch spectral sequence of $\Omega^{\mathrm{Spin}}_*(K({(\mathbb{Z}/3)}^2, 2))$ is trivial. Consequently, we have $\Omega^{\mathrm{Spin}}_7(K({(\mathbb{Z}/3)}^2, 2)) \cong H_7(K({(\mathbb{Z}/3)}^2, 2)) \cong {(\mathbb{Z}/3)}^4$.
      \end{Rmk}
      
      We close this section by discussing the case $M = S^2 \times S^3$, where we can use Sawashita's work \cite{Saw75} to obtain more concrete information about the group $\mathrm{haut}(S^2 \times S^3 | \dot{S^2 \times S^3})$. We will need the following result of Sawashita which is extracted from \cite[Proposition 2.4, Theorem 2.6, Lemma 4.1, Theorem 4.3]{Saw75}.
      
      \begin{Thm}[Sawashita]\label{thm:product}
      	Given a space $X$, let $\mathcal{I}^h(X)$ be the group of homotopy classes of homotopy equivalences of $X$ that act trivially on $H_*(X)$.
      	\begin{itemize}
      		\item[(1)] There is an isomorphism \[\pi_3(S^2) \xrightarrow{\cong} \mathcal{I}^h(S^2 \vee S^3)\] mapping a homotopy class $a[\eta]$ to the homotopy class of the homotopy equivalence \[S^2 \vee S^3 \rightarrow S^2 \vee S^3 \vee S^3 \xrightarrow{i_1 \vee a\eta \vee j_1} S^2 \vee S^3,\] where $S^2 \vee S^3 \rightarrow S^2 \vee S^3 \vee S^3$ is the quotient map collapsing the equator of $S^3$ to a point, $i_1 : S^2 \hookrightarrow S^2 \vee S^3$ and $j_1 : S^3 \hookrightarrow S^2 \vee S^3$ are the inclusion maps, and $a\eta$ denotes a map representing the homotopy class $a[\eta]$.
      		
      		\item[(2)] Identify $S^2 \vee S^3$ with the 4-skeleton of $S^2 \times S^3$. Then there is a splitting short exact sequence \[0 \rightarrow \pi_5(S^2 \times S^3) \rightarrow \mathcal{I}^h(S^2 \times S^3) \rightarrow \mathcal{I}^h(S^2 \vee S^3) \rightarrow 0\] where $\mathcal{I}^h(S^2 \times S^3) \rightarrow \mathcal{I}^h(S^2 \vee S^3)$ is induced by restriction on the 4-skeleton, and $\pi_5(S^2 \times S^3) \rightarrow \mathcal{I}^h(S^2 \times S^3)$ is the homomorphism defined in Remark \ref{rmk:homotopyconstruction}. The action of $\mathcal{I}^h(S^2 \vee S^3)$ on $\mathcal{I}^h(S^2 \times S^3)$ is as follows: identify $\mathcal{I}^h(S^2 \vee S^3)$ with $\pi_3(S^2)$ via the isomorphism in (1), let $a[\eta] \in \pi_3(S^2)$ and $([x], [y]) \in \pi_5(S^2) \oplus \pi_5(S^3) = \pi_5(S^2 \times S^3)$ be homotopy classes, then \[a[\eta] \cdot ([x], [y]) = ([x] + [(a\eta) \circ y], [y]).\]
      	\end{itemize}
      \end{Thm}
      
      Combining Theorem \ref{thm:product} with Baues' result stated in Remark \ref{rmk:homotopyconstruction}, we see that $\pi_5(S^2 \times S^3)$ is isomorphic to $\mathrm{haut}(S^2 \times S^3 | \dot{S^2 \times S^3})$. It follows from Theorem \ref{thm:diagram} that $\mathcal{I}^h(S^2 \vee S^3)$ is generated by the restriction of $\mathbf{t}_{S_1, \alpha}$ on $S^2 \vee S^3$ (this can also be seen from the effect of $\mathbf{t}_{S_1, \alpha}$ on $\pi_3(S^2 \times S^3)$), and $\mathrm{haut}(S^2 \times S^3 | \dot{S^2 \times S^3})$ contains $\mathbf{t}_{S_1 \# S_1, \alpha} - 8\mathbf{t}_{S_1, \alpha}$. Since $\mathcal{I}(S^2 \times S^3)$ is abelian, $\mathbf{t}_{S_1, \alpha}$ acts on $\mathbf{t}_{S_1 \# S_1, \alpha} - 8\mathbf{t}_{S_1, \alpha}$ trivially, thus $\mathbf{t}_{S_1 \# S_1, \alpha} - 8\mathbf{t}_{S_1, \alpha}$ corresponds to the generator of $\pi_5(S^2)$ via the isomorphism \[\pi_5(S^2) \oplus \pi_5(S^3) = \pi_5(S^2 \times S^3) \cong \mathrm{haut}(S^2 \times S^3 | \dot{S^2 \times S^3}).\] It follows that the homotopy equivalences defined by the homotopy classes $(0, 1), (1, 1) \in \pi_5(S^2) \oplus \pi_5(S^3) = \mathbb{Z}/2 \oplus \mathbb{Z}/2$ are not homotopic to any diffeomorphism of $S^2 \times S^3$.

    \section{Embeddings of $S^3$ in $S^2 \times S^3$}\label{sec:9}
    
      In this section we prove Theorem \ref{thm:4} and Corollary \ref{cor:fibration}.
      
    \begin{proof}[Proof of Theorem \ref{thm:4}]
      In Section \ref{sec:4} we determined the action of $\mathcal{I}(M_g)$ on $\pi_3(M_g)$. For $g = 1$, combining this with the computation of $\mathrm{MCG}^{\pm}(S^2 \times S^3)$ in \ref{6.3}, we see that a homotopy class $a[S_1 \circ \eta] + b[T_1] \in \pi_3(S^2 \times S^3) = \mathbb{Z}[S_1 \circ \eta] \oplus \mathbb{Z}[T_1]$ (denoted by $(a, b)$ for simplicity) can be represented by $f \circ T_1$ for some diffeomorphism $f \in \mathrm{Diff}(S^2 \times S^3)$ if and only if $\left\vert b \right\vert = 1$. This proves (1).
      
      Now we prove (2). We shall prove a slightly more general result relating embeddings of $S^3$ in $S^2 \times S^3$ to closed 5-manifolds with certain homology. Assume a homotopy class $(a, b)$ can be represented by an embedding of $S^3$, where $b \neq 0$. Since $\pi_2(\mathrm{SO}(2)) = 0$, the normal bundle of this embedding is trivial, so there is an embedding $T' : S^3 \times D^2 \hookrightarrow S^2 \times S^3$ extending it. (This extension is in fact unique up to isotopy since $\pi_3(\mathrm{SO}(2)) = 0$.) We perform a surgery on $T'$ and denote the resulting closed spin 5-manifold by $M'$. Using Seifert\textendash van Kampen theorem, Lefschetz duality and homology exact sequence of pairs, we see that $\pi_1(M') \cong \mathbb{Z}/{\left\vert b \right\vert}$ and the homology groups of $M'$ are given as follows.
      \begin{table}[h]
      	\centering
      	\begin{tabular}{c|cccccc}
      		$i$ & $0$ & $1$ & $2$ & $3$ & $4$ & $5$ \\
      		\hline
      		$H_i(M')$ & $\mathbb{Z}$ & $\mathbb{Z}/{\left\vert b \right\vert}$ & $0$ & $\mathbb{Z}/{\left\vert b \right\vert}$ & $0$ & $\mathbb{Z}$ \\
      	\end{tabular}
      \end{table}
      
      Conversely, let $M'$ be a closed spin 5-manifold with such homology and fundamental group. Choose a generator of $\pi_1(M')$, it can be represented by an embedding $S^1 \times D^4 \hookrightarrow M'$, and we perform a surgery on this embedding. The resulting manifold is a simply connected closed spin 5-manifold with homology isomorphic to $H_*(S^2 \times S^3)$, so by the classification of simply connected 5-manifolds \cite{Sma62, Bar65} it is diffeomorphic to $S^2 \times S^3$. The meridian sphere of this surgery represents an element of $\pi_3(S^2 \times S^3)$, which is $(a, b)$ or $(a, -b)$ for some $a \in \mathbb{Z}$. By the computation of $\mathrm{MCG}^{\pm}(S^2 \times S^3)$, the homotopy classes $(a + kb, b)$ and $(a + kb, -b)$ can be represented by embeddings of $S^3$ for each integer $k$. For $\left\vert b \right\vert = 2$ such a manifold is provided by the manifold $Q_0$ constructed in \cite[page 1168]{GT98}, and for $\left\vert b \right\vert \geq 3$ odd such a manifold is provided by a 5-dimensional lens space $L_{\left\vert b \right\vert}(\ell_1, \ell_2, \ell_3)$ (see for example \cite[Example 2.43]{Hat02}). This proves (2).
    \end{proof}
    
    \begin{Rmk}
    	\rm By a theorem of Whitney (see for example \cite[page 169]{Ran02}), each element of $\pi_3(S^2 \times S^3)$ can be represented by an immersion of $S^3$.
    \end{Rmk}
    
    \begin{proof}[Proof of Corollary \ref{cor:fibration}]
    	For any diffeomorphism $f \in \mathrm{Diff}(S^2 \times S^3)$, the absolute value of the coefficient of the $[T_1]$-factor of $[f \circ T']$ is the same as that of $[T']$. This shows that the map \[\mathrm{MCG}^{\pm}(S^2 \times S^3) = \pi_0\mathrm{Diff}(S^2 \times S^3) \rightarrow \pi_0\mathrm{Emb}(S^3, S^2 \times S^3)\] induced by the restriction map $f \mapsto f \circ T'$ is not surjective.
    \end{proof}

    \appendix
    
    \section{A direct proof of the claim that $\theta(W, \nu)$ is elementary}\label{a:obstruction}
    
    We use the notations in the proof of Lemma \ref{lem:surgery}. Write $N_0 = {\natural}^g(S^2 \times D^3) \times \{0\}$ and $N_1 = {\natural}^g(S^2 \times D^3) \times \{1\}$. By \cite[Corollary 1]{Kre99} we may assume that the normal $B$-structure $\nu : W \rightarrow B$ is 3-connected. Now the obstruction $\theta(W, \nu)$ is represented by the following data \[\Bigl(H_3(W, N_0) \leftarrow \mathrm{im}(d : \pi_4(B, W) \rightarrow \pi_3(W)) \rightarrow H_3(W, N_1), \lambda, \tilde{\mu}\Bigr)\] where $d$ is the connecting homomorphism in the homotopy exact sequence of $\nu : W \rightarrow B$, and $\mathrm{im}(d : \pi_4(B, W) \rightarrow \pi_3(W)) \rightarrow H_3(W, N_i)$ is given by $\pi_3(W) \rightarrow \pi_3(W, N_i) \rightarrow H_3(W, N_i)$. Using this exact sequence and the fact that $\nu : W \rightarrow B$ is 3-connected, we deduce that $\mathrm{im}(d : \pi_4(B, W) \rightarrow \pi_3(W)) = \pi_3(W)$ and there is a commutative diagram \[\begin{tikzcd}
    	\pi_3(W) \arrow[r] \arrow[d] & \pi_3(W, N_i) \arrow[d, "\cong"] \\
    	H_3(W) \arrow[r, "\cong"] & H_3(W, N_i)
    \end{tikzcd}\] where the vertical maps are Hurewicz homomorphisms. Thus \[\theta(W, \nu) = \Bigl(H_3(W) \leftarrow \pi_3(W) \rightarrow H_3(W), \lambda\Bigr)\] where $\lambda$ is the intersection form on $H_3(W)$ and the maps are Hurewicz homomorphisms ($\tilde{\mu} = 0$ since it maps into $\mathbb{Z}/1 = 0$ by the definition of $l_6^{\sim}(e)$). Viewing $W$ as a simply connected compact 6-manifold with boundary $M_g$, we deduce from the homology exact sequence of $(W, M_g)$ that $H_3(W)$ is free. The Hurewicz homomorphism $\pi_3(W) \rightarrow H_3(W)$ is surjective (which can be seen from, for example, the homotopy exact sequence of $(W, N_i)$), thus $H_3(W)$ is a direct summand of $\pi_3(W)$. By Poincar\'e duality, $\lambda$ is a skew-symmetric non-singular bilinear form on $H_3(W)$, thus $H_3(W)$ has a symplectic basis $(e_1, \cdots, e_n, f_1, \cdots, f_n)$, such that the nonzero values of $\lambda$ on them are \[\lambda(e_i, f_i) = 1, \quad \lambda(f_i, e_i) = -1, \quad i = 1, \cdots, n.\] The rank $n$ free abelian group $U$ generated by $e_1, \cdots, e_n$ is a direct summand of $H_3(W)$. View $U$ as a direct summand of $\pi_3(W)$ and let $U_0 = U_1$ be the image of $U$ under the Hurewicz homomorphism $\pi_3(W) \rightarrow H_3(W)$. One checks that $U$ and $U_i$ satisfy the conditions in the definition in \cite[Page 730]{Kre99}, proving that $\theta(W, \nu)$ is elementary.

    \bibliographystyle{amsalpha}
    \bibliography{MCGv1}

\end{document}